\documentclass[11pt]{article}

\usepackage[utf8]{inputenc}

\usepackage{hyperref}

\usepackage{scrpage2}

\usepackage{amsmath}

\usepackage{amsfonts,amssymb,amsthm,starfont}

\usepackage{mathtools}
\usepackage{enumerate}
\usepackage{nicefrac}
\usepackage{graphicx}
\usepackage{esint}
\usepackage{caption}
\usepackage{stmaryrd}
\usepackage[labelformat = brace, position=top]{subcaption}
\usepackage[normalem]{ulem}

\usepackage{tikz}
\usetikzlibrary{calc,intersections,through,backgrounds,patterns,arrows.meta}
\usepackage[top=1.0in, bottom=1.0in, left=1.0in,
right=1.0in]{geometry}

\numberwithin{equation}{section} %for longer documents

\newcounter{statement}
\numberwithin{statement}{section}

\newtheorem{thm}[statement]{Theorem}
\newtheorem{lemma}[statement]{Lemma}
\newtheorem{cor}[statement]{Corollary}
\newtheorem{prop}[statement]{Proposition}

\theoremstyle{remark}
\newtheorem{remark}[statement]{Remark}

\newcommand{\R}{\mathbb R} 				%real numbers
\newcommand{\N}{\mathbb N}

\newcommand{\ball}[2]{{B_{#2}\left(#1\right)}}				%ball of radius #2 around #1
\newcommand{\intd}{\, \mathrm{d}} 				%schönes d für Integrale
\newcommand{\warr}{\rightharpoonup}				%Pfeil für schwache Konvergenz
\newcommand{\Sph}{{\mathbb S}}				%Sphere in 2D
\newcommand{\F}{\mathcal F}								%Fourier Transform
\newcommand{\Hd}{\mathcal H}				%Hausdorff measures

\newcommand{\stcomp}[1]{{#1}^{\mathsf{c}}}		%complements

\newcommand{\eps}{\varepsilon}

\newcommand{\interior}[1]{{\kern0pt#1}^{\mathrm{o}}	%Definition of interior set
}

\newcommand{\id}{\operatorname{id}}

\makeatletter
\newcommand*\bigcdot{\mathpalette\bigcdot@{.7}}
\newcommand*\bigcdot@[2]{\mathbin{\vcenter{\hbox{\scalebox{#2}{$\m@th#1\bullet$}}}}}
\makeatother

\def\Xint#1{\mathchoice
{\XXint\displaystyle\textstyle{#1}}%
{\XXint\textstyle\scriptstyle{#1}}%
{\XXint\scriptstyle\scriptscriptstyle{#1}}%
{\XXint\scriptscriptstyle\scriptscriptstyle{#1}}%
\!\int}
\def\XXint#1#2#3{{\setbox0=\hbox{$#1{#2#3}{\int}$ }
\vcenter{\hbox{$#2#3$ }}\kern-.58\wd0}}

\def\dashint{\Xint-}

\title{A quantitative description of skyrmions in ultrathin ferromagnetic films and rigidity of degree $\pm1$ harmonic maps from $\mathbb{R}^2$ to $\mathbb{S}^2$}

\author{ Anne Bernand-Mantel\footnote{Universit\'e de Toulouse,
    Laboratoire de Physique et Chimie des Nano-Objets, UMR 5215 INSA,
    CNRS, UPS, 135 Avenue de Rangueil, F-31077 Toulouse Cedex 4,
    France} \and Cyrill B.\ Muratov\footnote{Department of
    Mathematical Sciences, New Jersey Institute of Technology, Newark,
    New Jersey 07102, USA.  Please use {muratov@njit.edu} for
    correspondence.
  } \and Theresa M. Simon\footnote{Institut f\"ur Angewandte
    Mathematik, Universit\"at Bonn, Endenicher Allee 60, 53115 Bonn,
    Germany}}

\begin{document}
\maketitle

\begin{abstract}
  We characterize skyrmions in ultrathin ferromagnetic films as local
  minimizers of a reduced micromagnetic energy appropriate for quasi
  two-dimensional materials with perpendicular magnetic anisotropy and
  interfacial Dzyaloshinskii-Moriya interaction. The minimization is
  carried out in a suitable class of two-dimensional magnetization
  configurations that prevents the energy from going to negative
  infinity, while not imposing any restrictions on the spatial scale
  of the configuration.  We first demonstrate existence of minimizers
  for an explicit range of the model parameters when the energy is
  dominated by the exchange energy. We then investigate the conformal
  limit, in which only the exchange energy survives and identify the
  asymptotic profiles of the skyrmions as degree 1 harmonic maps
  from the plane to the sphere, together with their radii, angles and
  energies.  A byproduct of our analysis is a quantitative rigidity
  result for degree $\pm 1$ harmonic maps from the two-dimensional
  sphere to itself.
\end{abstract}

\tableofcontents

\section{Introduction}\label{sec:introduction}

A skyrmion is a topologically nontrivial field configuration that
locally minimizes an energy functional of a nonlinear field theory.
As topological solitons \cite{manton2004}, they are localized, have
finite energy and exhibit quasi-particle properties, including
quantized topological charge, attractive or repulsive interactions
between each other, etc. Since its original formulation by Tony Skyrme
in the early 1960s \cite{skyrme62}, the mathematical concept of
skyrmion has spread over various branches of physics \cite{Rho16}. In
condensed matter physics, a revival of the skyrmion topic was
triggered by experimental observations of skyrmions in
non-centrosymmetric bulk magnetic materials \cite{muhlbauer09,yu10}
and ultrathin ferromagnets \cite{romming13,boulle16} with distinct top
and bottom interfaces \cite{hellman17}.  These magnetic skyrmions
consist of local swirls of spins that may exhibit nanometer size
\cite{romming13}, room temperature thermal stability \cite{boulle16}
and may be controlled via electric current \cite{jonietz11} or
electric field \cite{hsu16}. These properties are highly desirable for
information technology applications, making magnetic skyrmions
attractive for race track memory \cite{tomasello14}, spintronic logic \cite{zhang15sr},
as well as stochastic \cite{pinna18} and
neuromorphic computing \cite{prychynenko18}.

At the level of the continuum, the starting point in the analysis of
magnetic skyrmions in thin ferromagnetic films is the micromagnetic
energy functional \cite{bogdanov1989thermodynamically}
\begin{align}
  \label{eq:E0}
  E(m) := E_\mathrm{ex}(m) + E_\mathrm{a}(m) + E_\mathrm{Z}(m) +
  E_\mathrm{DMI}(m) + E_\mathrm{s}(m)
\end{align}
describing the energy (per unit of the film thickness) of a smooth map
$m : \R^2 \to \mathbb S^2$ that represents the normalized ($|m|= 1$)
magnetization vector field in a ferromagnet.  The terms in
\eqref{eq:E0} are, in order of appearance: the exchange (also called
the Dirichlet energy), the anisotropy, the Zeeman, the
Dzyaloshinskii-Moriya interaction (DMI), and the stray field energies,
respectively. The precise form of these terms is model-specific and
will be spelled out for the particular situation we are interested in
shortly.  Coming back to the magnetization $m$, its topology may be
characterized by the topological charge
\begin{align}
  \label{eq:q0}
  \mathcal N(m) = \frac{1}{4 \pi} \int_{\R^2} m \cdot (\partial_1 m
  \times \partial_2 m) \intd x.
\end{align}
This integer-valued quantity corresponds to the Brouwer degree of a
smooth map $m$ which is constant sufficiently far away from the
origin, up to the sign due to a particular choice of an orientation of
$\R^2$. The topologically nontrivial localized magnetization
configurations are, hence, characterized by a non-zero value of
$\mathcal N$ in \eqref{eq:q0}. See Hoffman \emph{et al.\
}\cite{hoffmann2017antiskyrmions} for a discussion of how to
distinguish between skyrmions and antiskyrmions independently of the
sign convention for $\mathcal{N}$.

In a two-dimensional model containing only the exchange energy
$E_\mathrm{ex}(m) = A_\mathrm{ex} \int_{\R^2} |\nabla m|^2 \intd x$,
where $A_\mathrm{ex}$ is the exchange stiffness, Belavin and Polyakov
\cite{belavin1975metastable} predicted the existence of skyrmion-like
solutions as energy minimizing configurations with constant energy and
an explicit profile, which we refer to later as Belavin-Polyakov
profiles. Note, however, that these solutions may not be considered
proper skyrmions, since they exhibit dilation invariance and thus do
not exhibit true particle-like properties. Furthermore, they are
easily seen to cease to exist in the presence of an additional
anisotropy term
$E_\mathrm{a}(m) = K_\mathrm{u} \int_{\R^2} |m'|^2 \intd x$. Here
$K_\mathrm{u}$ is the uniaxial anisotropy constant and
$m' = (m_1, m_2)$ is the in-plane component of the magnetization
vector $m = (m', m_3)$ \cite{derrick64}. Similarly, skyrmion solutions
are destroyed in the presence of an out-of-plane applied magnetic
field modeled by
$E_\mathrm{Z} (m) = -\mu_0 M_\mathrm{s} \int_{\R^2} H (1 + m_3) \intd
x$, where $H$ is the magnetic field strength, $M_\mathrm{s}$ is the
saturation magnetization, $\mu_0$ is the permeability of vacuum, and
we subtracted a constant to ensure that the Zeeman energy is finite
when $m(x) \to -e_3$ sufficiently fast as $|x| \to \infty$. Therefore,
additional energy terms are necessary to stabilize magnetic skyrmions.

Among the known stabilizing energies are higher order exchange
\cite{ivanov1990magnetic,abanov1998skyrmion}, DMI
\cite{bogdanov1989thermodynamically} and stray field
\cite{kiselev2011chiral,buettner2018theory} terms. In particular,
Bogdanov and Yablonskii \cite{bogdanov1989thermodynamically}
considered an additional DMI term of general form, which includes a
bulk DMI term
$E_\mathrm{DMI}^\mathrm{bulk}(m) = D_\mathrm{bulk} \int_{\R^2} m \cdot
(\nabla \times m) \intd x$, or an interfacial DMI term
$E_\mathrm{DMI}^\mathrm{surf}(m) = D_\mathrm{surf} \int_{\R^2} ( m_3
\nabla \cdot m' - m' \cdot \nabla m_3 ) \intd x$, where
$D_\mathrm{bulk}$ and $D_\mathrm{surf}$ are the bulk and the
interfacial DMI strengths \cite{nagaosa2013topological,
    hellman17}, respectively, and showed that these terms may give
rise to skyrmions.  Their model accounts for the stray field in an
infinite vortex-like magnetization configuration along the thickness
direction \cite{bogdanov1989thermodynamically}.  This prediction was
further verified numerically in the absence \cite{bogdanov89a} and in
the presence \cite{bogdanov1994thermodynamically} of an applied
out-of-plane magnetic field.  Finally, the analysis of B\"uttner et
al.\ indicates that stray field energy alone (starting with an exact
expression for the magnetostatic interaction energy of a
thickness-independent magnetization configuration in a film) may be
sufficient to stabilize magnetic skyrmions
\cite{buettner2018theory}. Notice that in all of the above studies it
is assumed that skyrmion solutions possess radial symmetry.

Mathematically, the question of existence of skyrmions as
topologically nontrivial energy minimizers was first systematically
addressed (under no symmetry assumptions) by Esteban
\cite{Esteban1986direct,esteban2004existence,esteban1990new} and by
Lin and Yang
\cite{lin2004existenceofenergy,lin2004existence}. Specifically, for
the energy of the form of \eqref{eq:E0} consisting of an exchange
energy with an additional Skyrme-type higher order term,
$E_\mathrm{ex}(m) = A_\mathrm{ex} \int_{\R^2} |\nabla m|^2 \intd x +
A_\mathrm{S} \int_{\R^2} |\partial_1 m \times \partial_2 m|^2 \intd
x$, and a special form of an anisotropy/Zeeman term (see the remark in
\cite[p. 2]{doring2017compactness}),
$E_\mathrm{a} (m) + E_\mathrm{Z}(m) = K \int_{\R^2} |m + e_3|^4 \intd
x$, existence of a skyrmion solution as a minimizer $m$ of $E$ with
$\mathcal N(m) = \pm 1$ was proved in
\cite{lin2004existence,Li2011Existence}.  Also curvature of the
underlying space has been explored as another possible mechanism for
ensuring existence of skyrmions
\cite{melcher2019curvature,kravchuk16}.

Turning to DMI-stabilized skyrmions, in situations where the energy
consists of exchange
$E_\mathrm{ex}(m) = A_\mathrm{ex} \int_{\R^2} |\nabla m|^2 \intd x$,
Zeeman
$E_\mathrm{Z}(m) = -\mu_0 M_\mathrm{s} H \int_{\R^2} (1 + m_3) \intd
x$, and bulk DMI
$E_\mathrm{DMI}(m) = D_\mathrm{bulk} \int_{\R^2} m \cdot ( \nabla
\times m) \intd x$ terms, existence of minimizers with non-zero
topological charge for suitable values of the parameters was
established by Melcher \cite{melcher2014chiral}, adapting an argument
by Brezis and Coron \cite{brezis1983large} for harmonic maps on
bounded domains. Furthermore, Melcher demonstrated that the obtained
minimizer is indeed a skyrmion, as the minimum of the energy is
attained for $\mathcal N = 1$ (expressed using the sign conventions of
the present paper).  In the regime of dominating exchange energy
$E_\mathrm{ex}$, D{\"o}ring and Melcher \cite{doring2017compactness}
analyzed the compactness properties of these solutions and proved that
they converge to a minimizer of $E_\mathrm{ex}$ of topological charge
$\mathcal N =1$ found by Belavin and Polyakov
\cite{belavin1975metastable}, which the lower order terms uniquely
determine.  However, as these limits do not decay sufficiently fast
for the Zeeman energy to be finite, they had to choose a faster
decaying version of the Zeeman energy
$E_\mathrm{Z} = K \int_{\R^2} |m+e_3|^p \intd x$ for $p \in (2,4]$,
which only corresponds to a physical model for $p=4$, and even then
only to the specific combination of anisotropy and Zeeman terms
analyzed in \cite{lin2004existence,Li2011Existence}.  Furthermore, Li
and Melcher \cite{li2018stability} proved that, for the above
mentioned physical choices of $E_\mathrm{ex}$, $E_\mathrm{Z}$ and
$E_\mathrm{DMI}$, axisymmetric skyrmions are stable also with respect
to symmetry-breaking perturbations and are indeed local minimizers of
the model considered by Melcher \cite{melcher2014chiral}.  For the
same model, Komineas, Melcher and Venakides \cite{komineas2019profile}
formally established asymptotic formulas for the skyrmion radius and
the energy by means of numerics and asymptotic matching.  Finally,
they also describe the skyrmion profile in a large radius regime on
the basis of formal asymptotic analysis \cite{komineas2019large}.
Existence of skyrmions with a uniaxial anisotropy term $E_\mathrm{a}$
rather than a Zeeman term has been shown by Greco
\cite{greco2019existence} in the context of cholesteric liquid
crystals.

As one expects the minimizers of a perturbed exchange energy to be
close to energy-minimizing harmonic maps (i.e., minimizers of the
Dirichlet energy), it is natural to analyze the rigidity of these
harmonic maps.  The proper context for such an analysis is the theory
of harmonic maps between manifolds, which is reviewed in papers by
Eells and Lemaire \cite{eells1978report,eells1998another}, and by
H{\'e}lein and Wood \cite{helein2008harmonic}.  Here, we only discuss
the immediately relevant results of the theory.  First, the
classification of harmonic maps from $\Sph^2$ to itself in the
mathematical literature is independently due to Lemaire
\cite{lemaire1978applications} and Wood \cite{wood1974harmonic}, see
also \cite[(11.5)]{eells1978report}.  Additionally, they observed that
any harmonic map from $\Sph^2$ to $\Sph^2$ is also energy-minimizing
in its homotopy class. It is worth noting that in the setting of maps
from $\R^2$ to $\Sph^2$ the classification result was also formally
obtained by Belavin and Polyakov \cite{belavin1975metastable}.
Second, a linear version of our stability result, namely that the
null-space of the Hessian only arise from minimality-preserving
perturbations, is also well-known in the case of the identity map
$\operatorname{id} :\Sph^2 \to \Sph^2$ and has been established by
Smith \cite[Example 2.13]{smith1975second}, as well as Mazet
\cite[Proposition 8]{mazet1973formule}.  A similar statement in the
equivalent setting of harmonic maps from $\R^2$ to $\Sph^2$ has
furthermore been recently proved by Chen, Liu and Wei for arbitrary
degrees \cite{chen2019nondegeneracy}.  However, to the best of our
knowledge the corresponding spectral gap estimate has only been
obtained in a related setting by Li and Melcher
\cite{li2018stability}, as well as for the problem of $H$-bubbles by
Isobe \cite{isobe2001asymptotic}, and by Chanillo and Malchiodi
\cite{chanillo2005asymptotic}.  We also point out that, based on
related stability considerations, Davila, del Pino and Wei constructed
solutions to the harmonic heat flow in which degree 1 harmonic maps
bubble off at a specified time and at specified blow-up locations
\cite{davila2019singularity}.  Furthermore, strict local minimality
results closely related to our rigidity result, Theorem
\ref{thm:quantitative_stability} below, have been given by Li and
Melcher \cite{li2018stability}; Di Fratta, Slastikov and Zarnescu
\cite{difratta2019sharp}; and Di Fratta, Robbins, Slastikov and
Zarnescu \cite{difratta2019landau}. Also, in the more restrictive
equivariant setting our rigidity result follows from \cite[Theorem
2.1]{gustafson07} by Gustafson, Kang and Tsai. Finally, very recently
Luckhaus and Zemas \cite{luckhaus2019stability} proved a quantitative
stability result for conformal maps from $\Sph^n$ to itself for all
$n\geq 2$ under a Lipschitz assumption and closeness in $H^1$ to the
identity.

\subsection{Informal discussion of results}\label{sec:intro_results}

In this paper, we analyze the energy $E_{Q,\kappa,\delta}$, to be
defined shortly in equation \eqref{energy_before_simplification}
below, which consists of the exchange, anisotropy, surface DMI energy,
as well as the nonlocal stray field energy that is appropriate for
thin films \cite{knupfer2017magnetic, muratov2017universal,
  muratov2017domain}.  We remark that, while our methods are capable
of dealing with a non-zero external field, we have chosen to consider
the physically most basic case of vanishing external field.  Note that
in this case the energy is unbounded from below, which can be seen by
considering large magnetic bubbles with topological charge
$\mathcal{N}=1$, see for example \cite{bernand2018skyrmion}.
Therefore, an absolute minimizer with the desired topology does not
exist, and instead we have to look for a local minimizer, which means
that we need to identify a suitable constraint.  We argue that in the
present context one possible choice is given by
\cite{bernand2019unraveling}
\begin{align}\label{apriori_constraint}
	\int_{\R^2} |\nabla m|^2 \intd x < 16 \pi.
\end{align}
As the Dirichlet energy in the wall of a magnetic bubble scales with
the radius, this bound clearly excludes such competitors.  In
contrast, to see why the condition \eqref{apriori_constraint} would
yield skyrmion solutions, we turn to the classical Belavin-Polyakov
bound relating the Dirichlet energy to the topological charge:
\begin{align}\label{topological_bound_intro}
	\int_{\R^2} |\nabla m|^2 \intd x \geq 8\pi |\mathcal{N}(m)|,
\end{align}
see the original paper by Belavin and Polyakov
\cite{belavin1975metastable} or Lemma \ref{lem:topological_bound}
below for the proof in the present context.  Together with the bound
\eqref{apriori_constraint}, it a priori excludes higher topological
charges and only allows $\mathcal{N} = -1,0,1$.  At the same time, we
emphasize that, due to the scale invariance of the Dirichlet energy in
two dimensions, the assumption \eqref{apriori_constraint} does not
impose any constraints on the actual size of the skyrmion.  Another
minor point is that the energy cannot distinguish between $m$ and $-m$
and thus only enforces $\lim_{|x|\to \infty} m(x) = e_3$ or
$\lim_{|x|\to \infty} m(x) = - e_3$.  For definiteness, we simply
choose the latter in an averaged sense, which together with the
assumption \eqref{apriori_constraint} defines our admissible class
$\mathcal{A}$ of magnetizations, see the definition in \eqref{defa}.

In Theorem \ref{thm:existence}, we prove that there exists an explicit
constant $C>0$ such that in the regime
$0<\frac{|\kappa| + \delta}{\sqrt{Q-1}} \leq C$ the energy
$E_{Q,\kappa,\delta}$, defined in equation
\eqref{energy_before_simplification} below, does indeed admit
minimizers over $\mathcal{A}$.  In comparison to Melcher's work
\cite{melcher2014chiral}, the main issues are, first, an a priori lack
of control of the decay of out-of-plane component $m_3+1$ due to the
absence of an external field, and second, the presence of the nonlocal
terms.  To restore control of $m_3+1$, we combine the
Gagliardo-Nirenberg-Sobolev inequality with a vectorial version of the
Modica-Mortola argument.  To handle the nonlocal terms, we mainly
appeal to interpolation inequalities.  The remaining argument closely
follows the methods developed by Brezis and Coron
\cite{brezis1983large} and Melcher \cite{melcher2014chiral}, ruling
out the vanishing and splitting alternatives of Lions'
concentration-compactness principle \cite{lions1984concentration}.
Vanishing, which heuristically is the collapse of skyrmions via
shrinking, is ruled out by combining the topological bound
\eqref{topological_bound_intro} with a construction giving
\begin{align}
	\inf_{\mathcal{A}} E_{Q,\kappa,\delta} < 8\pi,
\end{align}
so that the scale-dependent contributions to the energy cannot go to
zero.  In the case of splitting, i.e., two configurations drifting
infinitely far apart from each other, the combinatorics involved in
the requirement $\mathcal{N}=1$ and the two bounds in
\eqref{topological_bound_intro} and \eqref{apriori_constraint} imply
that at least one of the two pieces has $\mathcal{N}=1$.  As the
nonlocal interaction of two magnetic charges vanishes as they move
infinitely far apart, the two pieces essentially do not interact so
that the energy can be strictly lowered by discarding the piece with
$\mathcal{N} \neq 1$.  Thus splitting is excluded and the obtained
compactness is sufficiently strong to prove existence of minimizers.

The most important part of this paper is the description of the
asymptotic behavior of the obtained minimizers in Theorem
\ref{thm:convergence} for
$0<\frac{|\kappa| + \delta}{\sqrt{Q-1}} \ll 1$, corresponding to the
regime dominated by the Dirichlet energy: The minimizers of
$E_{Q,\kappa,\delta}$ approach the set of minimizers of the Dirichlet
energy $\int_{\R^2} |\nabla m|^2 \intd x$, i.e., the set of
Belavin-Polyakov profiles. Here, the challenge is to capture the fact
that the skyrmion radius converges to zero in the limit of dominating
exchange energy in order to compensate all Belavin-Polyakov profiles
having infinite anisotropy energy.  This intuition can be gained by
making an ansatz-based minimization of suitably truncated
Belavin-Polyakov profiles, which provides us with an upper bound for
the minimal energy in the form of a finite-dimensional reduced energy
depending only on the scale of truncation, the skyrmion radius and
rotation angle \cite{bernand2019unraveling}.  Therefore, in order to
find a matching lower bound and to conclude the proof, one has to
quantitatively control closeness of the minimizers to the set of
Belavin-Polyakov profiles.  To this end, we prove a rigidity result
for Belavin-Polyakov profiles, Theorem
\ref{thm:quantitative_stability}, estimating the Dirichlet
distance of $\mathring H^1$ maps of degree 1 to the set of
Belavin-Polyakov profiles (see the definition in
  \eqref{def_Dirichlet_distance}) in terms of the Dirichlet excess
$\int_{\R^2} |\nabla m|^2 \intd x - 8\pi$.  Said excess can be
directly linked to the scale of truncation, and the stability result
allows us to prove that the lower order contributions to
$E_{Q,\kappa,\delta}$ match the upper bound.  A subtle issue here is
the fact that the Belavin-Polyakov profile obtained in Theorem
\ref{thm:quantitative_stability} does not necessarily approach $-e_3$
at infinity or even have a limit which is close to $-e_3$.  This is
related to the logarithmic failure of the critical Sobolev embedding
$H^1 \not \hookrightarrow L^\infty$.  Instead, we have to ensure the
correct behavior at infinity by proving that otherwise the anistropy
energy is too large.  The coercivity properties of the reduced energy
finally allow to conclude the proof.

The proof of the rigidity result, Theorem
\ref{thm:quantitative_stability}, relies on first proving a
corresponding linear estimate in the form of a spectral gap estimate
for the Hessian at a Belavin-Polyakov profile.  To this end, we
diagonalize the Hessian using a vector-valued version of spherical
harmonics.  The main difficulty is then to pass to the nonlinear
estimate, specifically in the case where the Dirichlet excess is
small.  Existence of a Belavin-Polyakov profile that is close to the
minimizer follows from known compactness properties of minimizing
sequences in the harmonic map problem, and to remove the trivial
degeneracies of the Hessian resulting form the invariances of the
energy, we pick the closest Belavin-Polyakov profile in the
$\mathring H^1$-topology.  In order to then apply the spectral gap
estimate, we have to justify that the Hessian gives a good description
of the energy close to the minimizer.  However, the higher order terms
turn out to be radially weighted $L^p$-norms for which standard
attempts at estimation fail logarithmically.  Therefore, we have to
find some problem-specific cancellations, for which we exploit the
fact that the harmonic map problem is conformally invariant and that
the Belavin-Polyakov profiles are conformal maps.  This allows us to
formulate the rigidity problem for maps from $\Sph^2$ to $\Sph^2$,
where the error terms turn into unweighted $L^p$-norms which are
amenable to the Sobolev inequality.  The required cancellation is then
the fact that the average of the identity map over $\Sph^2$ vanishes.
Finally, we obtain a Moser-Trudinger type inequality for maps in
which the vanishing average assumption is replaced by closeness to the
identity in the $\mathring H^1$-topology.

\subsection{Outline of the paper}
In Section \ref{sec:main_results} we state and discuss our main
results in detail.  Section \ref{sec:explicit} is devoted to providing
an explicit representation of the energy $E_{\sigma,\lambda}$ by
continuously extending the nonlocal terms $F_\mathrm{vol}$ and
$F_\mathrm{surf}$.  In Section \ref{sec:rigidity} we give the proof of
Theorem \ref{thm:quantitative_stability}.  The upper bound for the
minimal energy and Theorem \ref{thm:existence}, the existence of
skyrmions, can be found in Section \ref{sec:existence}.  The proof of
Theorem \ref{thm:convergence} is completed in Section
\ref{section:conformal}.  Finally, Appendix \ref{sec:appendix}
collects an introduction to Sobolev spaces on the sphere, the proof of
the topological lower bound along with a classification of its degree
1 extremizers, and a number of calculations involving Bessel functions
necessary for calculating the energy of the ansatz.  Within each
subsection, we always first present all propositions, lemmas and
corollaries, while their proofs can be found at the end of the
subsection.  Remarks concerning notation can be found at the end of
Section \ref{sec:main_results}.

\section{Main results}\label{sec:main_results}
\subsection{The energy and the admissible class}

In this paper, we consider the following model
\cite{muratov2017domain}, based on a rigorous asymptotic expansion of
the stray field energy given in \cite{knupfer2017magnetic}: For the
quality factor $Q>1$, non-dimensionalized film-thickness $\delta>0$
and DMI-strength $\kappa$ and on
\begin{align}
  \mathcal{D}:= \{ m \in C^\infty(\R^2;\Sph^2): m+e_3 \text{ has
  compact support}\} 
\end{align}
we choose, recalling that $m = (m',m_3)$,
\begin{align}
  E_\mathrm{ex}(m) & := \int_{\R^2} |\nabla m|^2 \intd x,\\
  E_\mathrm{a}(m) & := Q \int_{\R^2} |m'|^2 \intd x,\\
  E_\mathrm{Z}(m) & := 0,\\
  E_\mathrm{DMI}(m) & := \kappa \int_{\R^2} \left(  m_3
                      \nabla \cdot m' - m'\cdot \nabla
                      m_3 \right) \intd x,\\ 
  E_{s}(m) & := -  \int_{\R^2} |m'|^2 \intd x + \delta
             \left(F_\mathrm{vol}(m') - F_\mathrm{surf}(m_3)
             \right),\label{nonlocal_approximation} 
\end{align}
where the normalized, nonlocal contributions $F_\mathrm{vol}(m')$
  and $F_\mathrm{surf}(m_3)$ of the volume and surface charges,
  respectively, are defined via
\begin{align}
  F_{\mathrm{vol}}(f) & := \frac{1}{4\pi}\int_{\R^2}
                        \int_{\R^2}\frac{\nabla \cdot f(x)\nabla
                        \cdot f(\tilde x)  }{|x-\tilde x|} \intd
                        \tilde x \intd
                        x,\label{nonlocal_real_space_1}\\ 
  F_{\mathrm{surf}}(\tilde f) & := \frac{1}{8\pi} \int_{\R^2}
                                \int_{\R^2} \frac{(\tilde f(x)-\tilde
                                f(\tilde x))^2}{|x-\tilde x|^3} \intd
                                \tilde x \intd
                                x,\label{nonlocal_real_space_2} 
\end{align}
for $f \in C_c^\infty(\R^2;\R^2) $ and $\tilde f \in C^\infty(\R^2)$
such that there exists $c \in \R$ with $\tilde f+ c$ having compact
support. They can be interpreted as multiples of the squares of the
$\mathring H^{-\frac{1}{2}}$-norm of $\nabla \cdot m'$ and the
$\mathring H^\frac{1}{2}$-norm of $m_3$, respectively, and an
extension of these terms of sufficient generality for our purposes can
be found in Section \ref{sec:explicit}.  In total, our functional
may then be expressed as
\begin{align}\label{energy_before_simplification}
  \begin{split}
    E_{Q,\kappa,\delta}(m) & := \int_{\R^2} \Big( |\nabla m|^2 + (Q-1)
    | m' |^2 -2\kappa m' \cdot \nabla m_3 \Big) \intd x + \delta
    \left( F_{\mathrm{vol}}(m') - F_{\mathrm{surf}}(m_3) \right),
  \end{split}
\end{align}
where we integrated by parts to simplify the DMI term.

In order to remove one of the parameters and make the mathematical
structure of the energy explicit, we further rescale our functional
\eqref{energy_before_simplification}.  We first point out that the
sign of $\kappa$ is not essential: If we have $\kappa < 0$, then
considering $\widetilde m( x) := m ( -x)$ gives
\begin{align}
	E_{Q,\kappa,\delta} (m) = E_{Q,-\kappa,\delta}\left(\widetilde m\right).
\end{align}
and thus we may additionally suppose $\kappa \geq 0$.
Furthermore, provided $\kappa + \delta>0$ we use the rescaling 
\begin{align}
  \bar x := \frac{Q-1}{\kappa + \delta} x \ \text{ and } \ \bar m(\bar
  x) := m\left(\frac{\kappa  + \delta}{Q-1}\bar x \right) 
\end{align}
in the energy \eqref{energy_before_simplification}, so that for
\begin{align}
  \sigma:= \frac{ \kappa + \delta}{\sqrt{Q-1}} \ \text{
  and } \ \lambda := \frac{\kappa }{ \kappa  + \delta}
\end{align}
we finally obtain
$E_{Q,\kappa,\delta}(m) = E_{\sigma,\lambda}(\bar m)$, where
\begin{multline}\label{energy}
    E_{\sigma,\lambda} (\bar m) := \int_{\R^2} |\nabla \bar m|^2
    \intd \bar x \\
    + \sigma^2 \Bigg( \int_{\R^2} |\bar m' |^2\intd \bar x
    - 2 \lambda \int_{\R^2} \bar m' \cdot \nabla \bar m_3
    \intd \bar x + (1-\lambda) \left( F_{\mathrm{vol}}(\bar m') -
      F_{\mathrm{surf}}(\bar m_3) \right) \Bigg)
  \end{multline}
for $\bar m  \in \mathcal{D}$.

Of course, the assumption on regularity and decay at infinity encoded
in $\mathcal{D}$ is much too restrictive to allow for existence of
minimizers.  In view of the discussion in Section
\eqref{sec:intro_results}, it would be natural to consider instead the
energy on the $\Sph^2$-valued variant of the homogeneous Sobolev
  space $\mathring H^1(\R^2)$, which we define as a space of functions
\begin{align}
  \mathring H^1(\R^2) := \left\{u \in H^1_{\mathrm{loc}} (\R^2) :
  \int_{\R^2} |\nabla u|^2 \intd x < \infty \right\},
\end{align}
equipped with the $L^2$-norm of the gradient (note, however, some
technical issues associated with such a critical space \cite[Section
1.3]{bahouri2011fourier}). Consequently, we aim to consider the energy
$E_{\sigma,\lambda}$ on the set
\begin{align}\label{defa}
  \mathcal{A}:= \left\{m \in \mathring H^1(\R^2;\Sph^2):
  \int_{\R^2} |\nabla m|^2 \intd x < 16\pi,\, m + e_3 \in L^2(\R^2), \
  \mathcal{N}(m) = 1 \right\}, 
\end{align}
where the condition $ m + e_3 \in L^2(\R^2)$ is the appropriate way of
prescribing $\lim_{|x| \to \infty} m(x) = -e_3$, see Lemma
\ref{lemma:sobolev}.  Note that the definition of the topological
charge, also referred to as the degree,
\begin{align}\label{def_top_charge}
  \mathcal{N}(m) := \frac{1}{4\pi} \int_{\R^2} m \cdot \big(
  \partial_1 m \times \partial_2 m \big) \intd x, 
\end{align}
is valid for all $m \in \mathring H^1(\R^2; \Sph^2)$ and is consistent
with equation \eqref{eq:q0} for smooth maps that are constant
sufficiently far from the origin.  However, due to the nonlocal terms,
some care needs to be taken in extending the energy to $\mathcal{A}$.
To avoid technicalities before the statement of results, we extend by
relaxation, i.e., for $m \in \mathring H^1(\R^2;\Sph^2)$ with
$m+e_3 \in L^2(\R^2;\R^3)$ we set
\begin{align}\label{def_energy_relaxation}
  E_{\sigma,\lambda} (m) := \inf\{ \liminf_{n\to \infty}
  E_{\sigma,\lambda}(m_n) : m_n \in \mathcal{D} \text{ for } n\in \N
  \text{ with } \lim_{n\to \infty} \|m_n - m\|_{H^1} =0 \}. 
\end{align}
Corollary \ref{cor:representation} states that the representation
\eqref{energy} is still valid, provided the nonlocal terms
$F_\mathrm{vol}$ and $F_\mathrm{surf}$ are interpreted
appropriately.

\subsection{Statement of the results}\label{sec:statement_main_results}
We first establish that the energy $E_{\sigma,\lambda}$ admits
minimizers over $\mathcal{A}$ for all $\lambda \in [0,1]$, provided
$\sigma$ is sufficiently small.
In particular, we get existence of skyrmions even in the case
$\lambda = 0$, which corresponds to no DMI being present.  Our model
therefore predicts skyrmions purely stabilized by the stray field.
The proof of Theorem \ref{thm:existence} below closely follows
the previous works by Melcher \cite{melcher2014chiral} and D{\"o}ring
and Melcher \cite{doring2017compactness} and relies on the
concentration compactness principle of Lions
\cite{lions1984concentration}.  The main new aspect is the inclusion
of the nonlocal terms due to the stray field, which we deal with by
standard interpolation inequalities.

\begin{thm}\label{thm:existence}
  Let $\sigma>0$ and $\lambda \in [0,1]$ be such that
  $\sigma^2(1+\lambda)^2 \leq 2$.  Then there exists
  $m_{\sigma,\lambda} \in \mathcal{A}$ such that
	\begin{align}
          E_{\sigma,\lambda}(m_{\sigma,\lambda}) = \inf_{\widetilde
          m\in \mathcal{A}} E_{\sigma,\lambda}(\widetilde m). 
	\end{align}
\end{thm}

\noindent Note that throughout the rest of the paper we suppress
$\lambda$ in the index of $m_{\sigma,\lambda}$ for simplicity of
notation.

We now turn to the heart of the paper, namely, the analysis of the
limit $\sigma \to 0$ in which the Dirichlet energy dominates.  As
was already pointed out by D{\"o}ring and Melcher
\cite{doring2017compactness}, in this limit one expects minimizers
$m_\sigma$ of $E_{\sigma,\lambda}$ to converge to minimizers of
\begin{align}
	F(m) :=\int_{\R^2} |\nabla m|^2 \intd x
\end{align}
for $m\in \mathring H^{1}(\R^2;\Sph^2)$ with $\mathcal{N}(m)=1$, i.e.,
minimizing harmonic maps of degree 1.  These have been identified by
Belavin and Polyakov \cite{belavin1975metastable}, see also Brezis and
Coron \cite[Lemma A.1]{brezis1985convergence} or Lemma
\ref{lem:topological_bound} below, to be given by the previously
mentioned Belavin-Polyakov profiles
\begin{align}\label{moduli_space}
  \mathcal{B}:=\left\{S \Phi(\rho^{-1}(\bigcdot - x)): S \in
  \operatorname{SO}(3),\,  
  \rho>0, \, x\in \R^2 \right\}, 
\end{align}
where $\Phi$ is a rotated variant of the stereographic projection with
respect to the south pole
\begin{align}\label{belavin-polyakov}
	\Phi(x) := \left(- \frac{2x}{1+|x|^2},  \frac{1- |x|^2}{1+ |x|^2} \right)
\end{align}
for $x\in \R^2$.  One can moreover see that they achieve equality in
the topological bound \eqref{topological_bound_intro} in view of
\begin{align}
	\int_{\R^2} |\nabla \phi|^2\intd x = 8\pi
\end{align}
for all $\phi \in \mathcal{B}$.  It is even known, see
\cite[(11.5)]{eells1978report}, that $\mathcal{B}$ comprises all
solutions $\phi : \R^2 \to \Sph^2$ of the harmonic map equation
\begin{align}\label{harmonic_map_equation}
  \Delta \phi + |\nabla \phi|^2 \phi =0 
\end{align}
with $\mathcal{N}(\phi)=1$, meaning all critical points of $F$ of
degree 1 are absolute minimizers.

The task then is to identify which Belavin-Polyakov profiles
$\phi= S\Phi(\rho^{-1}(\bigcdot - x))$ for
$S \in \operatorname{SO}(3)$ and $\rho>0$ are selected in the limit
$\sigma \to 0$.  By the requirement $m + e_3 \in L^2(\R^2; \R^3)$, we
can certainly expect to have $S e_3 = e_3$ in the limit, so that
$S=S_\theta$ for some angle $\theta \in [-\pi,\pi)$ and
\begin{align}\label{def_rotation}
	S_\theta:=
	\begin{pmatrix}
		\cos \theta & -\sin \theta & 0\\
		\sin \theta & \cos \theta & 0 \\
		0 & 0 &1 
	\end{pmatrix}.
\end{align}
However, even for such Belavin-Polyakov profiles it holds that
$\phi +e_3 \not \in L^2(\R^2;\R^3)$ due to logarithmic divergence of
the anisotropy term.  Consequently, we expect minimizers to be
truncated Belavin-Polyakov profiles which will shrink to keep the
anisotropy energy finite in the limit $\sigma \to 0$ in the spirit of
the construction by D{\"o}ring and Melcher \cite[Lemma
3]{doring2017compactness}.

Indeed, careful minimization in a corresponding class of ans{\"a}tze
\cite{bernand2019unraveling}, see also section \ref{sec:upper_bound},
leads one to believe that the optimal skyrmion radius $\rho_0$ is
given asymptotically by
\begin{align}
  \rho_0  \simeq   \frac{\bar g(\lambda)}{16\pi} \frac{1}{|\log
  \sigma |}, 
\end{align}
where the auxiliary function
\begin{align}\label{epsilon_definition}
  \bar g(\lambda) & :=\begin{cases}
    (8 + \frac{\pi^2}{4})\pi\, \lambda - \frac{\pi^3}{4}   & \text{ if
    } \lambda \geq \lambda_{c},\\ 
    \frac{128\lambda^2}{3\pi(1-\lambda)} + \frac{\pi^3}{8}(1-\lambda)&
    \text{ else},
 		\end{cases}
 \end{align}
 in which the critical threshold $\lambda_c$ is defined as
 \begin{align}
   \label{lamc}
   \lambda_{c} & := \frac{3\pi^2}{32+3\pi^2}, 
 \end{align}
 results from the balance of the DMI and stray field terms.  The
 function $\bar g(\lambda)$ can straightforwardly be seen to be
 continuous and satisfy
 \begin{align}
   \label{gcC}
   \frac{1}{C}\leq \bar g(\lambda) \leq C
 \end{align}
 for a universal constant $C>0$.  Furthermore, the two optimal
 rotation angles $\theta_0^+ \in [0, \frac{\pi}{2}]$ and
 $\theta_0^- \in [-\frac{\pi}{2},0]$ are asymptotically
\begin{align}\label{optimal_angle}
	\theta_0^\pm & := 
		\begin{cases}
                  0 & \text{ if } \lambda \geq \lambda_{c},\\
                  \pm \arccos\left(\frac{32\lambda}{3\pi^2(1-\lambda)}
                  \right) & \text{ else.}
		\end{cases}
\end{align}
Here, the angle $\theta_0^\pm = 0$ corresponds to a N{\'e}el-type
skyrmion profile present in the regime $\lambda \geq \lambda_c$ of DMI
dominating over the stray field, while skyrmions purely stabilized by
the stray field have Bloch-type profiles in view of
$\theta_0^{\pm} = \pm \frac{\pi}{2}$ for $\lambda = 0$. The
following convergence theorem confirms these expectations.

\begin{thm}\label{thm:convergence}
  Let $\lambda \in [0,1]$.  Let $m_{\sigma}$ be a minimizer of
  $E_{\sigma,\lambda}$ over $\mathcal{A}$.  Then there exist
  $x_\sigma\in \R^2$, $\rho_\sigma >0$ and
  $\theta_\sigma \in [-\pi,\pi)$ such that
  $m_\sigma - S_{\theta_\sigma}
  \Phi(\rho_\sigma^{-1}(\bigcdot-x_\sigma)) \to 0$ in
  $\mathring H^1(\R^2;\R^3)$ as $\sigma \to 0$, and
	\begin{align}
          \lim_{\sigma\to 0} |\log\sigma|\rho_\sigma
          = \frac{\bar g(\lambda)}{16\pi},\qquad \qquad
          \lim_{\sigma\to 0} |\theta_\sigma| = \theta^+_0,
	\end{align}
	as well as
	\begin{align}
          \lim_{\sigma\to 0} \frac{|\log \sigma|^2}{\sigma^2 \log
          |\log \sigma |} \left| E_{\sigma,\lambda}(m_\sigma) - 8\pi +
          \frac {\sigma^2}{|\log\sigma|}\left(\frac{\bar g^2 
          (\lambda)}{32\pi} - \frac{ \bar g^2 
          (\lambda)}{32\pi}\frac{\log|\log\sigma|}{|\log \sigma|}
          \right)  \right|    = 0. 
	\end{align}
\end{thm}

{\begin{remark} For the convergences in Theorem \ref{thm:convergence},
    our methods also allow to provide the following non-optimal
    (with the exception of the estimate
      \eqref{dirichletexcessoptimal}) rates:
		\begin{align}
                  \int_{\R^2} \left|\nabla \left( m_{\sigma}(x) -
                  S_{\theta_\sigma}
                  \Phi(\rho_\sigma^{-1}(x-x_\sigma))\right)
                  \right|^2 \intd x & \leq C  \sigma^2,\\ 
                  \left||\log\sigma|\rho_\sigma - \frac{\bar
                  g(\lambda)}{16\pi}\right|  & \leq  \frac{C}{|\log
                                               \sigma|},\\ 
                  \left| |\theta_\sigma| -\theta_0^+\right|^4 +
                  \left|\lambda-\lambda_{c} \right| \left|
                  |\theta_\sigma| -\theta_0^+ \right|^2 & \leq
                                                          \frac{C}{|\log\sigma|}, 
	\end{align}
	as well as
	\begin{align}\label{dirichletexcessoptimal}
          \frac{1}{C}\frac{\sigma^2}{|\log\sigma|^2}\leq \int_{\R^2}
          |\nabla m_\sigma|^2 \intd x - 8\pi \leq C
          \frac{\sigma^2}{|\log\sigma|^2},
	\end{align}
	and
	\begin{align}
          \left| \frac{|\log\sigma|}{\sigma^2}\left(
          E_{\sigma,\lambda}(m_\sigma) - 8\pi\right) - \left(-
          \frac{\bar g^2 (\lambda)}{32\pi} + \frac{ \bar g^2
          (\lambda)}{32\pi}\frac{\log|\log\sigma|}{|\log \sigma|}
          \right)  \right|   \leq \frac{C}{|\log\sigma|}, 
	\end{align}
	for $C>0$ universal and $\sigma \in (0,\sigma_0)$ with
        $\sigma_0>0$ small enough and universal.  In fact, our proof
        does establish all these rates except the one for the angles
        $\theta_\sigma$, whose proof relies on some lengthy, but
        elementary estimates.  Note that the loss in the rate of
        convergence of $\theta_\sigma$ to $\theta^\pm_0$ for the
        parameter $\lambda = \lambda_c$ coincides with N{\'e}el
        profiles becoming linearly unstable.
      \end{remark}

      While not strictly speaking adhering to a $\Gamma$-convergence
      framework, the proof of Theorem \ref{thm:convergence} is very
      much in the spirit of $\Gamma$-equivalence
      \cite{braides2008asymptotic} in that we compare the sequence of
      energies at minimizers to a sequence of finite-dimensional
      reduced energies.  This simplification allows us to explicitly
      compute approximate minimizers and even analyze their stability
      properties.  As is usual in the theory of $\Gamma$-convergence,
      the comparison is done via upper bounds obtained by construction
      and ansatz-free lower bounds.  The constructions have already
      been alluded to above.  The main ingredient for the lower bounds
      is the following Theorem \ref{thm:quantitative_stability}, a
      quantitative stability estimate for degree 1 harmonic maps from
      $\R^2$ to $\Sph^2$, i.e., for the maps in the set $\mathcal{B}$,
      see definition \eqref{moduli_space}.  Once we know that the
      minimizers are close to Belavin-Polyakov profiles, we use this
      information to estimate the remaining lower order terms in the
      energy.

      To state the theorem, we first introduce the family of all
      $\mathring{H^1}$-maps from $\R^2$ to $\Sph^2$ of degree 1 :
\begin{align}
  \mathcal{C}:=\left\{\tilde m \in \mathring H^1(\R^2;\Sph^2):
  \mathcal{N}(\tilde m) = 1\right\}. 
\end{align}
We next introduce a notion of distance between elements in this
  family and Belavin-Polyakov profiles, which we term the
  \emph{Dirichlet distance}:
\begin{align}\label{def_Dirichlet_distance}
  D(m;\mathcal{B}) := \inf_{\widetilde \phi \in \mathcal{B}}
  \left(\int_{\R^2}\left| \nabla \left( m - \widetilde \phi \right)
  \right|^2 \intd x\right)^\frac{1}{2}. 
\end{align}
With these definitions we have the following theorem.

\begin{thm}\label{thm:quantitative_stability}
  For every $m \in \mathcal{C}$ there exists $\phi \in \mathcal{B}$
  that achieves the infimum in the Dirichlet distance
    $D(m;\mathcal{B})$.
        Furthermore, there exists a universal constant $\eta>0$ such
        that
	\begin{align}
          \eta D^2(m;\mathcal{B}) \leq  F( m) - 8\pi . 
	\end{align}
      \end{thm}
      \noindent Notice that this result in the more
      restrictive equivariant setting is contained in \cite[Theorem
      2.1]{gustafson07}. 

Well understood compactness properties of minimizing sequences for the
Dirichlet energy \cite{lin1999mapping} ensure the existence of a
Belavin-Polyakov profile $\phi$ that is close to an almost minimizer
$m$ of the Dirichlet energy but do not provide us with a rate of
closeness. To overcome this issue, we pass to the corresponding
linearized problem, which can easily be solved using a suitable
vectorial version of spherical harmonics, see Proposition
\ref{prop:spectral_gap} below.  However, naive attempts at explicitly
estimating the error terms arising in the linearization procedure tend
to break down due to the logarithmic failure of the critical Sobolev
embedding $H^1 \not \hookrightarrow L^\infty$ in two dimensions.
Therefore, the main conceptual issue is to find additional
cancellations resulting form the structure of the problem.

The relevant structure, it turns out, is the fact that the harmonic
map problem is conformally invariant, and that all Belavin-Polyakov
profiles are conformal maps.  This allows us to reformulate the
problem as stability of the identity map
$\operatorname{id} : \Sph^2 \to \Sph^2$, denoted from now on as
  $\operatorname{id}_{\Sph^2}$, by considering
$\widetilde m := m \circ \phi^{-1}$.  Nonlinear terms can then be
estimated using the standard Sobolev embedding on the sphere, and the
required cancellation is that the identity map on the sphere has
average zero.  This idea leads us to the following estimates, which
when expressed on $\R^2$ also provides topologies in which $m$ itself
converges to $\phi$.

\begin{lemma}\label{lem:linearization_estimate}
  There exists a universal constant $\tilde \eta>0$ such that the
  following holds: Let $p \in [1,\infty)$. Then there exists a
  constant $C_p>0$ such that if $m \in H^1(\Sph^2;\Sph^2)$ satisfies
  $\int_{\Sph^2} |\nabla (m-\operatorname{id}_{\Sph^2})|^2 \intd \Hd^2
  \leq \tilde \eta$, then we have the estimate
	\begin{align}\label{linearization_estimate}
          \left(\int_{\Sph^2} |m - \operatorname{id}_{\Sph^2} |^p
          \intd \Hd^2 \right)^\frac{1}{p} \leq C_p \left(\int_{\Sph^2}
          |\nabla (m-\operatorname{\id}_{\Sph^2} )|^2 \intd \Hd^2
          \right)^\frac{1}{2}. 
	\end{align}
	Furthermore, there exists a universal constant $C>0$ such that
        the Moser-Trudinger type inequality
	\begin{align}\label{mosertrudinger}
          \int_{\Sph^2} e^{\frac{2\pi}{3}
          \frac{|m-\operatorname{id}_{\Sph^2} 
          |^2}{\| \nabla (m-\operatorname{id}_{\Sph^2}) \|_{2}^2}}  \intd
          \mathcal{H}^2 \leq C  
 	\end{align}
 	holds.
\end{lemma}

We furthermore point out that Theorem \ref{thm:quantitative_stability}
implies a corresponding statement for degree one harmonic maps on
$\Sph^2$, i.e., for minimizers of
\begin{align}
  F_{\Sph^2}\left (\widetilde m \right) := \int_{\Sph^2} \left| \nabla
  \widetilde m \right|^2 \intd \Hd^2 
\end{align}
over 
\begin{align}
  \label{eq:Csph2}
\mathcal{C}_{\Sph^2} := \left\{ \bar m \in H^1(\Sph^2;\Sph^2):
  \mathcal{N}_{\Sph^2}(\bar m) =1 \right\},   
\end{align}
where
\begin{align}\label{def_degree_on_sphere}
  \mathcal{N}_{\Sph^2}(\widetilde m) := \frac{1}{ 4 \pi}
  \int_{\Sph^2} \det (\nabla 
  \widetilde m) \intd \Hd^2 
\end{align}
denotes the degree for maps from $\widetilde m : \Sph^2 \to \Sph^2$,
see for example Brezis and Nirenberg \cite{brezis1995degree} or
Section \ref{sec:diff_geo} in the appendix for details.  Recalling the
definition \eqref{belavin-polyakov} of $\Phi$, it can be seen that the
minimizers are given by the set of M\"obius transformations
\begin{align}\label{def_b_on_sphere} 
  \mathcal{B}_{\Sph^2}
  & := \left\{ \phi \circ \Phi^{-1} : \phi \in
    \mathcal{B} \right\}  \\ 
  & = \left\{\Phi \circ f \circ \Phi^{-1}: f(z):= \frac{a z + b}{c z +
    d} \text{ for } a,b,c,d \in \mathbb{C} \text{ with } ad - bc \neq
    0  \right\},\label{moebius} 
\end{align}
where points $x \in \R^2$ in the plane are identified with the
  points $z \in \mathbb C$ in the complex plane.  Indeed, this
follows from the conformal invariance of the harmonic map problem, see
Lemma \ref{lem:conformal_invariance}.  The second equality is a
classical fact we will prove in Lemma \ref{lem:topological_bound} for
the convenience of the reader.  A similar nonlinear stability
statement for degree $-1$ maps is a simple result of the identity
$\mathcal{N}_{\Sph^2}(-\widetilde m) =
-\mathcal{N}_{\Sph^2}(\widetilde m)$ for
$\widetilde m \in H^1(\Sph^2;\Sph^2)$.

\begin{cor}\label{cor:stability_conformal}
  For $\widetilde m\in \mathcal{C}_{\Sph^2}$ we have
	\begin{align}
          \eta \min_{\widetilde \phi \in \mathcal{B}_{\Sph^2}}
          \int_{\R^2} \left| \nabla \left( \widetilde m - \widetilde
          \phi\right) \right|^2 \intd x \leq F_{\Sph^2}\left(
          \widetilde m \right)- 8\pi, 
	\end{align}
	where $\eta>0$ is the universal constant of Theorem
        \ref{thm:quantitative_stability}.  Furthermore, for
        $\widetilde m \in H^1(\Sph^2; \Sph^2)$ with
        $\mathcal{N}_{\Sph^2} (\widetilde m) = -1$ we have the
        corresponding statement
	\begin{align}
          \eta \min_{\widetilde \phi \in
          \left(-\mathcal{B}_{\Sph^2}\right)} \int_{\R^2} \left|
          \nabla \left( \widetilde m - \widetilde \phi\right)
          \right|^2 \intd x \leq F_{\Sph^2}\left( \widetilde m
          \right)- 8\pi. 
	\end{align}
\end{cor}
\noindent Notice that our result is stronger than the one in
\cite{luckhaus2019stability} for $n = 2$ in that it does not require
the assumption that the map $\widetilde m$ be Lipschitz and close in
$H^1(\Sph^2; \R^3)$ to $\mathcal{B}$.

\subsection{Notation}

Throughout the paper, the symbols $C$ and $\eta$ denote universal,
positive constants that may change from inequality to inequality, and
where we think of $C$ as large and $\eta$ as small.  Whenever we use
$O$-notation, the involved constants are understood to be
universal. For matrices $A\in \R^{n\times m}$ for $n,m\in \N$, we use
the Frobenius norm $|A|:= \sqrt{\operatorname{tr}(A^T A)}$.

\section{An explicit representation of the
  energy} \label{sec:explicit}

Here, we extend the functionals $F_{\mathrm{vol}}$ and
$F_{\mathrm{surf}}$ to a sufficiently big space of functions to ensure
that our energy $E_{\sigma,\lambda}$ has a practical representation on
$\mathcal{A}$, and that $F_{\mathrm{vol}}$ and
  $F_{\mathrm{surf}}$ are defined for the relevant components of the
stereographic projection $\Phi$.  The main tool to obtain the relevant
estimates will be the Fourier transform, for which we use the
convention
\begin{align}
  \label{Fourdef}
	\F f(k) := \int_{\R^2} e^{-\mathrm{i} k\cdot x} f(x) \intd x
\end{align}
for $f\in L^1(\R^2)$ and which we extend to functions
$f\in L^p(\R^2)$ for $1 < p \leq 2$ in the usual way, see
\cite[Section 5.4 and 5.6]{lieb-loss}.

The situation for $F_{\mathrm{surf}}$ is straightforward: As it is
obviously non-negative, we can simply use the definition
\eqref{nonlocal_real_space_2} for all
$f \in L^1_\mathrm{loc}(\R^2)$ with
  $F_\mathrm{surf}(f) < \infty$.  The Fourier space representation of
$F_{\mathrm{surf}}(f)$ obtained from
\eqref{surf_fourierrepresentation} for $f \in H^1(\R^2)$ below
will nevertheless be helpful to prove estimates and to compute the
surface charge contribution of a Belavin-Polyakov profile.
Furthermore, we define
\begin{align}
  F_{\mathrm{surf}}( f, g) & := \frac{1}{8\pi} \int_{\R^2} \int_{\R^2}
                             \frac{( f(x)- f(\tilde x))( g(x) -
                             g(\tilde x))}{|x-\tilde x|^3} \intd
                             \tilde x \intd
                             x\label{nonlocal_real_space_2_extended} 
\end{align}
whenever $f,g : \R^2 \to \R$ are measurable with
$F_{\mathrm{surf}}(f)<\infty$ and $F_{\mathrm{surf}}(g) <\infty$.

Turning to the volume charges, for measurable functions
$\tilde f,\tilde g: \R^2 \to \R^2$ with
$\nabla \cdot \tilde f \in L^2(\R^2)$ and
$\nabla \cdot \tilde g \in L^2(\R^2)$ we define
\begin{align}
  F_{\mathrm{vol}}\left( \tilde f\right)
  & := \frac{1}{2} \int_{\R^2}
    \frac{\left| \F\left(
    \nabla \cdot  \tilde
    f\right)\right|^2}{|k|}
    \frac{\intd
    k}{(2\pi)^2},\label{vol_extension_single}\\ 
  F_{\mathrm{vol}}\left( \tilde f, \tilde g\right)
  & := \frac{1}{2}
    \int_{\R^2}
    \frac{ \F\left(
    \nabla \cdot
    \tilde f\right)
    \overline{\F
    \left( \nabla
    \cdot  \tilde g
    \right)} }{|k|}
    \frac{\intd
    k}{(2\pi)^2},\label{vol_extension_mixed} 
	\end{align}
        the latter of which requires
        $F_{\mathrm{vol}}( \tilde f ) <\infty$ and
        $F_{\mathrm{vol}}\left(\tilde g\right)<\infty$.
	
        The following, standard lemma ensures this is indeed an
        extension of the original definition
        \eqref{nonlocal_real_space_1} and provides a number of
        interpolation inequalities for both $F_{\mathrm{surf}}$ and
        $F_{\mathrm{vol}}$ we will use throughout the paper.
	
\begin{lemma}\label{lemma:fourier_basic}
  For maps $f, g: \R^2 \to \R$ such that there exist $c, d \in \R$
  with $f +c, g+d \in H^1(\R^2)$ and $\tilde f : \R^2 \to \R^2$ such
  that $\nabla \cdot \tilde f \in L^2(\R^2)$ we have
  $F_{\mathrm{surf}}(f,g) \in \R$ and
	\begin{align}
          F_{\mathrm{surf}}( f) & \geq 0,\label{surf_nonnegative}\\
          F_{\mathrm{vol}}(\tilde f ) & \geq 0,\label{vol_nonnegative}\\
          | F_{\mathrm{surf}}(f,g) | & \leq \frac{1}{2} \| f+c \|_2
                                       \|\nabla g
                                       \|_{2}.\label{surf_interpolation} 
	\end{align}
	If, for $p \in (1,\infty)$, we additionally have
        $\tilde f, \tilde g \in L^p(\R^2;\R^2) \cap \mathring
        W^{1,p'}(\R^2;\R^2)$ with
        $\nabla \cdot \tilde g \in L^2(\R^2)$, then we have
        $F_\mathrm{vol}(\tilde f, \tilde g) \in \R $ with the estimate
	\begin{align}\label{vol_interpolation}
          F_{\mathrm{vol}}\left(\tilde f,\tilde g \right)
          & \leq C_p
            \left\|
            \tilde f
            \right
            \|_p
            \left\|\nabla
            \tilde g
            \right\|_{p'}. 
	\end{align}
	Finally, we have the representation
	\begin{align}\label{surf_fourierrepresentation}
          F_{\mathrm{surf}}(f,g) = \frac{1}{2} \int_{\R^2} |k|
          \mathcal{F} (f+c)  \overline{ \F (g+d)}  \frac{\intd
          k}{(2\pi)^2}, 
	\end{align}
	and for $\tilde f,\tilde g \in C_c^\infty(\R^2;\R^2)$
        we also have
	\begin{align}\label{vol_equivalence}
          \frac{1}{4\pi}\int_{\R^2} \int_{\R^2}\frac{\nabla \cdot
          \tilde f(x)\nabla \cdot \tilde g(\tilde x)  }{|x-\tilde x|}
          \intd \tilde x \intd x = \frac{1}{2} \int_{\R^2} \frac{
          \F\left( \nabla \cdot \tilde  f\right) \overline{\F \left(
          \nabla \cdot \tilde  g \right)} }{|k|} \frac{\intd
          k}{(2\pi)^2}.
	\end{align}
	In particular, the definition \eqref{vol_extension_single}
          extends that in \eqref{nonlocal_real_space_1}.
\end{lemma}

With these extensions, we prove that the representation \eqref{energy}
of $E_{\sigma,\lambda}$ is still valid.  Notice that the density
  result below is a variant of \cite[Lemma 4.1]{melcher2012global}
  (see also Schoen and Uhlenbeck \cite{schoen1983boundary}).

\begin{cor}\label{cor:representation}
  For $\sigma >0$, $\lambda >0$ and $m \in \mathring H^1(\R^2;\Sph^2)$
  with $m +e_3 \in L^2(\R^2;\R^3)$ there exists a sequence
  $m_n \in \mathcal{D}$ with
  $\lim_{n\to \infty} \| m_n - m\|_{H^1} =0$, and we have
	\begin{align}\label{energy_in_lemma}
  \begin{split}
    E_{\sigma,\lambda} ( m) = \int_{\R^2} |\nabla m|^2 \intd x + 
      \sigma^2 \biggl( & \int_{\R^2} | m' |^2\intd x -2 \lambda
      \int_{\R^2} m' \cdot \nabla m_3 \intd x  \\
    &  + (1-\lambda) \left( F_{\mathrm{vol}}( m') -
        F_{\mathrm{surf}}( m_3) \right) \biggr).
  \end{split}
\end{align}
\end{cor}

\begin{proof}[Proof of Lemma \ref{lemma:fourier_basic}]
  We first deal with the surface term. The estimate
  \eqref{surf_nonnegative} is trivial. The Fourier representation
  \eqref{surf_fourierrepresentation} follows immediately from
  \cite[Theorem 7.12, identity (4)]{lieb-loss}. The estimate
  \eqref{surf_interpolation} is then a straightforward consequence of
  the Cauchy-Schwarz inequality and Plancherel's theorem,
  \cite[Theorem 5.3]{lieb-loss}.
	
  Next, we turn to the volume terms.  Again, non-negativity
  \eqref{vol_nonnegative} is a trivial consequence of the definition
  \eqref{vol_extension_single}.  For
  $\tilde f, \tilde g \in C_c^\infty(\R^2;\R^2)$, the equality
  \eqref{vol_equivalence} is a result of \cite[Theorem 5.2, identity
  (2)]{lieb-loss}.
	
  By a density argument, it is sufficient to prove the interpolation
  result \eqref{vol_interpolation} still under the assumption
  $\tilde f, \tilde g \in C_c^\infty(\R^2;\R^2)$.  To this end, we
  define a vectorial variant of the Riesz transform
	\begin{align}
          T \tilde f :=  \mathcal{F}{^{-1}}\left( i \frac{k}{|k|}
          \cdot \mathcal{F} \tilde f \right). 
	\end{align}
	By the standard fact that
        $\F (\nabla \cdot \tilde f)(k) = i k \cdot \F \tilde f(k)$ for
        a.e.\ $k\in \R^2$ and Plancherel's identity,
        we have
	\begin{align}
          \frac{1}{2} \int_{\R^2} \frac{ \F\left( \nabla \cdot
          f\right) \overline{\F \left( \nabla \cdot  g \right)} }{|k|}
          \frac{\intd k}{(2\pi)^2} = \frac{1}{2} \int_{\R^2}
          T\left(\tilde f \right) \nabla \cdot \tilde g \intd x. 
	\end{align}
	By the Mihlin-H{\"o}rmander multiplier theorem \cite[Theorem
        6.2.7]{grafakos}, $T$ extends to a bounded operator from
        $L^p(\R^2;\R^2)$ to $L^p(\R^2)$ for all $p \in (1,\infty)$.
        As a result, H{\"o}lder's inequality implies the desired
        inequality \eqref{vol_interpolation}.
\end{proof}

\begin{proof}[Proof of Corollary \ref{cor:representation}]
  By the density result \cite[Lemma 4.1]{melcher2012global} we may
  choose a sequence $m_n \in C^\infty(\R^2;\Sph^2)$ with
  $m_n +e_3 \in L^2$ such that
  $\lim_{n \to \infty} \| m_n -m \|_{H^1} =0$.  The proof of
  \cite[Lemma 8]{doring2017compactness} implies that we may
  furthermore take $m_n + e_3$ to have compact support for all
  $n \in \N$, so that we have $m_n \in \mathcal{A}$.  The local terms
  are obviously continuous in the $H^1$-topology.  Continuity of
  $F_\mathrm{vol}$ and $F_\mathrm{surf}$ is ensured by Lemma
  \ref{lemma:fourier_basic}.
\end{proof}

\section{Rigidity of degree \texorpdfstring{$\pm1$}{±1} harmonic
  maps} \label{sec:rigidity}

The goal of this section is to prove Theorem
\ref{thm:quantitative_stability}, the quantitative stability statement
for Belavin-Polyakov profiles with respect to the Dirichlet energy
$F(m)$.  As explained in Section \ref{sec:statement_main_results}, it
will be helpful at times to think of maps
$\widetilde m :\Sph^2 \to \Sph^2$ by setting
$\widetilde m := m \circ \phi^{-1}$ for some appropriately chosen
$\phi \in \mathcal{B}$.  The maps $\phi \in \mathcal{B}$ have the nice
property of being \emph{conformal}, see \cite[Chapter 4, Definition
3]{docarmo}.  As such, the above re-parametrization leaves the
harmonic map problem invariant, see Lemma
\ref{lem:conformal_invariance}, and we gain compactness of the
underlying sets, as well as a greater conceptual clarity in some of
our arguments.

The definitions of gradients, Laplace operators and Sobolev spaces on
the sphere can be found in Section \ref{sec:diff_geo}. In particular,
we use the same symbol for the Euclidean and Riemannian versions of
gradients and Laplace operators as it is always clear from context
which one is meant.

\subsection{The spectral gap property for the linearized problem}
\label{sec:spectral_gap}

This subsection is devoted to the solution of the linear problem
corresponding to Theorem \ref{thm:quantitative_stability} and
Corollary \ref{cor:stability_conformal}, i.e., we establish the sharp
spectral gap property for the Hessian of $F$, or equivalently
$F_{\Sph^2}$, at minimizers. The notions and arguments needed are
fairly standard. Here we provide the proof for the
convenience of the reader.

Given a map $m\in \mathcal{C}$ that is close to $\phi \in \mathcal{B}$
in $\mathring{H}^1(\R^2;\R^3)$, by Lemma
\ref{lem:conformal_invariance} we also have that $m\circ \phi^{-1}$ is
close to $\operatorname{id}_{\Sph^2}$ in $H^1(\Sph^2;\R^3)$.
Therefore, we only have to compute the Hessian at the identity map
$\operatorname{id}_{\Sph^2} :\Sph^2 \to \Sph^2$.  The corresponding
Hessian on $\Sph^2$ and in local coordinates given by
$\phi \in \mathcal{B}$ is, respectively
\begin{align}
  \mathfrak{H}(\zeta,\xi)
  & := \int_{\Sph^2} \left(\nabla \zeta : \nabla \xi - 2 \zeta \cdot
    \xi \right)\intd \Hd^2,\label{representation_sphere_hessian}\\ 
  \mathfrak{H_\phi}(\zeta_\phi, \xi_\phi)
  & := \int_{\R^2}
    \left(\nabla \zeta_\phi :
    \nabla \xi_\phi -
    \zeta_\phi \cdot \xi_\phi
    \, |\nabla \phi |^2
    \right)\intd x, 
\end{align}
see \cite{mazet1973formule,smith1975second}, defined for tangent
vector fields $\zeta, \xi \in H^{1}(\Sph^2; T\Sph^2)$, see equation
\eqref{def:sobolev_tangent_field} for the definition of this space,
and
\begin{align}
  \zeta_\phi,\xi_\phi \in  H^{1}_{\mathrm{w}}(\R^2;T_{\phi}\Sph^2) :=
  \left\{ \tilde \xi_\phi \in H^{1}_{\mathrm{w}}(\R^2;\R^3) : \, \tilde
  \xi_\phi(x) \cdot \phi(x) = 0 \text{ for almost all } x \in \R^2
  \right\}. 
\end{align}
Here, we introduced a vector-valued variant
  $H^{1}_{\mathrm{w}}(\R^2;\R^3) $ of the weighted Sobolev space
\begin{align}\label{def_weighted_sobolev}
  H^{1}_{\mathrm{w}}(\R^2) := \left\{ u  \in
  H^{1}_{\mathrm{loc}}(\R^2) : \int_{\R^2} \left(|\nabla 
  u|^2 + \frac{|u|^2}{1+|x|^4}\right) \intd x < \infty 
  \right\},
\end{align}
arising from $H^1(\Sph^2;\R^3)$ under parametrization by
$\phi \in \mathcal B$ due to Lemma \ref{lem:conformal_invariance}, and
the pullback
$T_\phi \Sph^2:= \bigcup_{x\in \R^2} \{\phi(x)\} \times
T_{\phi(x)}\Sph^2 $ of the tangent bundle $T\Sph^2$ of the sphere.  In
particular, note that $\zeta_\phi \cdot \xi_\phi |\nabla \phi|^2$ is
integrable.  As $\operatorname{id}_{\Sph^2}$ is a minimizer of
$F_{\Sph^2}$, the Hessians are non-negative bilinear forms in the
sense that for all $\xi \in H^{1}(\Sph^2; T\Sph^2)$ and
$\xi_\phi \in H^{1}_{\mathrm{w}}(\R^2;T_{\phi}\Sph^2)$ we have
\begin{align}
  \mathfrak{H}(\xi,\xi) & \geq 0,\label{poincare_tangent_on_sphere}\\
  \mathfrak{H}_\phi(\xi_\phi,\xi_\phi) & \geq
                                         0.\label{poincare_for_tangent_fields} 
\end{align}
In view of identity \eqref{representation_sphere_hessian}, the
inequality \eqref{poincare_tangent_on_sphere} can be interpreted as a
Poincar{\'e} type inequality on the space $H^{1}(\Sph^2; T\Sph^2)$
that does not rely on subtracting averages.

The next step is to identify the null space of the Hessian:
\begin{align}
  J:= \left\{\zeta \in H^1\left( \Sph^2; T \Sph^2 \right) :
  \mathfrak{H}(\zeta,\zeta) =0 \right\}. 
\end{align}
It is well known that $\zeta \in J$ is equivalent to $\zeta$ solving
the so-called \emph{Jacobi equation}
\begin{align}\label{jacobi-equation_sphere}
  L(\zeta)(y):= -\Delta  \zeta (y)  - 2 \zeta (y) - 2 (\nabla y :
  \nabla  \zeta(y))y = 0 
\end{align}
for all $y\in \Sph^2$, where the Laplace-Beltrami operator is taken
component-wise.  We call solutions to the Jacobi equation \emph{Jacobi
  fields}.  In local coordinates given by $\phi \in \mathcal{B}$,
i.e., for $\zeta_\phi := \zeta \circ \phi^{-1}$, this equation is
\begin{align}\label{jacobi-equation}
  L_{\phi}( \zeta_\phi):= - \Delta  \zeta_\phi - |\nabla \phi |^2
  \zeta_\phi - 2( \nabla \phi : \nabla \zeta_\phi)  \,  \phi = 0. 
\end{align}
For our purposes we only need to rigorously ensure that $\zeta \in J$
solves a weak version of the equation in local coordinates under the
(a posteriori unnecessary) assumption that $\zeta \in J$ is smooth,
which we will do in Lemma \ref{lem:jacobi} below for the convenience
of the reader.

\begin{lemma}\label{lem:jacobi}
  Let $\zeta \in J$ be smooth and $\phi \in \mathcal{B}$.  Then for
  all $\xi \in \mathring{H}^{1}(\R^2;\R^3) \cap L^\infty(\R^2;\R^3)$
  the function $\zeta_\phi := \zeta \circ \phi$ satisfies
	\begin{align}\label{jacobi_weak}
          \int_{\R^2} \left( \nabla \zeta_\phi : \nabla \xi -
          \zeta_\phi \cdot \xi |\nabla \phi |^2 - 2 \phi \cdot \xi \,
          \nabla \phi : \nabla \zeta_\phi \right) \intd x =0. 
	\end{align}
\end{lemma}

Using the characterization of the Hessian $\mathfrak{H}$ as the second
derivative of $F_{\Sph^2}$ at $\operatorname{id}_{\Sph^2}$, we can
readily find Jacobi fields: If for $\eps>0$ and $t\in (-\eps,\eps)$,
the function $u_t$ is a smooth curve of minimizers of $F_{\Sph^2}$
with $u_0 = \operatorname{id}_{\Sph^2}$, then
$\frac{\intd}{\intd t}|_{t=0} u_t \in J$.  To use this idea, we recall
the representation
\begin{align}\label{variations_through_harmonic_maps}
	u_t = \Phi \circ f_t \circ \Phi^{-1}
\end{align}
for $f_t(z) = \frac{a_t z + b_t}{c_t z +1}$ and $a_t$,
$b_t$, $c_t \in \mathbb{C}$ with $a_t -b_tc_t \neq 0$ and $a_0 = 1$,
$b_0 =c_0=0$, see equation \eqref{moebius}.  Differentiating in $t$,
we see that $ j \in J$, where by the chain rule we have
\begin{align}
  j(y) := \left(\nabla \Phi \circ \Phi^{-1}(y) \right) g \circ
  \Phi^{-1}(y) \qquad y \in \Sph^2,
\end{align}
for the complex polynomial
$g(z) := -\frac{\intd}{\intd t}|_{t=0}c_t z^2 +
\frac{\intd}{\intd t}|_{t=0} a_t z + \frac{\intd}{\intd t}|_{t=0}
b_t$.
In particular, we know that $\operatorname{dim} J \geq 6$.

In the following Proposition \ref{prop:spectral_gap}, we prove that
all Jacobi fields arise in such a manner and we compute the spectral
gap.  To this end, we use to the notion of vector spherical harmonics
\cite[Chapter 5.2]{freeden2009spherical}, as they turn out to
diagonalize $\mathfrak{H}$.  They are related to the spherical
harmonics $Y_{n,j} : \Sph^2 \to \R$ for $n\geq 0$ and
$j = -n, \ldots n$, which are eigenfunctions of the Laplace-Beltrami
operator $\Delta$ with eigenvalues $-n(n+1)$.  Here, we take them to
be normalized such that they form a real-valued, orthonormal system
for $L^2(\Sph^2)$.  Their definition is well-known and we do not need
their explicit expressions in the following (an interested reader may
refer to \cite[Chapter 3.4]{freeden2009spherical}).  The vector
spherical harmonics, see \cite[equation (5.36)]{freeden2009spherical}
are defined for $y\in \Sph^2$ as
\begin{align}
  \mathcal{Y}^{(1)}_{0,0} (y) & := \frac{1}{\sqrt{4\pi}} y,
\end{align}
and for $n \geq 1$ and $j = -n,\ldots ,n$ as
\begin{align}
  \mathcal{Y}_{n,j}^{(1)}(y) & := Y_{n,j}(y) \, y,\\
  \mathcal{Y}_{n,j}^{(2)}(y) & := \frac{1}{\sqrt{n(n+1)}} \, \nabla
                               Y_{n,j}(y),\\ 
  \mathcal{Y}_{n,j}^{(3)}(y) & := \frac{1}{\sqrt{n(n+1)}} \, y \times
                               \nabla Y_{n,j}(y). 
\end{align}
Similarly to their scalar counterparts, they are eigenfunctions with
eigenvalues $-n(n+1)$ for a suitably defined vectorial
Laplace-Beltrami operator \cite[Theorem 5.28 and Definition
5.26]{freeden2009spherical}: For $\xi \in C^2(\Sph^2;\R^3)$, using the
projections $\pi_{\mathrm{t}}$ and $\pi_{\mathrm{n}}$ onto tangential
and normal components, we set
\begin{align}
  \Delta_v \xi :=  \pi_{\mathrm{n}}(\Delta +2) (\pi_{\mathrm{n}} \xi)
  + \pi_{\mathrm{t}} \Delta (\pi_{\mathrm{t}} \xi), 
\end{align}
where $\Delta$ is to be understood as the component-wise
Laplace-Beltrami operator.  Furthermore, they form an orthonormal
system for $L^2(\Sph^2;\R^3)$, see \cite[Theorem
5.9]{freeden2009spherical}.

Turning to tangential vector fields, we note that by the above
  results the set
$\{\mathcal{Y}_{n,j}^{(2)},\mathcal{Y}_{n,j}^{(3)}: n\geq 1, j= -n,
\ldots,n\}$ of tangential vector spherical harmonics forms an
orthonormal system for
\begin{align}\label{h1_orthogonality}
  L^2(\Sph^2;T\Sph^2) := \left\{\xi \in L^2(\R^2;\R^3): \xi(y) \cdot y
  =0 \text{ for almost all } y \in  \Sph^2 \right\}. 
\end{align}
Additionally, we can use the fact that the vector spherical harmonics
are eigenfunctions of $\Delta_v$ to integrate by parts, see equation
\eqref{manifold_tangentialbeltrami} below, to obtain
\begin{align}\label{eigenfunctionswrth1}
  \int_{\Sph^2} \nabla \mathcal{Y}_{n,j}^{(k)} : \nabla
  \mathcal{Y}_{p,i}^{(o)} \intd \Hd^2 = n(n+1) \delta_{n,p}
  \delta_{j,i} \delta_{k,o} 
\end{align}
for $n, p\geq 1$, $j = - n, \ldots n$, $i = - p,\ldots p$
and $k, o \in \{2,3\}$, where the $\delta$ symbols denote the
corresponding Kronecker deltas.

With this information, we are finally able to characterize both the
space of Jacobi functions $J$ and the spectral gap of the Hessian
$\mathfrak{H}$ with respect to the $\mathring H^1$-scalar product.  To
this end, we define the space of tangent vector fields which are
$\mathring H^1$-orthogonal to the space $J$ of Jacobi fields
\begin{align}
  \mathbb{H}^1:=\left\{\xi \in H^1(\Sph^2;T\Sph^2): \int_{\Sph^2}
  \left( \nabla \xi : \nabla \zeta \right) \intd \Hd^2 = 0 \text{ for
  all } \zeta \in J \right\}. 
\end{align}
The choice of the $\mathring H^1$-scalar product is motivated by
Theorem \ref{thm:quantitative_stability} requiring us to estimate the
$\mathring H^1$-distance of any given $m\in \mathcal{C}$ to
$\mathcal{B}$.

\begin{prop}\label{prop:spectral_gap}
  We have
  $J= \operatorname{span}\left\{\mathcal{Y}_{1,j}^{(2)},
    \mathcal{Y}_{1,j}^{(3)}; j= -1,0,1 \right\}$.  In particular, all
  Jacobi fields are smooth and it holds that
  $\operatorname{dim} J =6$.  Furthermore, we have the spectral gap
  property
  \begin{align}\label{spectral_gap_inequality}
    \mathfrak{H}(\xi,\xi) \geq \frac{2}{3} \int_{\Sph^2} |\nabla \xi
    |^2 \intd \mathcal{H}^2 
	\end{align}
	for all $\xi \in \mathbb{H}^1$. Finally, the $L^2$-orthogonal
        projection
        $\pi_J : L^2(\Sph^2; T\Sph^2) \to L^2(\Sph^2;T\Sph^2)$ onto
        $J$ is well-defined and orthogonal with respect to the
          inner product in $\mathring H^1(\Sph^2)$.
\end{prop}
\noindent Thus all Jacobi fields arise from variations of the form
  \eqref{variations_through_harmonic_maps}.

Having presented all statements of this subsection, we provide
their proofs below.

\begin{proof}[Proof of Lemma \ref{lem:jacobi}]
  \textit{Step 1: We have $\zeta \in J$ if and only if the
    condition
  \begin{align}\label{linear_condition}
    \int_{\Sph^2} \left(\nabla  \zeta : \nabla  \xi - 2 \zeta \cdot
    \xi \right) \intd \Hd^2 =0 
	\end{align}
	holds for all $\xi \in H^1(\Sph^2;T\Sph^2)$.}\\
  Let $ \zeta \in J$, meaning we have
  $\zeta \in H^{1}(\Sph^2;T\Sph^2)$ with
  $\mathfrak{H}( \zeta, \zeta) = 0$.  As $\mathfrak{H}$ is a
  non-negative bilinear form, the Cauchy-Schwarz inequality implies
  for all $ \xi \in H^1(\Sph^2;T\Sph^2)$ that
	\begin{align}
          0 \leq | \mathfrak{H} ( \zeta, \xi) | \leq
          \mathfrak{H}^\frac{1}{2} ( \zeta, \zeta)
          \mathfrak{H}^\frac{1}{2} ( \xi,  \xi) =0,
	\end{align}
        which yields \eqref{linear_condition}.  Furthermore, by
        choosing $ \xi = \zeta$ we get that
        \eqref{linear_condition} is equivalent to
        $\mathfrak{H}\left( \zeta, \zeta\right) = 0$.
	
        \textit{Step 2: Prove equation \eqref{jacobi_weak}.}\\
        If
        $\xi \in \mathring{H^{1}}(\R^2;\R^3)\cap L^\infty(\R^2;\R^3)
        \subset H^1_\mathrm{w}(\R^2; \R^3)$ satisfies
        $\xi(x) \cdot \phi(x) =0$ for almost all $x\in \R^2$, then the
        statement immediately follows from \eqref{linear_condition}
        and Lemma \ref{lem:conformal_invariance}.  Consequently, it is
        sufficient to consider
        $\xi \in \mathring{H}^{1}(\R^2;\R^3)\cap L^\infty(\R^2;\R^3)$
        with $\xi = (\xi\cdot \phi) \phi$.  In this case, by virtue of
        $\zeta_{\phi} \cdot \xi =(\xi\cdot \phi) (\zeta_{\phi} \cdot
        \phi) =0$ almost everywhere we can write
	\begin{align}\label{check_Jacobi_equation}
	  \begin{split}
            \int_{\R^2} & \Big(\nabla \zeta_\phi : \nabla \xi -
            \zeta_\phi \cdot \xi |\nabla \phi |^2 - 2(\nabla \phi
            :\nabla \zeta_\phi)(\phi \cdot\xi) \Big) \intd x\\
            & = \int_{\R^2} \Big( \nabla \zeta_\phi : \nabla [
              (\phi \cdot \xi) \phi] - 2(\nabla \phi
              :\nabla \zeta_\phi)(\phi \cdot\xi) \Big) \intd x \\
            & = \left( \sum_{i=1}^2 \sum_{j=1}^3 \int_{\R^2} \phi_j
              \partial_i(\xi \cdot \phi) \partial_i
              \zeta_{\phi,j}\intd x \right) - \int_{\R^2} (\nabla \phi
            : \nabla \zeta_\phi) (\phi \cdot \xi) \intd x.
	  \end{split}
	\end{align}
	For $i=1,2$ using the identity
        $0 = \partial_i (\zeta_{\phi}\cdot \phi) = \sum_{j=1}^3
        \left(\zeta_{\phi,j} \partial_i \phi_j + \phi_j \partial_i
          \zeta_{\phi,j}\right)$ we obtain
	\begin{align}\label{to_be_integrated_by_parts}
	  \begin{split}
            \sum_{i=1}^2 \sum_{j=1}^3 \int_{\R^2} \phi_j
            \partial_i(\xi \cdot \phi) \partial_i \zeta_{\phi,j}\intd
            x & = - \sum_{i=1}^2 \sum_{j=1}^3
            \int_{\R^2}\zeta_{\phi,j} \partial_i(\xi \cdot \phi)
            \partial_i \phi_j \intd x.
	  \end{split}
	\end{align}
	
	As $\xi$, $\phi$ and $\zeta_\phi$ are bounded by assumption
        and $|\nabla \phi|$ decays quadratically at infinity, we can
        integrate by parts on the right-hand side of equation
        \eqref{to_be_integrated_by_parts} without incurring additional
        boundary terms at infinity. Thus we get
	\begin{align}
	  \begin{split}
            - \sum_{i=1}^2 \sum_{j=1}^3 \int_{\R^2}\zeta_{\phi,j}
            \partial_i(\xi \cdot \phi) \partial_i \phi_j \intd x &
            =\int_{\R^2} (\xi \cdot \phi) \Big(\zeta_\phi \cdot \Delta
            \phi+ \nabla \zeta_\phi : \nabla \phi \Big) \intd x.
	  \end{split}
	\end{align}
	The first term drops out due to $\phi$ solving the harmonic
        map equation $\Delta \phi + |\nabla \phi|^2 \phi = 0$ and
        $\zeta_\phi$ being a tangent field.  The remaining term
        cancels with the other term on the right-hand side of equation
        \eqref{check_Jacobi_equation}, yielding
          \eqref{jacobi_weak}.
\end{proof}

\begin{proof}[Proof of Proposition \ref{prop:spectral_gap}]

  As the tangential vector spherical harmonics form an orthonormal
  basis for $L^2(\Sph^2;T\Sph^2)$, for each
  $\xi \in H^1(\Sph^2;T\Sph^2)$ we have the Plancherel identity
\begin{align}\label{Plancherel_spherical_harmonics}
  \int_{\Sph^2} |\xi|^2 \intd \Hd^2 =  \sum_{n\geq 1} \sum_{j=-n}^n
  \sum_{k=2,3} \left(\int_{\Sph^2} \xi \cdot \mathcal{Y}_{n,j}^{(k)}
  \intd \Hd^2\right)^2. 
\end{align}

For $N\in \N$ let
\begin{align}
  H_N := \operatorname{span}\left\{\mathcal{Y}_{n,j}^{(k)}: 0\leq
  n\leq N; \, j= -n,\ldots ,n; \, k=2,3 \right\}. 
\end{align}
In view of inequality \eqref{poincare_tangent_on_sphere} the
expression $\int_{\Sph^2} \nabla \zeta:\nabla \xi \intd \Hd^2$ defines
a scalar product on $H^1(\Sph^2;T\Sph^2)$, which we call the
$\mathring H^1$-scalar product.  Let $\pi_N$ be the
$\mathring H^1$-orthogonal projection onto $H_N$ and let
$\xi \in H^1(\Sph^2;T\Sph^2)$.  Then for $1\leq n \leq N$,
$j=-n,\ldots,n$ and $k=2,3,$ we can integrate by parts, see identity
\eqref{manifold_tangentialbeltrami}, and use the fact that
$\mathcal{Y}_{n,j}^{(k)}$ is an eigenvector of $\Delta_v$ with
eigenvalue $-n(n+1)$ to get
\begin{align}
  0 = \int_{\Sph^2} \nabla (\xi - \pi_N\xi) : \nabla
  \mathcal{Y}_{n,j}^{(k)} \intd \Hd^2 = n(n+1) \int_{\Sph^2}  (\xi -
  \pi_N\xi) \cdot  \mathcal{Y}_{n,j}^{(k)} \intd \Hd^2. 
\end{align}
As a result we have
\begin{align}
  \int_{\Sph^2}  \pi_N\xi \cdot  \mathcal{Y}_{n,j}^{(k)} \intd \Hd^2
  =\int_{\Sph^2}  \xi \cdot  \mathcal{Y}_{n,j}^{(k)} \intd \Hd^2 
\end{align}
for all $1\leq n\leq N$, $j=-n,\ldots,n$ and $k=2,3$, so that we get
\begin{align}\label{projection_l2_vs_h1}
  \pi_N \xi = \sum_{n = 1}^N \sum_{j=-n}^n
  \sum_{k=2,3}\left(\int_{\Sph^2}  \xi \cdot  \mathcal{Y}_{n,j}^{(k)}
  \intd \Hd^2 \right) \mathcal{Y}_{n,j}^{(k)}. 
\end{align}
Therefore, from identity \eqref{eigenfunctionswrth1} we obtain 
\begin{align}\label{equality_on_finite_sums}
  \sum_{n = 1}^N \sum_{j=-n}^n \sum_{k=2,3} n(n+1) \left(\int_{\Sph^2}
  \xi \cdot \mathcal{Y}_{n,j}^{(k)} \intd \Hd^2\right)^2 =
  \int_{\Sph^2} | \nabla \pi_N \xi |^2 \intd \Hd^2 \leq \int_{\Sph^2}
  | \nabla \xi |^2 \intd \Hd^2. 
\end{align}

In the limit $N\to \infty$, we consequently deduce
\begin{align}\label{lower_bound_Plancherel}
  \sum_{n \geq 1} \sum_{j=-n}^n \sum_{k=2,3} n(n+1) \left(\int_{\Sph^2}
  \xi \cdot \mathcal{Y}_{n,j}^{(k)} \intd \Hd^2\right)^2 \leq
  \int_{\Sph^2} | \nabla \xi |^2 \intd \Hd^2. 
\end{align}
By the identities \eqref{projection_l2_vs_h1} and
\eqref{Plancherel_spherical_harmonics} we have $\pi_N \xi \to \xi$ in
$L^2(\Sph^2;T\Sph^2)$, which implies that
$\nabla \pi_N \xi \warr \nabla \xi$ in $L^2(\Sph^2;\R^9)$.
As a result, from the equality in \eqref{equality_on_finite_sums}
and lower semicontinuity of the $L^2(\Sph^2; \R^9)$ norm we get
\begin{align}\label{upper_bound_Plancherel}
  \sum_{n \geq 1} \sum_{j=-n}^n \sum_{k=2,3} n(n+1)
  \left(\int_{\Sph^2} \xi \cdot \mathcal{Y}_{n,j}^{(k)} \intd
  \Hd^2\right)^2 = \lim_{N\to \infty} \int_{\Sph^2} | \nabla \pi_N \xi
  |^2 \intd \Hd^2 \geq \int_{\Sph^2} | \nabla \xi |^2 \intd \Hd^2. 
\end{align}
Combining the two inequalities \eqref{lower_bound_Plancherel} and
\eqref{upper_bound_Plancherel} we obtain
\begin{align}\label{Plancherel_gradient}
  \sum_{n \geq 1} \sum_{j=-n}^n \sum_{k=2,3} n(n+1)
  \left(\int_{\Sph^2} \xi \cdot \mathcal{Y}_{n,j}^{(k)} \intd
  \Hd^2\right)^2 =  \int_{\Sph^2} | \nabla \xi |^2 \intd \Hd^2. 
\end{align}

By the two equalities \eqref{Plancherel_spherical_harmonics} and
\eqref{Plancherel_gradient} we obtain the representation
\begin{align}
  \mathfrak{H}(\xi,\xi) = \sum_{n\geq 1} \sum_{j=-n}^n
  \sum_{k=2,3} \big( n(n+1) -2 \big) \left(\int_{\Sph^2} \xi \cdot
  \mathcal{Y}_{n,j}^{(k)} \intd \Hd^2\right)^2, 
\end{align}
from which the representation $J = H_1$ immediately follows by virtue
of $n(n+1) \geq 6$ for $n\geq 2$.  To deduce the spectral gap
property, note that by the same token we have the sharp estimate
$n(n+1) -2 \geq \frac{2}{3} n(n+1)$.  If $\xi \in \mathbb{H}^1$ we
therefore have
\begin{align}
  \mathfrak{H}(\xi,\xi) \geq \frac{2}{3} \sum_{n\geq 1} \sum_{j=-n}^n
  \sum_{k=2,3} n(n+1) \left(\int_{\Sph^2} \xi \cdot
  \mathcal{Y}_{n,j}^{(k)} \intd \Hd^2\right)^2 = \frac{2}{3}
  \int_{\Sph^2} |\nabla \xi |^2 \intd \Hd^2, 
\end{align}
which concludes the proof of estimate \eqref{spectral_gap_inequality}.

Finally, in view of the fact that the tangential vector spherical
harmonics are an orthonormal system for $L^2(\Sph^2;T\Sph^2)$ we see
that
\begin{align}
  \label{piJ}
  \pi_J (\xi) := \sum_{j=-1,0,1} \left[ \left( \int_{\Sph^2} \xi \cdot
  \mathcal{Y}_{1,j}^{(2)} \intd \Hd^2 \right) 
  \mathcal{Y}_{1,j}^{(2)} + \left( \int_{\Sph^2} \xi \cdot
  \mathcal{Y}_{1,j}^{(3)} \intd \Hd^2 \right) \mathcal{Y}_{1,j}^{(3)} 
  \right] 
\end{align}
is the $L^2$-orthogonal projection onto $J$, which by identity
\eqref{projection_l2_vs_h1} coincides with the
$\mathring H^1$-orthogonal projection.
\end{proof}

\subsection{From linear stability to rigidity}

In order to make use of the spectral gap property of Proposition
\ref{prop:spectral_gap}, we first have to find a degree one harmonic
map to which to apply it.  It turns out that it is advantageous to
take $\phi \in \mathcal{B}$ minimizing the Dirichlet distance
$D(m;\mathcal{B})$, see definition \eqref{def_Dirichlet_distance},
between $\mathcal B$ and $m \in \mathcal{C}$, which is possible due to
the following Lemma \ref{lem:closest_harmonic_map}.  As in its proof
it is more convenient to deal with Belavin-Polyakov profiles rather
than M{\"o}bius transformations, we formulate it in the
$\R^2$-setting.

\begin{lemma}\label{lem:closest_harmonic_map}
  For any $m \in \mathcal{C}$ there exists $ \phi \in \mathcal{B}$
  such that
	\begin{align}
          D(m;\mathcal{B}) = \left( \int_{\R^2} | \nabla ( m
          -\phi) |^2 \intd x \right)^{\frac12}.
	\end{align}
\end{lemma}

With this statements, we are in a position to prove a local version of
Theorem \ref{thm:quantitative_stability} by projecting $m-\phi$ onto a
vector field tangent to $\phi$ and using Lemma
\ref{lem:linearization_estimate} to control the resulting higher order
terms.

\begin{lemma}\label{lem:local_nonlinear_stability}
  Let $\tilde \eta>0$ be as in Lemma
  \ref{lem:linearization_estimate}. For $m \in \mathcal{C}$ with
  $D^2(m;\mathcal{B}) < \tilde \eta $ we have
	\begin{align}
          \left(\frac{2}{3} - \frac{2}{3}C_4^2  D(m;\mathcal{B})
          - \frac{19}{12}C_4^4 D^2(m;\mathcal{B})\right)
          D ^2(m;\mathcal{B}) \leq F(m) - 8\pi, 
	\end{align}
	where $C_4$ is the constant from Lemma
        \ref{lem:linearization_estimate}.
\end{lemma}

\begin{proof}[Proof of Lemma \ref{lem:closest_harmonic_map}]
  Towards a contradiction, we assume that
  $\inf_{\phi \in \mathcal{B}} \int_{\R^2} |\nabla (m- \phi)|^2 \intd
  x$ is not attained.  Throughout the proof, we ignore the
  relabeling of subsequences without further comment.
	
  \textit{Step 1: If the infimum is not attained, then
    $\int_{\R^2} |\nabla (m- \phi)|^2 \intd x > \int_{\R^2} |\nabla
    m|^2 \intd x + 8\pi$ for all $\phi \in \mathcal{B}$.}

  For $n\in \N$, let $R_n \in \operatorname{SO}(3)$,
  $0<\rho_n < \infty$ and $x_n \in \R^2$ be such that
  $\phi_n:=R_n \Phi\left(\rho_n^{-1}\left( \bigcdot - x_n\right)
  \right) \in \mathcal{B}$ satisfies
	\begin{align}
          \lim_{n\to\infty} \int_{\R^2} | \nabla ( m -  \phi_n ) |^2
          \intd x = \inf_{\phi \in \mathcal{B}} \int_{\R^2} |\nabla
          (m- \phi)|^2 \intd x . 
	\end{align}	
	As $\operatorname{SO(3)}$ is compact, there exists a
        subsequence and $R \in \operatorname{SO(3)}$ such that
        $\lim_{n \to \infty} R_n = R$.  By direct computation, we have
        uniformly for all $\tilde \phi \in \mathcal{B}$
	\begin{align}
          \lim_{n\to \infty} \int_{\R^2} \left|\nabla
          \left(R_{n} \tilde \phi\right) - \nabla \left(R\tilde \phi
          \right) \right|^2 \intd x =0. 
	\end{align}
        We may thus suppose that $R_n = R$ for all $n \in \N$.  Due to
        the fact that there does not exist an optimal approximating
        Belavin-Polyakov profile we have
        $\lim_{n \to \infty} \rho_n = 0$,
        $\lim_{n\to \infty} \rho_n = \infty$, or
        $\lim_{n\to \infty} x_n = \infty$.
	
	Let us first deal with the case $\lim_{n\to \infty} \rho_n =
        0$, which implies $\nabla \phi_n \warr 0$ in $L^2$.
        Consequently, by expanding the square we get
	\begin{align}
          \inf_{\phi \in \mathcal{B}} \int_{\R^2} |\nabla (m- \phi)|^2
          \intd x =  \lim_{n\to\infty} \int_{\R^2} | \nabla ( m -
          \phi_n ) |^2 \intd x  = \int_{\R^2} | \nabla  m|^2 \intd x +
          8\pi. 
	\end{align}
	As the infimum is not achieved, we obtain
	\begin{align}
          \int_{\R^2} |\nabla (m- \phi)|^2 \intd x > \int_{\R^2}
          |\nabla m|^2 \intd x + 8\pi 
	\end{align}
	for all $\phi \in \mathcal{B}$.
	
	In the case $\lim_{n\to \infty} \rho_n = \infty$, we rescale
        $m_n:= m(\rho_n x + x_n)$ and observe $\nabla m_n \warr 0$ in
        $L^2$.  Similarly as in the previous case we thus get for
          all $\phi \in \mathcal{B}$:
	\begin{align}
          \int_{\R^2} |\nabla (m- \phi)|^2 \intd x >
          \lim_{n\to\infty} \int_{\R^2} | \nabla ( m_n - R \Phi ) |^2
          \intd x  = \int_{\R^2} | \nabla  m|^2 \intd x + 8\pi. 
	\end{align}
	
	In dealing with the case $\lim_{n\to \infty} x_n = \infty$ we
        may consequently assume that
        $\lim_{n\to \infty} \rho_n = \rho \in (0,\infty)$.  Once again
        we then get $\nabla \phi_n \warr 0$ in $L^2$, and we conclude
        as in the first case.
	
\textit{Step 2: Derive the contradiction.}\\
	Expanding the square in the result of Step 1 yields 
	\begin{align}\label{contradiction_bp}
          \int_{\R^2} \nabla m : \nabla \phi
          \intd x < 0
	\end{align}
	for every $\phi \in \mathcal{B}$.  Now, for $x\in \R^2$ we
        define the four Belavin-Polyakov profiles:
          \begin{align}
            \label{eq:phisigns}
            \phi_{+,+}(x) & := (\Phi_1(x), \Phi_2(x), \Phi_3(x)), \\
            \phi_{-,+}(x) & :=  (\Phi_2(x), \Phi_1(x), -\Phi_3(x)), \\
            \phi_{+,-}(x) & :=  (-\Phi_1(x), -\Phi_2(x),\Phi_3(x)), \\
            \phi_{-,-}(x) & :=  (-\Phi_2(x), -\Phi_1(x), -\Phi_3(x)).
          \end{align}
         It is straightforward to see that
          \begin{align}
            \label{phipp}
          \int_{\R^2} \nabla m_3 \cdot \nabla \phi_{+,+;3} \intd x
          & = - \int_{\R^2} \nabla  m_3 \cdot \nabla \phi_{-,-;3}
            \intd x, \\ 
          \int_{\R^2} \nabla m_3 \cdot \nabla \phi_{+,-;3} \intd x
          & = - \int_{\R^2} \nabla m_3 \cdot \nabla \phi_{-,+;3}
            \intd x, \\
          \int_{\R^2} \nabla m' : \nabla \phi_{\pm,+}' \intd x
          & = - \int_{\R^2} \nabla m' : \nabla \phi_{\pm,-}' \intd
            x. \label{phimm} 
	\end{align}
     Therefore, by \eqref{contradiction_bp} and
        \eqref{phipp}--\eqref{phimm} we get
        \begin{align}
         \begin{split}
          0 > & \int_{\R^2} \nabla m : \left( \nabla \phi_{+,+} + \nabla
                \phi_{+,-} + \nabla \phi_{-,+} + \nabla \phi_{-,-} \right)
                \intd x  \\
              & = \int_{\R^2} \nabla m_3 \cdot \left( \nabla
                \phi_{+,+;3} + \nabla 
                \phi_{+,-;3} + \nabla \phi_{-,+;3} + \nabla \phi_{-,-;3} \right)
                \intd x  \\
              & \ \ \ \ + \int_{\R^2} \nabla m' : \left( \nabla
                \phi_{+,+}' + \nabla 
                \phi_{+,-}' + \nabla \phi_{-,+}' + \nabla \phi_{-,-}' \right)
                \intd x = 0,
          \end{split}
        \end{align}
        a contradiction.
\end{proof}

\begin{proof}[Proof of Lemma \ref{lem:local_nonlinear_stability}]
  
  Lemma \ref{lem:closest_harmonic_map} ensures the existence of
  $\phi \in \mathcal{B}$ such that
	\begin{align}
		\int_{\R^2} | \nabla (m - \phi)  |^2 \intd x = D^2(m;\mathcal{B}).
	\end{align}
	As $\phi$ arises from $\Phi$ purely by invariances of the
        energy, we may without loss of generality suppose $\phi =\Phi$
        by re-defining $m$.  Throughout the proof, we abbreviate
        $\tilde J := \{\xi \circ \Phi: \xi \in J \}$.
	
        \textit{Step 1: We decompose $m - \Phi$ into a vector field
          parallel to $\Phi$, a Jacobi field and a tangent vector
          field normal to Jacobi fields. Furthermore, we state a few
          preliminary estimates and identities.}\\
	For $\xi \in H^1_{\mathrm{w}}(\R^2; T_\Phi\Sph^2)$, let
        $\pi_{\tilde J}(\xi) : = \pi_J(\xi \circ \Phi^{-1})
        \circ\Phi$, where $\pi_J$ is defined in \eqref{piJ}, which
        makes sense in view of Lemma \ref{lem:conformal_invariance}.
        We decompose $\zeta := m-\Phi$ pointwise into the three parts:
	\begin{align}
          \zeta_\|
          & := \left( \zeta \cdot  \Phi \right) \Phi =
            -\frac{1}{2}|m-\Phi|^2
            \Phi, \label{normal_projection}\\ 
          \zeta_{\tilde J}
          & := \pi_{\tilde J} (\zeta  - \zeta_\|),\\
          \zeta^*
          & := \zeta - \zeta_\| - \zeta_{\tilde J},\label{zeta_star} 
	\end{align}
        where we noted that
        $\zeta - \zeta_\| \in \mathring{H}^1(\R^2; \R^3) \cap
        L^\infty(\R^2; \R^3) \subset H^1_\mathrm{w}(\R^2; \R^3)$.
        Since by Proposition \ref{prop:spectral_gap} the map
        $\pi_{ J}$ is both an $L^2(\Sph^2;T\Sph^2)$-orthogonal and an
        $\mathring H^1(\Sph^2; T\Sph^2)$-orthogonal projection, Lemma
        \ref{lem:conformal_invariance} implies that
	\begin{align}
	\label{additional_orthogonality}
          \int_{\R^2} \zeta_{\tilde J} \cdot \zeta^* |\nabla
          \Phi|^2 \intd x & = 0,\\ 	
	\label{J_star_orthogonal}
          \int_{\R^2} \nabla \zeta_{\tilde J} : \nabla \zeta^*
          \intd x & = 0. 
	\end{align}
        Lemma \ref{lem:linearization_estimate}, which we may apply due
        to our smallness assumption $D^2(m;\mathcal{B}) <\tilde \eta$,
        together with Lemma \ref{lem:conformal_invariance} tells us
        that
	\begin{align}\label{linearization_estimate_4}
          \int_{\R^2} |\zeta_\| |^2 | \nabla \Phi |^2 \intd x =
          \frac{1}{4} \int_{\R^2}  |m-\Phi |^4 |\nabla \Phi |^2 \intd
          x \leq \frac{C_4^4}{4} D^4(m;\mathcal{B}). 
	\end{align}
	
        \textit{Step 2: We claim that \begin{align} \int_{\R^2} |\zeta
            |^2 | \nabla \Phi |^2 \intd x \leq \int_{\R^2} |\zeta^*
            |^2 | \nabla \Phi |^2 \intd x + \frac{5}{4}C_4^4
            D^4(m;\mathcal{B}).
	\end{align}
      }\\
      Indeed, by construction we have
      $(\zeta_{\tilde J} + \zeta^*) \cdot \zeta_\parallel = 0$ almost
      everywhere. Therefore, by \eqref{additional_orthogonality} we
      obtain
	\begin{align}\label{some_identity}
	  \begin{split}
            \int_{\R^2} |\zeta |^2 | \nabla \Phi |^2 \intd x & =
            \int_{\R^2}\left( |\zeta_\| |^2 + 2 \zeta_\|\cdot (
              \zeta_{\tilde J} + \zeta^* ) + |\zeta_{\tilde J}|^2 + 2
              \zeta_{\tilde J} \cdot \zeta^* + | \zeta^*|^2 \right) |
            \nabla \Phi |^2
            \intd x \\
            & = \int_{\R^2}\left( |\zeta_\| |^2 + |\zeta_{\tilde J}|^2
              + | \zeta^*|^2 \right) | \nabla \Phi |^2 \intd x.
	  \end{split}
	\end{align}
        
	The $\zeta_\|$-term in \eqref{some_identity} is controlled by
        \eqref{linearization_estimate_4}, so that we only have to
        estimate the $\zeta_{\tilde J}$-term. Furthermore, since
        $\zeta_{\tilde J}$ is a Jacobi field, Proposition
        \ref{prop:spectral_gap} implies that we can find $\eps>0$ and
        a smooth map $\phi : (-\eps,\eps) \to \mathcal{B}$ such that
        $\phi(0)= \Phi$ and
        $\partial_t \phi(t) |_{t=0} = \zeta_{\tilde J}$.
        Differentiating the expression
        $\int_{\R^2} |\nabla (m - \phi(t))|^2 \intd x$ in $t$ and
        using the fact that $t=0$ is its minimum, we obtain
	\begin{align}\label{orthogonality}
          \int_{\R^2} \nabla \zeta : \nabla \zeta_{\tilde J} \intd x = 0.
	\end{align}	
	Thus we have together with
        $\zeta_{\tilde J} = \zeta - \zeta_\| - \zeta^*$ and the identity
        \eqref{J_star_orthogonal} that
	\begin{align}\label{identity_1}
	  \begin{split}
            \int_{\R^2} | \nabla \zeta_{\tilde J} |^2 \intd x & =
            \int_{\R^2} \nabla (\zeta-\zeta_\|- \zeta^*) : \nabla
            \zeta_{\tilde J} \intd x = - \int_{\R^2} \nabla \zeta_\| :
            \nabla \zeta_{\tilde J} \intd x .
		\end{split}
	\end{align}
	By Proposition \ref{prop:spectral_gap}, $\zeta_{\tilde J}$ is
        smooth, and we may use Lemma \ref{lem:jacobi}, the fact
        that $\zeta_\| \cdot \zeta_{\tilde J} =0$ almost
        everywhere, as well as \eqref{normal_projection} to
        obtain
	\begin{align}\label{identity_2}
	  \begin{split}
            - \int_{\R^2} \nabla \zeta_\| : \nabla \zeta_{\tilde J}
            \intd x & = - \int_{\R^2} 2 (\zeta_\| \cdot \Phi ) (
            \nabla \Phi : \nabla \zeta_{\tilde J}) \intd x =
            \int_{\R^2} |m-\Phi|^2 (\nabla \Phi : \nabla \zeta_{\tilde
              J}) \intd x .
	  \end{split}
	\end{align}
	The two identities \eqref{identity_1} and \eqref{identity_2}
        allow us to obtain from the Cauchy-Schwarz inequality and the
        estimate \eqref{linearization_estimate_4} that
	\begin{align}\label{jacobi_component_H1}
          \int_{\R^2} | \nabla \zeta_{\tilde J}|^2 \intd x \leq C_4^4
          D^4(m;\mathcal{B}). 
	\end{align}
	This and the Poincar{\'e} type inequality
        \eqref{poincare_for_tangent_fields} furthermore implies
	\begin{align}\label{jacobi_component_L2}
          \int_{\R^2} | \zeta_{\tilde J}|^2 |\nabla \Phi |^2 \intd x
          \leq C_4^4 D^4(m;\mathcal{B}), 
	\end{align}
	which together with \eqref{linearization_estimate_4}
        yields the claim.

        \textit{Step 3: We also claim that
	\begin{align}\label{claim_3}
          \int_{\R^2}  |\nabla \zeta|^2 \intd x =  \int_{\R^2}
          \left( | \nabla \zeta_\| |^2 +2 \nabla \zeta_\| : \nabla
          (\zeta - \zeta_\|) +  | \nabla \zeta_{\tilde J} |^2 +
          |\nabla \zeta^* |^2 \right) \intd x 
	\end{align}
	and 
	\begin{align}\label{claim_step_4}
          \int_{\R^2}2 \nabla \zeta_\| : \nabla (\zeta - \zeta_\|)
          \intd x + \int_{\R^2}|\nabla  \zeta_\| |^2 \intd x  \geq -
          2 C_4^2D^3(m;\mathcal{B}) - C_4^4D^4(m;\mathcal{B}). 
	\end{align}
      }\\
      The first equality follows directly from
      \eqref{J_star_orthogonal}. With the help of the identities
      $\partial_k[ \Phi \cdot (\zeta - \zeta_\|)] = 0$ and
      $\partial_k \zeta_{\|,l} = - \frac{1}{2} |m-\Phi |^2 \partial_k
      \Phi_l-\frac{1}{2} \Phi_l \partial_k |m-\Phi|^2 $ a.e. for
      $k=1,2$ and $l=1,2,3$ obtained from \eqref{normal_projection},
      the second term in the right-hand side of \eqref{claim_3} is
	\begin{align}\label{some_mixed_term}
	  \begin{split}
            \int_{\R^2} 2 \nabla \zeta_\| : \nabla (\zeta - \zeta_\|)
            \intd x & = \int_{\R^2} \sum_{k=1}^2\sum_{l=1}^3 \left(-
              |\zeta|^2\, \partial_k \Phi_l \, \partial_k (\zeta
              -\zeta_\|)_l - \Phi_l \, \partial_k |\zeta|^2\,
              \partial_k (\zeta -
              \zeta_\| )_l \right) \intd x \\
            & = \int_{\R^2} \sum_{k=1}^2\sum_{l=1}^3 \left( (\zeta -
              \zeta_\| )_l \, \partial_k |\zeta|^2\, \partial_k \Phi_l
              - |\zeta|^2 \,\partial_k \Phi_l \, \partial_k (\zeta
              -\zeta_\|)_l \right) \intd x .
	  \end{split}
	\end{align}
	As $|\nabla \Phi(x)| = O(|x|^{-2})$ for $|x| \to \infty$,
        we may integrate by parts in the first term to get
	\begin{align}
	  \begin{split}
            & \quad \int_{\R^2} \sum_{k=1}^2\sum_{l=1}^3 (\zeta -
            \zeta_\| )_l \, \partial_k |\zeta|^2\,\partial_k \Phi_l
            \intd x \\
            & = - \int_{\R^2} \left( |\zeta|^2 (\zeta - \zeta_\|
              )\cdot \Delta \Phi + \sum_{k=1}^2\sum_{l=1}^3
              |\zeta|^2\, \partial_k (\zeta - \zeta_\| )_l \,
              \partial_k \Phi_l \right)\intd x.
	  \end{split}
	\end{align}
	The first term drops out by the harmonic map equation
        $\Delta \Phi + |\nabla \Phi |^2\Phi = 0$ and the fact
          that $\Phi \cdot (\zeta - \zeta_\|) = 0$ almost everywhere,
        while the second one combines with the second term on the
        right-hand side of identity \eqref{some_mixed_term} to give
	\begin{align}
	  \begin{split}
            \int_{\R^2}2 \nabla \zeta_\| : \nabla (\zeta - \zeta_\|)
            \intd x = - \int_{\R^2} 2 |\zeta|^2 \,\nabla \Phi : \nabla
            \zeta \intd x + \int_{\R^2} 2 |\zeta|^2 \,\nabla \Phi :
            \nabla \zeta_\| \intd x.
	  \end{split}
	\end{align}
	The Cauchy-Schwarz inequality and the equality in
        \eqref{def_Dirichlet_distance} applied to the first term on
        the right-hand side, as well as Young's inequality applied to
        the second term imply
	\begin{align}
	  \begin{split}
            - \int_{\R^2}2 \nabla \zeta_\| : \nabla (\zeta - \zeta_\|)
            \intd x & \leq 2 \left(\int_{\R^2}|\zeta|^4 |\nabla
              \Phi|^2 \intd x \right)^\frac{1}{2}
            D(m;\mathcal{B}) \\
            & \qquad + \int_{\R^2}|\zeta|^4 |\nabla \Phi|^2 \intd x
            +\int_{\R^2}|\nabla \zeta_\| |^2 \intd x,
	  \end{split}
	\end{align}
	By estimate \eqref{linearization_estimate_4}, this gives the
        second part of the claim.
	
        \textit{Step 4: Conclusion.}\\
	We now use Steps 2 and 3 to decompose the conclusion of Lemma
        \ref{lem:relation_excess_hamiltonian} in terms of $\zeta^*$,
        $\zeta_\|$ and $\zeta_{\tilde J}$, taking care to not estimate
        $\int_{\R^2} |\nabla \zeta^* |^2 \intd x$ in the identity
        \eqref{claim_3} and only to estimate one third of the
        remaining terms, by which we obtain
	\begin{align}\label{let_star_live}
	 \begin{split}
           F(m) - 8\pi & = \int_{\R^2} | \nabla (m-\Phi)|^2 -
           (m-\Phi)^2 |\nabla \Phi |^2 \intd x \\
           & \geq \mathfrak{H}_{\Phi}(\zeta^*, \zeta^*) + \frac{2}{3}
           \int_{\R^2} \left( | \nabla \zeta_\| |^2 +2 \nabla \zeta_\|
             : \nabla (\zeta - \zeta_\|) + | \nabla \zeta_{\tilde J}
             |^2 \right) \intd x\\
           &\qquad - \frac{2}{3} C_4^2D^3(m;\mathcal{B}) -
           \frac{19}{12} C_4^4D^4(m;\mathcal{B}).
	\end{split}
	\end{align}
	As we have $\zeta^*\circ \Phi^{-1} \in \mathbb{H}^1$ by
        construction, see definition \eqref{zeta_star}, the spectral
        gap proved in Proposition \ref{prop:spectral_gap} along with
        Lemma \ref{lem:conformal_invariance} gives
	\begin{align}
          \mathfrak{H}_{\Phi}(\zeta^*, \zeta^*)
          & =
            \mathfrak{H}(\zeta^*\circ
            \Phi^{-1},
            \zeta^*\circ
            \Phi^{-1})  \geq
            \frac{2}{3}
            \int_{\Sph^2} \left|
            \nabla\left(
            \zeta^*\circ
            \Phi^{-1}\right)
            \right|^2 \intd
            \Hd^2 = \frac{2}{3}
            \int_{\R^2} | \nabla
            \zeta^* |^2 \intd
            x. 
	\end{align}
	As a result, the estimate \eqref{let_star_live} and Step 3 imply
	\begin{align}
	  \begin{split}
            F(m) - 8\pi & \geq \frac{2}{3} \int_{\R^2} | \nabla
            (m-\Phi)|^2 \intd x - \frac{2}{3} C_4^2D^3(m;\mathcal{B})
            - \frac{19}{12} C_4^4 D^4(m;\mathcal{B}),
	  \end{split}
	\end{align}
	concluding the proof.
\end{proof}

\subsection{Proofs of Theorem \ref{thm:quantitative_stability}, Lemma
  \ref{lem:linearization_estimate} and Corollary
  \ref{cor:stability_conformal}}

Having collected all the necessary intermediate statements in the
two previous subsections, we now proceed with proving our main
results, starting with Theorem \ref{thm:quantitative_stability}.

\begin{proof}[Proof of Theorem \ref{thm:quantitative_stability}]
  \textit{Step 1: For all $\alpha>0$ there exists $\beta>0$ such that
    for all $m\in \mathcal{C}$ with $F(m) -8\pi \leq \beta$ we have
    $D(m;\mathcal{B}) < \alpha$.}\\
  Towards a contradiction, we assume that there exists $\alpha>0$ and
  a sequence of $m_n \in \mathcal{C}$ for $n \in \N$ such that
  $F(m_n) -8\pi \leq \frac{1}{n}$ and
  $D(m_n;\mathcal{B}) \geq \alpha$.  Then, for $r > 0$ we introduce
  the L{\'e}vy concentration function
  $Q_n(r) := \sup_{x\in \R^2} \mu_n (\ball{x}{r})$ associated with the
  measure $\mu_n$ such that $\intd \mu_n := |\nabla m_n|^2 \intd x$.
  Observe that $Q_n(r)$ is a non-decreasing continuous function of $r$
  and satisfies $\lim_{r\to 0} Q_n(r) = 0$ and
  $\lim_{r \to \infty} Q_n(r) = F(m_n) \geq 8\pi$.  Using translation
  and scale invariance of $F$, we can thus assume the sequence $(m_n)$
  to satisfy
\begin{align}\label{compactness_normalization}
  \int_{\ball{0}{1}} |\nabla m_n|^2 \intd x = Q_n(1)  = 4\pi
\end{align}
for all $n\in \N$.  Lemma \ref{lem:conformal_invariance} implies that
the sequence of $\widetilde m_n (y) := m_n \circ \Phi^{-1}(y)$ for
$y \in \Sph^2$ and $n\in \N$ is a minimizing sequence for
$F_{\Sph^2}$.  Consequently, \cite[Theorem 1'']{lin1999mapping}
implies that there exists a harmonic map
$\widetilde m :\Sph^2 \to \Sph^2$ and a defect measure $\nu$ on
$\Sph^2$ supported on an at most countable set such that
$\widetilde m_n \warr \widetilde m$ in $H^1(\Sph^2;\R^3)$ and
$|\nabla \widetilde m_n|^2 \intd y \stackrel{*}{\warr} |\nabla
\widetilde m|^2 \intd y + \nu$ as Radon measures.  Furthermore,
\cite[Theorem 5.8]{lin1999mapping} implies that for $\nu \not= 0$
there exist $P \in \mathbb N$ points $\{y_p\}_{p=1}^P \in \Sph^2$ such
that $\nu = \sum_{p=1}^P 8\pi N_p \delta_{y_p}$ for some $N_p \in \N$.

If we had $\nu \neq 0$, then on account of
$\lim_{n\to \infty} F_{\Sph^2}(\widetilde m_n) = 8\pi$ there can be at
most a single defect $y_1 \in \Sph^2$ such that
$\nu = 8\pi\delta_{y_1}$, and we must have $\nabla \widetilde m = 0$
almost everywhere.  We must consequently have
$y_1 \in \overline{\Phi(\ball{0}{1})}$, which, however, implies
\begin{align}
  \lim_{n\to \infty} \int_{\ball{\Phi^{-1}(y_1)}{1}} |\nabla
  m_n|^2 \intd x = 8\pi, 
\end{align}
contradicting the second equality in
\eqref{compactness_normalization}.
  
Therefore, we must have $\nu = 0$, in which case the convergence
$|\nabla \widetilde m_n|^2 \intd y \stackrel{*}{\warr} |\nabla
\widetilde m|^2 \intd y$ gives
\begin{align}
	\int_{\Sph^2} |\nabla \widetilde m|^2 \intd y = 8\pi,
\end{align}
which, in turn, implies that $m_n \to m$ in
$\mathring{H}^1(\R^2 ; \R^3)$.  However, this contradicts the
assumption $D(m_n; \mathcal{B}) \geq \alpha$ for all $n \in \N$.

\textit{Step 2: Conclusion.}\\
By Step 1 we can choose $\beta>0$ such that for all
$m \in \mathcal{C}$ with $F(m) - 8\pi \leq \beta$ we have
$D^2(m;\mathcal{B}) \leq \tilde \eta$, where $\tilde \eta$ is as in
Lemma \ref{lem:linearization_estimate}. If we additionally
choose $\beta>0$ small enough, Lemma
\ref{lem:local_nonlinear_stability} thus implies that there exists
$ \hat \eta >0$ such that
\begin{align}
  \hat \eta D^2(m;\mathcal{B}) \leq F(m) - 8\pi
\end{align}
for all $m \in \mathcal{C}$ with $F(m) - 8\pi \leq \beta$.  For
$m\in \mathcal{C}$ with $F(m) - 8\pi \geq \beta$ we use Lemma
\ref{lem:relation_excess_hamiltonian} together with
$\int_{\R^2} (m-\phi)^2 |\nabla \phi|^2\intd x \leq 32 \pi$ for all
$\phi \in \mathcal{B}$ to get 
\begin{align}
  D^2(m;\mathcal{B}) \leq F(m) + 24\pi \leq \left( 1 +
  \frac{32 \pi}{\beta}\right) \left(F(m) - 8\pi  \right). 
\end{align} 
Given $m\in \mathcal{C}$, existence of $\phi$ such that
\begin{align}
  \int_{\R^2} \left| \nabla \left( m - \phi \right)\right|^2
  \intd x = D^2(m;\mathcal{B})
\end{align}
was proved in Lemma \ref{lem:closest_harmonic_map}.  Thus the
theorem holds for
$\eta := \min\left\{\hat \eta ,\left( 1 +
    \frac{32\pi}{\beta}\right)^{-1} \right\}$.
\end{proof}

\begin{proof}[Proof of Lemma \ref{lem:linearization_estimate}]

  By Jensen's inequality, it is sufficient to prove the estimate for
  $p \geq 2$.  The Sobolev inequality applied to the map
  $u(y):= m(y) - y - \dashint_{\Sph^2} \left( m(\tilde y) -\tilde y
  \right) \intd \Hd^2(\tilde y)$ for $y\in \Sph^2$ implies (see, for
  example, \cite[Theorem 4]{beckner1993sharp})
	\begin{align}
	  \begin{split}
            & \quad \left( \dashint_{\Sph^2} \left|  m(y)  - y  -
                \dashint_{\Sph^2} \left(  m(\tilde y) - \tilde y
                \right) \intd  \Hd^2(\tilde y) \right |^p  \intd
              \Hd^2(y) \right)^\frac{2}{p} \\ 
            & \leq \frac{p-2}{2} \dashint_{\Sph^2}\left| \nabla\left(
                m(y) - y \right) \right|^2 \intd \Hd^2(y) +
            \dashint_{\Sph^2} \left| m (y) - y - \dashint_{\Sph^2}
              \left( m(\tilde y) - \tilde y \right) \intd \Hd^2(\tilde
              y) \right |^2 \intd \Hd^2(y).
	  \end{split}
	\end{align}
	The sharp Poincar{\'e} type inequality, following from the
        first nontrivial eigenvalue of the negative Laplace-Beltrami
        operator $-\Delta$ on $\Sph^2$, see for example \cite[Theorem
        3.67]{freeden2009spherical}, implies
	\begin{align}
          \dashint_{\Sph^2} \left|  m (y) - y  - \dashint_{\Sph^2}
          \left(  m(\tilde z) - \tilde y \right) \intd  \Hd^2(\tilde
          y) \right |^2  \intd \Hd^2(y) \leq \frac{1}{2}
          \dashint_{\Sph^2}\left| \nabla\left(  m(y)  - y \right)
          \right|^2 \intd \Hd^2(y), 
	\end{align}
	so that we obtain the Sobolev-Poincar{\'e} inequality
	\begin{align}\label{sobolev}
	   \begin{split}
             \left( \dashint_{\Sph^2} \left| m (y) - y -
                 \dashint_{\Sph^2} \left( m(\tilde y) - \tilde y
                 \right) \intd \Hd^2(\tilde y) \right |^p \intd
               \Hd^2(y) \right)^\frac{2}{p} \leq \frac{p -1}{2}
             \dashint_{\Sph^2}\left| \nabla\left( m(y) - y \right)
             \right|^2 \intd \Hd^2(y).
	  \end{split}
	\end{align}

	As the right-hand side of this estimate is part of the desired
        Sobolev inequality, we only have to control the average.  To
        this end, for the moment we only consider the case that
        $m$ is smooth.  By symmetry of $\Sph^2$, we obviously have
	\begin{align}\label{average_vanishes_1}
          \dashint_{\Sph^2} \left(  m(\tilde y) - \tilde y \right)
          \intd  \Hd^2(\tilde y) = \dashint_{\Sph^2}  m(\tilde y)
          \intd  \Hd^2(\tilde y). 
	\end{align}
	With the goal of finding a similar cancellation for the
        remaining term, we work towards writing it on the image of
        $m$.
	
	Setting $\widetilde m := m \circ \Phi$ and using Lemma
        \ref{lem:conformal_invariance}, we have
	\begin{align}\label{average_vanishes_2}
          \dashint_{\Sph^2}  m \intd \Hd^2 = \frac{1}{8\pi}
          \int_{\R^2} \widetilde m |\nabla \Phi|^2 \intd x. 
	\end{align}
	The identity
        $|\nabla \Phi |^2 - |\nabla \widetilde m|^2 = 2\nabla \Phi :
        \nabla (\Phi - \widetilde m) + |\nabla (\Phi - \widetilde
        m)|^2$ and Cauchy-Schwarz inequality imply
	\begin{align}\label{average_vanishes_3}
	  \begin{split}
            & \quad \left| \frac{1}{8\pi}\int_{\R^2} \widetilde m
              \left(|\nabla \widetilde m|^2 - |\nabla \Phi |^2
              \right) \intd x \right| \\
            & \leq \frac{1}{\sqrt{2\pi}} \left(\int_{\R^2} |\nabla
              (\widetilde m-\Phi)|^2 \intd x\right)^\frac{1}{2} +
            \frac{1}{8\pi}\int_{\R^2} |\nabla (\widetilde m-\Phi)|^2
            \intd x.
	  \end{split}
	\end{align}
	With the goal of further rewriting
        $\frac{1}{8\pi}\int_{\R^2} \widetilde m |\nabla \widetilde
        m|^2\intd x$, we use the estimate \eqref{completed_square} to
        get
	\begin{align}\label{topological_bound_density}
          |\nabla \widetilde m|^2 \geq 2 | \widetilde m\cdot(\partial_1
          \widetilde m \times \partial_2 \widetilde m)|, 
	\end{align}
	and thus we obtain
	\begin{align}
	  \begin{split}
            & \quad \left| \frac{1}{8\pi}\int_{\R^2} \widetilde m
              \left( 2 | \widetilde m\cdot(\partial_1 \widetilde m
                \times \partial_2 \widetilde m) | - |\nabla \widetilde
                m|^2\right) \intd x \right|\\
            & \leq \frac{1}{8\pi} \int_{\R^2}\left( |\nabla \widetilde
              m|^2 - 2 |\widetilde m\cdot(\partial_1 \widetilde m
              \times \partial_2 \widetilde m)| \right)
            \intd x\\
            & \leq \frac{1}{8\pi} \int_{\R^2} \left( |\nabla
              \widetilde m|^2 - 2 \widetilde m\cdot(\partial_1
              \widetilde m \times \partial_2 \widetilde m) \right)
            \intd x\\
            & = \frac{1}{8\pi} \int_{\R^2} |\nabla \widetilde m|^2
            \intd x- \mathcal{N}\left(\widetilde m\right)
	  \end{split}
	\end{align}
	By choosing $\tilde\eta > 0$ small enough, we get
        $\mathcal{N}(\widetilde m)=1$ due to
        $\mathring H^1$-continuity of $\mathcal{N}$, so that the above
        turns into
	\begin{align}
          \left| \frac{1}{8\pi}\int_{\R^2} \widetilde m \left(  2 |
          \widetilde m\cdot(\partial_1 \widetilde m \times \partial_2
          \widetilde m ) |  - |\nabla \widetilde m|^2 \right) \intd x
          \right|\ \leq \frac{1}{8\pi} \left(\int_{\R^2} |\nabla
          \widetilde m|^2 \intd x - 8\pi\right). 
	\end{align}
	By Lemma \ref{lem:relation_excess_hamiltonian}, we get
	\begin{align}\label{perlen_vor_die_saeue}
	  \begin{split}
            \frac{1}{8\pi}\left(\int_{\R^2} |\nabla \widetilde m|^2
              \intd x - 8\pi\right) \leq \frac{1}{8\pi}\int_{\R^2}
            |\nabla (\widetilde m -\Phi)|^2 \intd x,
	  \end{split}
	\end{align}
	so that the above gives
	\begin{align}\label{average_vanishes_4}
          \left| \frac{1}{8\pi}\int_{\R^2} \widetilde m \left(  2 |
          \widetilde m\cdot(\partial_1 \widetilde m \times \partial_2
          \widetilde m ) | - |\nabla \widetilde m|^2 \right) \intd x
          \right|\ \leq  \frac{1}{8\pi}\int_{\R^2} |\nabla (\widetilde
          m -\Phi)|^2 \intd x. 
	\end{align}
	
	We now aim to use the area formula \cite[Theorem
        2.71]{Ambrosio} to rewrite the first term on the left-hand
        side as an integral over the image of $\widetilde m$.  As we
        have
        $\widetilde m \cdot \partial_1\widetilde m = \widetilde m
        \cdot \partial_2\widetilde m = 0$ everywhere, the two vectors
        $\widetilde m$ and
        $\partial_1\widetilde m \times \partial_2\widetilde m$ are
        parallel.  Therefore, we have
        \begin{align}
          \begin{split}
            |\widetilde m \cdot (\partial_1\widetilde m \times
            \partial_2\widetilde m)|^2 & = |\partial_1\widetilde m
            \times \partial_2\widetilde m|^2= \partial_1 \widetilde m
            \cdot (\partial_2 \widetilde m \times
            (\partial_1\widetilde m \times \partial_2\widetilde m))
            )\\ 
            & = |\partial_1 \widetilde m|^2|\partial_2 \widetilde m|^2
            - (\partial_1 \widetilde m \cdot \partial_2 \widetilde
            m)^2,
         \end{split}
        \end{align}
        so that
        $|\widetilde m \cdot (\partial_1\widetilde m \times
        \partial_2\widetilde m )|$ is the modulus of the Jacobian of
        $\widetilde m$.  Consequently, the area formula gives
	\begin{align}\label{something1}
	  \begin{split}
            \frac{1}{4\pi} \int_{\R^2} \widetilde m\, |\widetilde m
            \cdot(\partial_1\widetilde m \times \partial_2 \widetilde
            m) | \intd x
            & = \dashint_{\Sph^2} z \, \Hd^0(\{\widetilde m^{-1}(z)\})
            \intd \Hd^2(z) \\ 
            & = \dashint_{\Sph^2} z \, \Hd^0(\{ m^{-1}(z)\}) \intd
            \Hd^2(z).
	  \end{split}
	\end{align}
	For all $z \in \Sph^2$ there exists at least one $y \in
        \Sph^2$ such that $ m(y) = z$ since $ m$ has non-zero degree,
        see \cite[Property 1]{brezis1995degree}.  On account of
        $\Hd^0$ being the counting measure, this means that we have
        $\Hd^0(\{ m^{-1}(z)\}) \geq 1$ for almost all $z \in \Sph^2$.
        By symmetry of the sphere we get
	\begin{align}\label{something2}
	  \begin{split}
            \left| \dashint_{\Sph^2} z\, \Hd^0(\{m^{-1}(z)\}) \intd
              \Hd^2(z) \right| & =\left| \dashint_{\Sph^2} z\,
              \left(\Hd^0(\{m^{-1}(z)\}) -1\right) \intd \Hd^2(z)
            \right|\\
            & \leq \dashint_{\Sph^2} \left(\Hd^0(\{m^{-1}(z)\})
              -1\right) \intd \Hd^2(z).
	  \end{split}
	\end{align}
	Going back to $\R^2$ and exploiting that averaging leaves
        constant functions invariant, by
        \eqref{topological_bound_density} we obtain
	\begin{align}\label{something3}
	  \begin{split}
            \dashint_{\Sph^2} \left(\Hd^0(\{m^{-1}(z)\}) -1\right)
            \intd \Hd^2(z)
            &  = \frac{1}{4\pi} \int_{\R^2} |\widetilde m \cdot
            (\partial_1 \widetilde m \times \partial_2 \widetilde m) |
            \intd x - 1\\ 
            & \leq \frac{1}{8\pi} \int_{\R^2} |\nabla \widetilde m|^2
            \intd x - 1.
	  \end{split}
	\end{align}
        Straightforwardly concatenating the estimates
        \eqref{something1}, \eqref{something2}, \eqref{something3} and
        \eqref{perlen_vor_die_saeue}, we then obtain
	\begin{align}\label{average_vanishes_5}
          \left|\frac{1}{4\pi} \int_{\R^2} \widetilde m\, |\widetilde
          m \cdot(\partial_1 \widetilde m \times \partial_2 \widetilde
          m) | \intd x \right| \leq \frac{1}{8\pi} \int_{\R^2}
          |\nabla(\widetilde m-\Phi)|^2 \intd x. 
	\end{align}
	
	Adding the estimates \eqref{average_vanishes_3},
        \eqref{average_vanishes_4}, and \eqref{average_vanishes_5} we
        get
	\begin{align}
	  \begin{split}
            \left| \frac{1}{8\pi}\int_{\R^2} \widetilde m |\nabla \Phi
              |^2 \intd x \right| \leq \frac{1}{\sqrt{2\pi}}
            \left(\int_{\R^2} |\nabla ( \widetilde m- \Phi)|^2 \intd x
            \right)^\frac{1}{2}+ \frac{3}{8\pi}\int_{\R^2} |\nabla
            (\widetilde m-\Phi)|^2 \intd x.
	  \end{split}
	\end{align}
	Concatenating the identities \eqref{average_vanishes_1} and
        \eqref{average_vanishes_2}, and again applying Lemma
        \ref{lem:conformal_invariance}, we thus obtain
	\begin{align}
	  \begin{split}
            \left| \dashint_{\Sph^2} \left( m(\tilde y) - \tilde y
              \right) \intd \Hd^2(\tilde y) \right| & \leq
            \frac{1}{\sqrt{2\pi}} \left(\int_{\Sph^2} |\nabla (
              m(y)-y)|^2 \intd \Hd^2(y)\right)^\frac{1}{2} \\
            & \qquad \qquad \qquad \qquad \qquad +
            \frac{3}{8\pi}\int_{\Sph^2} |\nabla ( m(y)-y)|^2 \intd
            \Hd^2(y).
	  \end{split}
	\end{align}
	Under the assumption
        $\int_{\R^2} |\nabla (\widetilde m-\Phi)|^2 \intd x \leq
        \tilde\eta$ for $\tilde\eta >0$ small enough, we then get
	\begin{align}\label{average_estimated}
          \left| \dashint_{\Sph^2} \left( m(\tilde y)
          - \tilde y 
          \right) \intd  \Hd^2(\tilde y)  \right|  \leq
          \frac{1}{\sqrt{\pi}} \left(\int_{\Sph^2} |\nabla
          (m(y)-y)|^2 \intd y\right)^\frac{1}{2}, 
	\end{align}
	which by a density result of Schoen and Uhlenbeck
        \cite{schoen1983boundary} even holds for all
        $m \in H^1(\Sph^2;\Sph^2)$ with
        $\int_{\Sph^2} |\nabla m |^2 \intd x - 8\pi \leq \tilde \eta$.
        Together with inequality \eqref{sobolev}, this proves the
        desired Sobolev inequality.
	
	In order to obtain exponential integrability, we exploit the
        Moser-Trudinger inequality, \cite[Theorem 2]{moser1971sharp},
        which says
	\begin{align}
		\int_{\Sph^2} e^{4\pi \left| u \right|^2 } \intd \Hd^2 \leq C
	\end{align}
	for a universal constant $C>0$ and $u : \Sph^2 \to \R$
        satisfying $\int_{\Sph^2} |\nabla u |^2 \intd \Hd^2 \leq 1$
        and $\dashint_{\Sph^2} u \intd \Hd^2 = 0$.  To this end, for
        $i=1,2,3$ and $y \in \Sph^2$ we define, assuming without
          loss of generality that
          $m \not= \operatorname{id}_{\Sph^2}$:
	 \begin{align}
           u_i(y) :=\left(\int_{\Sph^2} |\nabla (m(\tilde y)-\tilde
           y)|^2 \intd \Hd^2(\tilde y)\right)^{-\frac{1}{2}} \left(
           \left(m (y) - y\right) - \dashint_{\Sph^2} \left( m(\tilde
           y) - \tilde y \right) \intd  \Hd^2(\tilde y) \right)_i. 
	 \end{align}
	 By H{\"o}lder's inequality and the Moser-Trudinger inequality
         we have 
	 \begin{align}
           \int_{\Sph^2} e^{\frac{4\pi}{3} |u|^2}
           \intd \Hd^2 = \int_{\Sph^2} \prod_{i=1}^3 e^{\frac{4\pi}{3} |u_i |^2}
           \intd \Hd^2  \leq \prod_{i=1}^3 \left(
           \int_{\Sph^2}  e^{4\pi |u_i |^2} \intd \Hd^2\right)^{1/3}
           \leq  C. 
	 \end{align}
	 At the same time, as a result of the inequality
         $(a+b)^2 \leq 2 (a^2 + b^2)$ for $a,b \in \R$, we have for
         all $y \in \Sph^2$ that
	 \begin{align}\label{mt1}
           2 |u_i(y) |^2 \geq \left(\int_{\Sph^2} |\nabla (m(y)-y)|^2
           \intd \Hd^2(y)\right)^{-1} \left( | m_i(y) - y_i|^2 - 2
           \left| \dashint_{\Sph^2} \left( m(\tilde y) - \tilde y
           \right)_i \intd  \Hd^2(\tilde y) \right|^2 \right), 
	 \end{align}
	 which by the estimate \eqref{average_estimated} gives
	  \begin{align}\label{mt2}
            2 |u_i(y) |^2 \geq \frac{|m_i(y)-y_i|^2}{\|\nabla (m
            -\operatorname{id}_{\Sph^2}) \|^2_2} - C. 
	 \end{align}
	 The two estimates \eqref{mt1} and \eqref{mt2} together imply
	 \begin{align}
           \int_{\Sph^2} e^{\frac{2\pi}{3}\frac{|m(y)-y|^2}{\|\nabla
           (m -\operatorname{id}_{\Sph^2}) \|^2_2}} \intd \Hd^2(y)
           \leq C, 
	 \end{align}
	 concluding the proof.
\end{proof}

\begin{proof}[Proof of Corollary \ref{cor:stability_conformal}]
  The statement for degree 1 maps immediately follows from Theorem
  \ref{thm:quantitative_stability} and Lemma
  \ref{lem:conformal_invariance}.  The estimate for degree $-1$ maps
  is a simple consequence of the fact that $m \in H^1(\Sph^2;\Sph^2)$
  has degree $-1$ if and only if $-m$ has degree 1.
\end{proof}

\section{Existence of minimizers}\label{sec:existence}
The goal of this section is to show that minimizers of the energy
$E_{\sigma,\lambda}$ over $\mathcal{A}$ exist in an appropriate range
of parameters by proving Theorem \ref{thm:existence}.

\subsection{First lower bounds}
Here, we describe the basic coercivity properties of
$E_{\sigma,\lambda}$.  To do so, we ensure that exchange and
anisotropy energies being finite forces either
$m-e_3 \in L^2(\R^2;\R^3)$ or $m+e_3 \in L^2(\R^2;\R^3)$, justifying
our choice of the latter as an assumption in the definition of
$\mathcal{A}$. Recall that $m = (m', m_3)$, where $m' = (m_1, m_2)$ is
the in-plane component of the magnetization vector.

\begin{lemma}\label{lemma:sobolev}
  Let $m\in \mathring H^1(\R^2;\Sph^2)$ be such that
  $m' \in L^2(\R^2;\R^2)$.  Then we have $m -e_3 \in L^2(\R^2;\R^3)$
  or $m+e_3 \in L^2(\R^2;\R^3)$ with the estimate
	\begin{align}\label{upordown}
          \min\left\{\int_{\R^2} |m_3-1|^2 \intd x, \int_{\R^2}
          |m_3+1|^2 \intd x  \right\} \leq \frac{1}{4\pi} \int_{\R^2}
          |\nabla m |^2\intd x  \int_{\R^2} (1-m_3^2) \intd x. 
	\end{align}
	In particular, for $m\in \mathcal{A}$ we have
	\begin{align}\label{controlina}
          \int_{\R^2} |m_3+1|^2 \intd x  \leq 4 \int_{\R^2} (1-m_3^2)
          \intd x. 
	\end{align}
\end{lemma}
Using this estimate, we are in a position to prove that
$E_{\sigma,\lambda}$ is bounded from below and controls both the
Dirichlet and the anisotropy energies.

\begin{lemma}\label{lemma:lower_bound}
  Let $\sigma>0$ and $\lambda \in [0,1]$.  Let
  $m\in \mathring H^1(\R^2;\Sph^2)$ with $m+e_3 \in L^2(\R^2;\R^3)$
  and $\int_{\R^2} |\nabla m|^2 \intd x < 16\pi$.  Then
  $E_{\sigma,\lambda}(m) < \infty$, and we have the lower bounds
	\begin{align}\label{apriori_dirichlet_excess}
          \left(1- \sigma^2\frac{(1+\lambda)^2}{4}\right)\int_{\R^2}
          |\nabla m|^2 \intd x & \leq E_{\sigma,\lambda}(m),\\ 
          \left(1- \sigma^2\frac{(1+\lambda)^2}{2}\right) \int_{\R^2}
          |\nabla m|^2 \intd x + \frac{\sigma^2}{2} \int_{\R^2}  | m'
          |^2 \intd x & \leq
                        E_{\sigma,\lambda}(m).\label{apriori_anisotropy} 
	\end{align}
	In particular, in the regime $\sigma (1+\lambda) \leq 2$ we
        have $E_{\sigma,\lambda}(m) \geq 0$ for all
        $m\in \mathcal{A}$.
\end{lemma}

\begin{proof}[Proof of Lemma \ref{lemma:sobolev}]
  \textit{Step 1: We have either $m_3-1 \in L^2(\R^2;\R^3)$ or $m_3 +
    1 \in L^2(\R^2;\R^3)$.}\\ 
  For almost all $x\in \R^2$ with $m_3(x) <1$ we have the inequality
  $\frac{|\nabla m_3|^2}{1-m_3^2} \leq |\nabla m|^2$, see for example
  \cite[(3.17)]{muratov2017domain} for the argument, which together
  with the fact that $|\nabla m_3|(x) = 0$ for almost all $x\in \R^2$
  with $|m_3(x)|=1$ and together with H\"older's inequality gives
	\begin{align}\label{modica-mortola}
	  \begin{split}
            \int_{\R^2} |\nabla m_3| \intd x = \int_{\{m_3<1\}}
            |\nabla m_3| \intd x \leq \left(\int_{\R^2} |\nabla m
              |^2\intd x\right)^\frac{1}{2}\left( \int_{\R^2}
              (1-m_3^2)\intd x \right)^\frac{1}{2}<\infty.
	  \end{split}
	\end{align}
 	The co-area formula, see for example \cite[Theorem
        3.40]{Ambrosio}, reads
 	\begin{align}
 		\int_{\R^2} |\nabla m_3| \intd x = \int_{-1}^{1} P(\{m_3 > t \}) \intd t.
 	\end{align}
 	Therefore, there exists $t \in (-\frac{1}{2},\frac{1}{2})$ such that
 	\begin{align}
 		 P(\{m_3 > t \}) \leq \int_{\R^2} |\nabla m_3| \intd x <\infty.
 	\end{align}
 	The isoperimetric inequality \cite[Theorem 3.46]{Ambrosio}
        then implies
 	\begin{align}
          \min\left\{ |\{m_3 > t \}| , |\{m_3 \leq t \} |\right\} \leq
          \frac{1}{4\pi}  P^2(\{m_3 > t \})<\infty. 
 	\end{align}
 	
 	In the following we only deal with the case
        $ |\{m_3\leq t \}| <\infty$, as the other case can be treated
        similarly.  Using $0\leq 1-m_3(x) \leq 2$ for all $x\in \R^2$,
        and $1+m_3(x) > 1+t \geq \frac{1}{2}$ for all
        $x\in \{m_3> t \}$, we have
 	\begin{align}\label{cant_think_of_better_label}
 	  \begin{split}
            \int_{\R^2} |m_3 -1|^2 \intd x & \leq 2\int_{\{m_3 > t
              \}} (1- m_3) \intd x + 4 |\{m_3 \leq t \}| \\
            & \leq 4 \int_{\{m_3 > t \}} (1+ m_3) (1-m_3) \intd x + 4
            |\{m_3\leq t \}|\\
            & \leq 4 \int_{\R^2} (1- m_3^2) \intd x + 4 |\{m_3\leq t
            \}| <\infty.
 	  \end{split}
 	\end{align}
 
        \textit{Step 2: Prove the quantitative estimates.}\\
	By \cite[Theorem 2]{fleming1960integral} the optimal
        Gagliardo-Nirenberg-Sobolev inequality is
	\begin{align}
          \| u \|_{2} \leq \frac{1}{2 \sqrt{\pi}} \| Du \|_{1}
	\end{align}
	for any $u \in L^r(\R^2)$ with $r\in [ 1,\infty)$.
	Combining this inequality with estimate \eqref{modica-mortola} yields
	\begin{align}
	  \begin{split}
            \min\left\{\int_{\R^2} |m_3-1|^2 \intd x, \int_{\R^2}
              |m_3+1|^2 \intd x  \right\} & \leq \frac{1}{4\pi} \left(
              \int_{\R^2} |\nabla m_3| \intd x \right)^2 \\ 
            & \leq \frac{1}{4\pi} \int_{\R^2} |\nabla m |^2\intd x
            \int_{\R^2} (1-m_3^2)\intd x,
	  \end{split}
	\end{align}
	which concludes the proof of estimate \eqref{upordown}.
        Finally, inequality \eqref{controlina} for $m \in \mathcal{A}$
        follows from $m+e_3 \in L^2(\R^2; \R^3)$ and
        $\int_{\R^2} |\nabla m|^2 \intd x < 16\pi$.
\end{proof}

\begin{proof}[Proof of Lemma \ref{lemma:lower_bound}]
  By Lemma \ref{lemma:fourier_basic}, we have
  $0 \leq F_{\mathrm{vol}}(m')< \infty $.  Furthermore, the estimate
  \eqref{surf_interpolation} gives
	\begin{align}
          F_{\mathrm{surf}}(m_3) \leq \frac{1}{2} \| m_3 + 1\|_{2} \|
          \nabla m_3 \|_{2}. 
	\end{align}
	Lemma \ref{lemma:sobolev} and the assumption
        $ \| \nabla m\|_{2} < 4 \sqrt{\pi}$ for all $m\in \mathcal{A}$
        then imply
	\begin{align}
          F_{\mathrm{surf}}(m_3) \leq \frac{1}{4 \sqrt{\pi}} \| m'
          \|_{2}\|\nabla m \|^2_{2} \leq \|m' \|_{2}\|\nabla m
          \|_{2}. 
	\end{align}
	To handle the DMI term, note that the Cauchy-Schwarz
        inequality gives
	\begin{align}
          2 \int_{\R^2} m' \cdot \nabla m_3 \intd x  \leq 2 \| m'
          \|_{2} \| \nabla m_3 \|_{2}. 
	\end{align}

	Combining these insights and applying Young's inequality we
        obtain
	\begin{align}
	  \begin{split}
            E_{\sigma,\lambda}(m) & \geq \int_{\R^2} \left( |\nabla
              m|^2 + \sigma^2 | m' |^2 \right) \intd x -
            \sigma^2(1+\lambda) \| m' \|_{2} \| \nabla m \|_{2}  \\
            & \geq \left(1-
              \sigma^2\frac{(1+\lambda)^2}{4}\right)\int_{\R^2}
            |\nabla m|^2 \intd x
	  \end{split}
	\end{align}
	and
	\begin{align}
	  \begin{split}
            E_{\sigma,\lambda}(m) \geq \int_{\R^2}\left(\left(1-
                \sigma^2\frac{(1+\lambda)^2}{2}\right) |\nabla m|^2 +
              \frac{\sigma^2}{2} | m' |^2 \right)\intd x ,
	  \end{split}
	\end{align}
	which gives the desired statements.
\end{proof}

\subsection{Upper bounds via minimization of a reduced
  energy}\label{sec:upper_bound}

We now turn to defining a simplified energy that reduces the
minimization to finding the best Belavin-Polyakov profile taking the
correct value at infinity.  As these profiles have logarithmically
divergent anisotropy energy, a truncation is necessary to make sense
of the energy.

Let
\begin{align}
	f(r) := \frac{2r}{1+r^2}
\end{align}
be the in-plane modulus of the N{\'e}el-type Belavin-Polyakov profile
\begin{align}
	\Phi(x) = \left( - \frac{2x}{1+|x|^2}, \frac{1-|x|^2}{1+ |x|^2} \right).
\end{align}
We consider the truncation at scale $L>1$ defined as
\begin{align}
  \Phi_{L}(x)&:= \left( - f_L(|x|) \frac{x}{|x|} ,
               \operatorname{sign}(1-|x|)\sqrt{ 1- f_L^2(|x|)}
               \right),\label{def_profiles}
\end{align}
where
\begin{align}\label{bp_truncation}
  f_{L}(r)&:= \begin{cases}
    f(r) & \text{ if } r\leq  L^\frac{1}{2},\\
    f\left( L^{\frac{1}{2}}\right) \frac{K_1\left( r
        L^{-1}\right)}{K_1\left( L^{-\frac{1}{2}} \right)}  & \text{
      if } r >  L^\frac{1}{2} .
  \end{cases}
\end{align}
Here, $K_1$ is the modified Bessel function of the second kind of
order $1$, for a more detailed discussion see Section
\ref{subsec:bessel}.  The ans{\"a}tze are then given by
\begin{align}\label{ansatz}
	 \phi_{\rho,\theta,L}(x):=S_\theta \Phi_{L}(\rho^{-1} x)
\end{align}
for $\rho >0$, $\theta \in [-\pi,\pi)$, $L>1$ and where $S_\theta$ is
given by \eqref{def_rotation}.  For convenience, we also define
$\phi_{\rho,\theta,\infty}(x) := S_\theta \Phi(\rho^{-1}x)$.

For later computations, it turns out to be convenient to not quite use
the variables $\rho$ and $L$ in the definition of the reduced energy,
but to make the substitutions
\begin{align}\label{scale_of_cutoff_to_excess}
  \widetilde \rho = |\log \sigma | \rho, \qquad  \widetilde L =
  \frac{L}{2 \sqrt{\pi}}. 
\end{align}
We furthermore divide the energy by $\frac{\sigma^2}{|\log \sigma|}$.
As we will show below, for $\sigma >0$, $\lambda \in [0,1]$ and a
constant $K>0$ the rescaled energy of $\phi_{\rho,\theta,L}$ is then
given to the leading order in $\sigma \ll 1$ by
\begin{align}\label{defreducedenergy}
  \begin{split}
    \mathcal{E}_{\sigma,\lambda;K} \left(\widetilde \rho,\theta,
      \widetilde L\right) & := |\log\sigma| \left(\sigma \widetilde
      L\right)^{-2} + \frac{4\pi \log\left(K \widetilde
        L^2\right)}{|\log\sigma| } \, \widetilde \rho^2 -
    g(\lambda,\theta) \widetilde \rho,
  \end{split}
\end{align}
on the domain 
\begin{align}\label{defVsigma}
  V_{\sigma}:=\left\{\left(\widetilde \rho,\theta, \widetilde L\right)
  :  \widetilde \rho >0  , \, \theta \in [-\pi,\pi), \,  \widetilde L
  \geq \frac{1}{4 \sigma \sqrt{\pi}} \right\}, 
\end{align}
where
\begin{align}
  g(\lambda,\theta):=  8\pi \lambda \cos\theta + \frac{\pi^3}{8}
  (1-\lambda) \left(1 - 3 \cos^2\theta  \right). 
\end{align}
The first term on the right-hand side of definition
\eqref{defreducedenergy} represents the Dirichlet excess.  The
constant $4\sqrt{\pi}$ in the bound for $\tilde L$ in the definition
of $V_\sigma$ is the result of the a priori estimate \eqref{apriori_L}
for the Dirichlet excess arising in the proof of the lower bound in
Section \ref{section:conformal}.  The second term captures the
logarithmic blowup of the anisotropy energy as the profile approaches
a Belavin-Polyakov profile.  Finally, the third term combines the
contributions of the DMI and stray field terms.

The details of the truncation \eqref{bp_truncation} will only enter
through the constant $K>0$, with our construction giving $K = K^*$,
where
\begin{align}\label{K*}
  K^* := \frac{16\pi}{e^{2(1+\gamma)}},
\end{align}
in which $\gamma \approx 0.5772$ is the Euler-Mascheroni constant.  As
this constant matches the one obtained in Lemma
\ref{lem:lower_bound_anisotropy_pinned} below, we do expect our ansatz
to be optimal.  However, we will lose a constant factor of
$\frac{2}{3}-\eps$ for $\eps>0$ sufficiently small in the application
of the rigidity Theorem \ref{thm:quantitative_stability}, so that we
get different energies $\mathcal{E}_{\sigma,\lambda;K^*}$ and
$\mathcal{E}_{\sigma,\lambda;(\frac32+\eps)^{-1} K^*}$ appearing in
the upper and lower bounds for $\min E_{\sigma,\lambda}$.
Nevertheless, we will see that the stability properties of the two
reduced energies are strong enough to prove Theorem
\ref{thm:convergence}.

The following lemma contains an estimate comparing
$E_{\sigma,\lambda}(\phi_{\rho,\theta,L})$ with
$\mathcal{E}_{\sigma,\lambda;K}\left(\widetilde \rho,
  \theta,\widetilde L\right)$.  The rate
$\smash{\sigma^\frac{1}{2}|\log\sigma|}$ is likely not optimal, but it
is sufficient for our argument.  Additionally, we keep track of a
number of identities which will be useful later.

\begin{lemma}\label{lem:approximation_physical}
  There exist universal constants $C>0$ and $\sigma_0>0$ such that for
  all $\sigma \in (0,\sigma_0)$ and for all $\lambda \in [0,1]$ we
  have the following: For $\rho \in (0,1]$, $\theta \in [-\pi,\pi)$
  and $L\geq \frac{1}{2\sigma}$ we have
  $\phi_{\rho,\theta,L} \in \mathcal{A}$ and
  $\left( |\log \sigma| \rho, \theta, \frac{L}{2 \sqrt{\pi}} \right)
  \in V_{\sigma}$.  Furthermore, it holds that
	\begin{align}\label{approximation_physical}
          \left|\frac{|\log \sigma|}{\sigma^2} \left(
          E_{\sigma,\lambda}(\phi_{\rho,\theta,L}) - 8\pi \right) -
          \mathcal{E}_{\sigma,\lambda;K^*} \left(|\log \sigma |
          \rho,\theta,  \frac{L}{2 \sqrt{\pi}} \right) \right|
          \leq C\sigma^{\frac{1}{4}}  |\log\sigma|, 
	\end{align}
	where $K^*$ is defined in equation \eqref{K*}.  Additionally,
        for any $\rho>0$ we have
	\begin{align}
          \int_{\R^2} |\nabla \phi_{\rho,\theta,L}|^2 \intd x - 8\pi
          &
            \leq \frac{4\pi}{L^2}+ \frac{C \log^2 L }{L^3},\label{BP_excess}\\ 
          \int_{\R^2} \left|\nabla \left( \phi_{\rho, \theta,\infty} -
          \phi_{\rho,\theta,L} \right)\right|^2 \intd x
          & \leq C
            L^{-2}, \label{BP_difference}\\ 
          \int_{\R^2} |\phi'_{\rho,\theta,L}|^2 \intd x
          & \leq 4\pi\rho^2 \log \left( \frac{4 L^2}{e^{2(1+\gamma)}}
            \right) + \frac{C\rho^2 \log^2
            L}{L}, \label{BP_anisotropy}\\ 
          \int_{\R^2} |\phi_{\rho, \theta, \infty;3} - \phi_{\rho,
          \theta, L;3}|^2 \intd x 
          & \leq \frac{C\rho^2}{L}, \label{BP_outofplane}\\ 
          \int_{\R^2} 2 \phi'_{\rho,\theta,\infty}\cdot \nabla
          \phi_{\rho,\theta,\infty;3} \intd x
          & = 8\pi \rho \cos \theta ,\label{BP_DMI_infty}\\
          \int_{\R^2} 2 \phi'_{\rho,\theta,L}\cdot \nabla
          \phi_{\rho,\theta,L;3} \intd x
          &= 8\pi \rho \cos\theta + O\left( \rho L^{-\frac{1}{2}}
            \right),\label{BP_DMI}\\ 
          F_{\mathrm{vol}}\left( \phi'_{\rho,\theta,L}\right) -
          F_{\mathrm{surf}}\left(\phi_{\rho,\theta,L;3} \right)
          & =
            \left(\frac{3\pi^3}{8}\cos^2\theta
            -
            \frac{\pi^3}{8}
            \right)
            \rho
            +O\left(
            \rho
            L^{-\frac{1}{4}}\right).\label{BP_stray_field} 
	\end{align}
      \end{lemma}

      Having thus established the correspondence between
      $E_{\sigma,\lambda}$ and $\mathcal{E} _{\sigma,\lambda;K}$, we
      can carry it out explicitly for the first part of the following
      statement.  The stability properties of
      $\mathcal{E}_{\sigma,\lambda;K}$ are collected in the second
      part, which will yield convergence of the skyrmion radius and
      angle in Section \ref{section:conformal}.

\begin{prop}\label{prop:reduced_minimization}
  There exists a universal constant $\sigma_0 >0$ such that for all
  $\sigma \in (0,\sigma_0)$, $\lambda \in [0,1]$ and
  $K \in [ \frac{1}{2}K^*, 2 K^* ]$ we have the following:
	
\begin{enumerate}[(i)]
\item The function $\mathcal{E}_{\sigma,\lambda;K}$ has at most two
  global minimizers $(\rho_0, \theta_0^\pm, L_0)$ over $V_\sigma $ and
  no further critical points in $V_\sigma$.  Recalling the definitions
  \eqref{epsilon_definition} and \eqref{lamc}, the minimizers are
  given by
	\begin{align}
          \rho_0 & = \frac{\bar g(\lambda)}{16\pi} +
                   O\left(\frac{\log|\log\sigma|}{|\log \sigma|}
                   \right),\\ 
          \theta_0^\pm & = \begin{cases}
            0 & \text{ if } \lambda \geq \lambda_{c},\\ 
            \pm \arccos\left(\frac{32 \lambda}{3\pi^2 (1-\lambda)}
            \right)  & \text{ else,} 
          \end{cases}\\
          L_0 & =\left(\frac{8 \sqrt{\pi} }{\bar g(\lambda)} +
                O\left(\frac{\log|\log\sigma|}{|\log\sigma|}\right)
                \right)\frac{|\log\sigma|}{\sigma}.
	\end{align}
	Furthermore, we have
	\begin{align}\label{asymptotics_minimal_energy}
          \min_{V_\sigma} \mathcal{E}_{\sigma,\lambda;K} = -
          \frac{\bar g^2(\lambda)}{32\pi} +\frac{\bar
          g^2(\lambda)}{32\pi}\frac{\log|\log\sigma|}{|\log \sigma|}
          - \frac{\bar g^2(\lambda)}{64\pi}\frac{ \log\left(\frac{\bar
          g^2(\lambda)}{64 \pi e K}\right)}{|\log\sigma|} +
          O\left(\frac{\log^2|\log \sigma| }{|\log \sigma
          |^2}\right)
	\end{align}
	and
	\begin{align}\label{theta_0_maximizer}
		g(\lambda,\theta_0^\pm) = \bar g(\lambda).
	\end{align}
      \item Let $(\rho_\sigma,\theta_\sigma,L_\sigma) \in V_\sigma$ be
        such that
        \begin{align}
          \label{reduced_stability_assumption}
          \mathcal{E}_{\sigma,\lambda;K}
          (\rho_\sigma,\theta_\sigma,L_\sigma)  \leq \min_{V_\sigma}
          \mathcal{E}_{\sigma,\lambda;K} + \frac{\bar
          g^2(\lambda)}{64\pi |\log \sigma |}. 
        \end{align}
        \begin{align}
          \label{limrhosthetas}
          \lim_{\sigma \to 0}
          \rho_\sigma = 
          \frac{\bar g(\lambda)}{16 \pi}, \qquad \qquad 
          \lim_{\sigma \to 0} |\theta_\sigma| & = \theta_0^+,
        \end{align}
        and there exists a universal constant $C>0$ such that
        \begin{align}\label{stability_L}
          \frac{1}{C}\frac{|\log\sigma|}{\sigma}\leq L_\sigma
          \leq C \frac{|\log\sigma|}{\sigma}. 
	\end{align}
\end{enumerate}
\end{prop}

\begin{remark}
  We point out that it is possible to show the rates
	\begin{align}
          \left( \rho_\sigma - \frac{\bar g(\lambda)}{16\pi}
          \right)^2 \leq \frac{C}{|\log \sigma|}, \qquad \left|
          |\theta_\sigma| -\theta_0^+\right|^4 + 
          \left|\lambda-\lambda_{c} \right| \left| |\theta_\sigma|
          -\theta_0^+ \right|^2 \leq \frac{C}{|\log\sigma|}
	\end{align}
        for some $C > 0$ universal in the setting of the second part of
        Proposition \ref{prop:reduced_minimization}.  However, as we
        do not attempt to capture the sharp rates and as the proof is
        a somewhat lengthy calculus exercise, we will not reproduce it
        here.
\end{remark}

Finally, we present two corollaries to these bounds.  The first simply
translates the minimal energy for $\mathcal{E}_{\sigma,\lambda;K}$
into an upper bound for the minimal value of $E_{\sigma,\lambda}$.

\begin{cor}\label{cor:ansatz}
  There exist universal constants $\sigma_0>0$ and $C>0$ such that for
  all $\sigma \in (0,\sigma_0)$ and $\lambda \in [0,1]$ we have
\begin{align}
  \frac{|\log\sigma|}{\sigma^2}\left( \inf_{ \mathcal{A}}
  E_{\sigma,\lambda}- 8\pi\right)  \leq \min_{V_\sigma}
  \mathcal{E}_{\sigma,\lambda;K^*} +C\sigma^{\frac{1}{4}} |\log\sigma|
  . 
\end{align}
\end{cor}

While Corollary \ref{cor:ansatz} is concerned with asymptotically
precise minimization, the existence of minimizers relies on an upper
bound by $8\pi$ for general $\sigma>0$ (see also
\cite{melcher2014chiral}). 

\begin{lemma}\label{cor:upper_bound}
	For $\sigma>0$ and $\lambda \in [0,1]$ we have
	\begin{align}
          \inf_{\mathcal{A}} E_{\sigma,\lambda} <  8\pi.
	\end{align}
\end{lemma}

\begin{proof}[Proof of Lemma \ref{lem:approximation_physical}]
  The computations for $\Phi_L$ can be found in Lemma
  \ref{lem:bessel_calculation}, so that we only have to translate them
  to $\phi_{\rho,\theta,L}$ here.  By assumption, we have
  $L\geq \frac{1}{2 \sigma}$, so that we can indeed apply Lemma
  \ref{lem:bessel_calculation} for $\sigma \in (0,\sigma_0)$ with
  $\sigma_0>0$ small enough.
	
  Scale invariance of the Dirichlet energy allows to translate the
  bounds \eqref{stereo_excess} and \eqref{stereo_homogeneous_h1} into
  the bounds \eqref{BP_excess} and \eqref{BP_difference}, as well as
  to obtain the bound
	\begin{align}\label{interior_of_constraint}
		\int_{\R^2} | \nabla \phi_{\rho,\sigma,L}|^2 \intd x <  16 \pi.
	\end{align}
	The fact that $\mathcal{N}(\phi_{\rho,\sigma,L}) =1$ follows
        from $\mathcal{N}(\Phi_L) = 1$ and scale and rotation
        invariance of $\mathcal{N}$.
	
	The bound \eqref{BP_anisotropy} follows directly from the
        estimate \eqref{stereo_anisotropy} via rescaling.
        Similarly, we get the bound \eqref{BP_outofplane} from the
        bound \eqref{stereo_out_of_plane}, as well as
        $\phi_{\rho,\theta,L} +e_3 \in L^2(\R^2;\R^3)$ from
        $\Phi_L + e_3 \in L^2(\R^2;\R^2)$ on account of
        $\Phi_L \in \mathcal{A}$.  Together with
        $\mathcal{N}(\phi_{\rho,\theta,L}) =1$ and the estimate
        \eqref{interior_of_constraint} we therefore obtain
        $\phi_{\rho,\theta,L} \in \mathcal{A}$ for all
          $\sigma \in (0, \sigma_0)$ with $\sigma_0$ small enough
          universal.
	
          For $x= (x_1,x_2) \in \R^2$ we define
          $x^\perp := (-x_2,x_1)$.  The fact that
          $ (\Phi'_L)^\perp \cdot \nabla \Phi_{L,3} = 0$ everywhere
          and the identity \eqref{stereo_dmi} allow us to calculate
          the DMI term to be
 	\begin{align}
 	  \begin{split}
            \int_{\R^2} 2 \phi'_{\rho,\theta,L} \cdot \nabla
            \phi_{\rho,\theta,L,3} \intd x = 2 \rho \cos\theta
            \int_{\R^2} \Phi'_L \cdot \nabla \Phi_{L,3} \intd x = 8\pi
            \rho \cos \theta + O\left(\rho L^{-\frac{1}{2}}\right),
 	  \end{split}
 	\end{align}
 	taking care of estimate \eqref{BP_DMI}.  The same argument
        using the identity \eqref{stereo_dmi_exact} instead of the
        identity \eqref{stereo_dmi} gives the identity
        \eqref{BP_DMI_infty}.
 	
 	Due to
        $\nabla \cdot \phi'_{\rho,\theta,L} = \cos \theta \nabla \cdot
        \phi'_{\rho,0,L}$ and equation \eqref{stereo_vol}, the
        contribution of the volume charges is
 	\begin{align}
          F_{\mathrm{vol}}(\phi'_{\rho,\theta,L})= \rho
          F_{\mathrm{vol}}(\Phi'_L) \cos^2 \theta = \frac{3\pi^3}{8}
          \rho \cos^2 \theta  + O\left(\rho L^{-\frac{1}{4}} \right). 
 	\end{align}
 	As a result of the identity \eqref{stereo_surf}, the surface
        charges are simply given by
 	\begin{align}
          F_{\mathrm{surf}}(\phi_{\rho,\theta,L,3}) = \rho
          F_{\mathrm{surf}}(\Phi_{L,3}) =  \frac{\pi^3}{8}\rho  +
          O\left(\rho L^{-\frac{1}{2}} \right). 
 	\end{align}
 	Combined, these two estimates give the identity
        \eqref{BP_stray_field}.
 	
 	Taking everything together and recalling
          \eqref{defreducedenergy}, we obtain
 	\begin{align}
 	  \begin{split}
            E_{\sigma,\lambda}(\phi_{\rho,\theta,L}) &  = 8\pi +
            \frac{4\pi}{L^2} + 4\pi \sigma^2 \rho^2 \log \left(
              \frac{4 L^2}{e^{2(1+\gamma)}} \right)   - \sigma^2
            \rho\, g(\lambda,\theta) \\ 
            & \quad +O\left(\frac{\log^2 L}{L^3}\right) +
            O\left(\frac{\sigma^2 \rho^2 \log^2 L}{L}\right) +
            O\left(\lambda \sigma^2 \rho L^{-\frac{1}{2}} \right) +
            O\left((1-\lambda)\sigma^2 \rho L^{-\frac{1}{4}}\right),
 	  \end{split}
 	\end{align}
 	which for a given $(\rho,\theta,L) \in V_{\sigma}$ and
        $\rho\leq 1$ translates into the estimate
        \eqref{approximation_physical}.
\end{proof}

\begin{proof}[Proof of Proposition \ref{prop:reduced_minimization}]
  \textit{Step 1: Minimization in $\theta$.}\\
  We define
  $\Delta (\lambda,\theta):= \bar g (\lambda)-g(\lambda,\theta)$
  and for $\lambda < \lambda_{c} = \frac{3\pi^2}{32+3\pi^2}$ calculate
 	\begin{align}\label{defdeltag1}
          \Delta (\lambda,\theta) =   \frac{3\pi^
          3}{8}(1-\lambda)\left(\cos\theta -
          \frac{32\lambda}{3\pi^2(1-\lambda)} \right)^2. 
 	\end{align}
 	For $\lambda \geq \lambda_{c}$ we instead have
 	\begin{align}\label{defdeltag2}
 	  \begin{split}
            \Delta (\lambda,\theta) & =\frac{3\pi^3}{4} \left(
              \frac{\lambda}{\lambda_{c}} - 1\right)(1-\cos\theta) +
            \frac{3\pi^3}{8}(1-\lambda) (1-\cos\theta)^2.
 	  \end{split}
 	\end{align}
 	
        By inspection we have $\Delta (\lambda,\theta) \geq 0$
        for all $\lambda \in [0, 1]$ and
          $\theta \in [-\pi, \pi)$, and $\Delta (\lambda,\theta)=0$
        if and only if $\theta=\theta_0^\pm$, where
        $\theta_0^\pm$ is given by \eqref{optimal_angle}.
        In particular, we
        get the identity \eqref{theta_0_maximizer} and the fact that
        $-\bar g(\lambda)$ is the minimal value of
        $-g(\lambda,\theta)$ which is achieved at the two minima
        $\theta =  \theta_0^\pm$.  Therefore, we have for all
        $\sigma \in (0,\sigma_0)$, $\lambda \in [0,1]$, $\rho >0$, and
        $L>\frac{1}{4\sigma\sqrt{\pi}}$ that
 	\begin{align}
 	  \begin{split}
            \mathcal{E}_{\sigma,\lambda;K}(\rho,\theta,L) & =
            |\log\sigma| \left(\sigma L \right)^{-2} + \frac{4\pi
              \log\left(K L^2 \right)}{|\log \sigma|} \rho^2 - \bar
            g(\lambda) \rho + \Delta (\lambda,\theta) \rho.
 	  \end{split}
 	\end{align}

        \textit{Step 2: Minimization in $\rho$.}\\
	We observe that
	\begin{align}
          \frac{4\pi\log(KL^2)}{|\log \sigma|} \left( \rho -
          \frac{\bar
          g(\lambda) |\log \sigma| }{8\pi \log(K L^2)} \right)^2  =
          \frac{4\pi\log\left(K L^2 
          \right)}{|\log \sigma|}  \rho^2 - \bar g(\lambda) \rho +
          \frac{\bar g^2(\lambda)|\log \sigma|}{16 \pi \log(KL^2)}. 
	\end{align}
	Consequently, the quantity
	\begin{align}\label{minimizerrho}
          \rho_0(L) := \frac{\bar
          g(\lambda) |\log \sigma|}{8 \pi \log(K L^2)} 
	\end{align}
	minimizes the map
	$
		\rho \mapsto \mathcal{E}_{\sigma,\lambda;K}(\rho,\theta_0,L),
	$
	and we have the identity
	\begin{align}\label{stability_representation}
 	  \begin{split}
            \mathcal{E}_{\sigma,\lambda;K}(\rho,\theta,L) & =
            |\log\sigma| \left(\sigma L \right)^{-2} - \frac{\bar
              g^2(\lambda)|\log \sigma|}{16 \pi \log(KL^2)} \\
            & \quad \quad + \Delta (\lambda,\theta) \rho + 4\pi
            \frac{\log(KL^2)}{|\log \sigma|} \left( \rho - \frac{\bar
                g(\lambda)|\log \sigma|}{8\pi\log(K L^2)} \right)^2.
 	  \end{split}
 	\end{align}
 	
        \textit{Step 3: Minimization in $L$.}\\
	We make the substitution $t= K^{-1}L^{-2}$ and minimize
	\begin{align}\label{deff}
          f(t):= K \frac{|\log\sigma|}{\sigma^2} t +  \frac{\bar
          g^2(\lambda) |\log \sigma|}{16\pi \log t} 
	\end{align}
	in $0<t\leq \frac{16\pi}{K}\sigma^2<1$ for $\sigma \in
        (0,\sigma_0)$ with $\sigma_0>0$ small enough universal, in
          view of the assumption on $K$.  Since $f(t) \nearrow 0$ as
        $t \searrow 0$, the function attains its minimum over this
        interval.  We also observe that
        $\lim_{\sigma \to 0} f\left(\frac{16\pi}{K} \sigma^2 \right) =
        \infty$, so that the minimum is achieved for
        $0< t< \frac{16\pi}{K}\sigma^2<1$ for $\sigma_0>0$
          small enough.
	
	We calculate
	\begin{align}\label{f_derivative}
          f'(t) = K \frac{|\log \sigma|}{\sigma^2} - \frac{\bar
          g^2(\lambda) |\log\sigma|}{16\pi t \log^2 t} 
	\end{align}
	and note that $0<t_0 < \frac{16\pi}{K}\sigma^2<1$ solves
        $f'(t_0) = 0$ if and only if
	\begin{align}
          - t_0^\frac{1}{2}\log\left(t_0^\frac{1}{2}\right) =
          \frac{\bar g(\lambda) \sigma}{8 \pi^\frac{1}{2}
          K^\frac{1}{2}} 
	\end{align}
	 for $\sigma_0$
        small enough.  In turn, for
        $s_0:= \log\left(t_0^\frac{1}{2} \right)$ this equation is
        equivalent to
	\begin{align}
          s_0 e^{s_0} = -\frac{\bar g(\lambda) \sigma}{8
          \pi^\frac{1}{2} K^\frac{1}{2}}. 
	\end{align}
	Solutions to this equation only exist provided
        $\frac{\bar g(\lambda) \sigma}{8 \pi^\frac{1}{2}
          K^\frac{1}{2}}\leq e^{-1}$, which is the case for
        $\sigma \in (0,\sigma_0)$ with $\sigma_0>0$ small enough.
        Under this condition there are precisely two solutions
        given by
	\begin{align}
          s_{0,i} = W_{i}\left(-\frac{\bar g(\lambda) \sigma}{8
          \pi^\frac{1}{2} K^\frac{1}{2}} \right) 
	\end{align}
	for $i=-1$ and $i=0$, where $W_i$ are the two real-valued
        branches of the Lambert $W$-function, see Corless et al.\
        \cite{corless1996lambertw}. In terms of $t_0$, these are
	\begin{align}\label{def_t}
          t_{0,i} := \exp\left(2 W_{i}\left(-\frac{\bar g(\lambda)
          \sigma}{8 \pi^\frac{1}{2} K^\frac{1}{2}} \right)\right). 
	\end{align}
	
	As $W_0$ is smooth at $0$, with $W_0(0)=0$, we have
	\begin{align}
		t_{0,0} = 1 + O(\sigma) > \frac{16\pi}{K}\sigma^2
	\end{align}
	for $\sigma_0$
        sufficiently small, so that this solution is irrelevant to us.
        The point $t_{0,-1}$ being the only other critical point
        of $f$, we get that it indeed is the minimizer over
        $0<t < \frac{16\pi}{K}\sigma^2$ for $\sigma \in (0,\sigma_0)$
        with $\sigma_0>0$ small enough.  Consequently, the minimum is
        taken at $t_0:= t_{0,-1}$, and exploiting the identity
        $f'(t_{0})t_{0} =0$ the minimal value can be seen to be
	\begin{align}\label{exact_minimization}
	  \begin{split}
		\min_{V_\sigma} \mathcal{E}_{\sigma,\theta;K} & =
		f(t_{0}) \\
		& = \frac{\bar g^2(\lambda) |\log \sigma|}{64\pi}
                \left(W_{-1}^{-2}\left(-\frac{\bar g(\lambda)
                      \sigma}{8 \pi^\frac{1}{2} K^\frac{1}{2}} \right)
                  + 2 W_{-1}^{-1}\left(-\frac{\bar g(\lambda)
                      \sigma}{8 \pi^\frac{1}{2} K^\frac{1}{2}} \right)
                \right).
	  \end{split}
	\end{align}
	
	In order to determine the behavior of the minimal energy as
        $\sigma \to 0$, for $-e^{-1} <s<0$ with $|s| \ll1$ we
        refer to the expansion
	\begin{align}\label{lambert_expansion}
          W_{-1}(s) = \log(-s) - \log|\log(-s)| +
          O\left(\frac{\log|\log(-s)|}{|\log(-s)|} \right), 
	\end{align}
	see Corless et al.\ \cite[equation (4.19), as well as the
        discussion following equation (4.20)]{corless1996lambertw}.
        Combined with
        $\log\left( \frac{\bar g(\lambda) \sigma}{8 \pi^\frac{1}{2}
            K^\frac{1}{2}}\right) = \log\sigma + \log\left(\frac{\bar
            g(\lambda)}{8 \pi^\frac{1}{2} K^\frac{1}{2}}\right)$ and
        \eqref{gcC}, this gives
	\begin{align}
	  \begin{split}
            W_{-1}\left(-\frac{\bar g(\lambda) \sigma}{8
                \pi^\frac{1}{2} K^\frac{1}{2}}  \right) & = \log\sigma
            + \log\left(\frac{\bar g(\lambda)}{8 \pi^\frac{1}{2}
                K^\frac{1}{2}}\right) -\log|\log \sigma| - \log
            \left(1 + \frac{\log\left(\frac{\bar g(\lambda)}{8
                    \pi^\frac{1}{2} K^\frac{1}{2}}\right)}{\log\sigma}
            \right)\\ 
            & \quad + O\left(\frac{\log|\log\sigma|}{|\log \sigma|}
            \right)\\ 
            & = -|\log\sigma| -\log|\log \sigma| +
            \log\left(\frac{\bar g(\lambda)}{8 \pi^\frac{1}{2}
                K^\frac{1}{2}}\right) +
            O\left(\frac{\log|\log\sigma|}{|\log\sigma |} \right)
	  \end{split}
	\end{align}
	for $\sigma_0$
        sufficiently small.  With $\frac{1}{1+r} = 1-r + O(r^2)$ for
        $r\ll1$ we obtain
	\begin{align}
          W_{-1}^{-1}\left(-\frac{\bar g(\lambda) \sigma}{8 \pi^\frac{1}{2}
          K^\frac{1}{2}}  \right)
          & = - \frac{1}{|\log\sigma|} +
            \frac{\log|\log
            \sigma|}{|\log \sigma|^2} -
            \frac{ \log\left(\frac{\bar
            g(\lambda)}{8 \pi^\frac{1}{2}
            K^\frac{1}{2}}\right)}{|\log\sigma|^2}
            +O\left(
            \frac{\log^2|\log\sigma|}{|\log\sigma
            |^3}\right),\\ 
          W_{-1}^{-2}\left(-\frac{\bar g(\lambda) \sigma}{8
          \pi^\frac{1}{2} K^\frac{1}{2}}  \right)
          & = \frac{1}{|\log
            \sigma|^2}  +
            O\left(\frac{\log|\log
            \sigma|}{|\log
            \sigma
            |^3}\right). 
	\end{align}
	As a result, we have
	\begin{align}
          \min_{V_\sigma} \mathcal{E}_{\sigma,\theta;K} = - \frac{\bar
          g^2(\lambda)}{32\pi} +\frac{\bar
          g^2(\lambda)}{32\pi}\frac{\log|\log\sigma|}{|\log \sigma|}
          - \frac{\bar g^2(\lambda)}{64\pi}\frac{ \log\left(\frac{\bar
          g^2(\lambda)}{64 \pi e K}\right)}{|\log\sigma|} +
          O\left(\frac{\log^2|\log \sigma |}{|\log \sigma
          |^2}\right). 
	\end{align}
	Recalling the relation $t = K^{-1} L^{-2}$, the definition
        \eqref{def_t}, and using the identity
        $W_{-1}(s)e^{W_{-1}(s)}= s$ for $s \in (-e^{-1},0)$, we
        also get that the optimal truncation scale is
	\begin{align}\label{logt0L0}
          L_0: = K^{-\frac{1}{2}} t_{0}^{-\frac{1}{2}} = -
          W_{-1}\left(-\frac{\bar g(\lambda) \sigma}{8 \pi^\frac{1}{2}
          K^\frac{1}{2}} \right)\frac{8 \pi^\frac{1}{2}}{\bar
          g(\lambda) \sigma} = \frac{8 \pi^\frac{1}{2} }{\bar
          g(\lambda)} \frac{|\log\sigma|}{\sigma} +
          O\left(\frac{\log|\log\sigma|}{\sigma}\right). 
	\end{align}
	Finally, recalling the definition \eqref{minimizerrho} the
        optimal skyrmion radius is given by
	\begin{align}
          \rho_0(L_0) = \frac{\bar g(\lambda)}{16\pi} +
          O\left(\frac{\log|\log\sigma|}{|\log \sigma|}  \right). 
	\end{align}
	This concludes the proof of the first part of Proposition
        \ref{prop:reduced_minimization}.

        \textit{Step 4: Proof of stability for $L_\sigma$.}\\
        Let now $(\rho_\sigma,\theta_\sigma,L_\sigma) \in V_{\sigma}$
        be such that \eqref{reduced_stability_assumption} holds.
        Let $t_\sigma := K^{-1}L_\sigma^{-2}$ and note that by
        \eqref{defVsigma} we have
        $t_\sigma \in \left(0,\frac{16\pi}{K}\sigma^2\right)$, so that
        $f(t_\sigma)$ is defined.  The fact that
        $\min_{V_\sigma} \mathcal{E}_{\sigma,\lambda;K}= f(t_{0})$,
        the representation \eqref{stability_representation}, and the
        assumption \eqref{reduced_stability_assumption} imply
        that
	\begin{align}\label{ftofsigma}
          f(t_\sigma) - f(t_{0}) \leq  \mathcal{E}_{\sigma,\lambda;K}
          (\rho_\sigma,\theta_\sigma,L_\sigma) - \min_{V_\sigma}
          \mathcal{E}_{\sigma,\lambda;K} \leq  \frac{\bar
          g^2(\lambda)}{64\pi |\log \sigma|}. 
	\end{align}
	
	For $s_\sigma:= \frac{t_\sigma}{t_0}$ and
        $\widetilde f_{\sigma}(s):= \frac{64\pi |\log\sigma|}{\bar
          g^2(\lambda)}(f(s t_0) - f(t_0))$ with
        $s \in \left( 0 ,\frac{16\pi \sigma^2}{K t_0}\right)$ this
        translates to
	\begin{align}\label{boundbym}
		\widetilde f_{\sigma}(s_\sigma) \leq 1.
	\end{align}
	In order to explicitly compute $\widetilde f_{\sigma}$, we use
        the definitions \eqref{deff} and \eqref{def_t} together with
        the fact that for all $0 < \tilde s < e^{-1}$ we have
        $W_{-1}(-\tilde s) e^{W_{-1}(-\tilde s)} = -\tilde s$,
        obtaining
	\begin{align}\label{tildef}
	  \begin{split}
            \widetilde f_{\sigma}(s) & = \frac{64\pi K t_0 |\log
              \sigma|^2}{\bar g^2(\lambda) \sigma^2} (s-1) + 4 |\log
            \sigma |^2 \left( \frac{1}{\log(s t_0)} - \frac{1}{\log
                t_0} \right)\\
            & =\frac{|\log \sigma|^2}{W_{-1}^2\left(-\frac{\bar
                  g(\lambda) \sigma}{8\pi^\frac{1}{2}K^\frac{1}{2}}
              \right)} (s-1) - \frac{2|\log \sigma|^2}{\left(2
                W_{-1}\left( -\frac{\bar g(\lambda)
                    \sigma}{8\pi^\frac{1}{2}K^\frac{1}{2}}\right) +
                \log s\right) W_{-1}\left( -\frac{\bar g(\lambda)
                  \sigma}{8\pi^\frac{1}{2}K^\frac{1}{2}}\right)} \log
            s\\
            & = \frac{|\log \sigma|^2} {W_{-1}^2\left(-\frac{\bar
                  g(\lambda)\sigma}{8\pi^\frac{1}{2}K^\frac{1}{2}}\right)}
            \left(s-1 - \frac{2 \left| W_{-1}\left( -\frac{\bar
                      g(\lambda) \sigma}
                    {8\pi^\frac{1}{2}K^\frac{1}{2}} \right) \right|
                \log s} { 2\left|W_{-1} \left( -\frac{\bar
                      g(\lambda)\sigma}
                    {8\pi^\frac{1}{2}K^\frac{1}{2}} \right) \right|
                -\log s } \right) .
	  \end{split}
	\end{align}
	
	Under the assumption $s\in (0,1)$, this expression can be
        estimated as
	\begin{align}
		\begin{split}
                  \widetilde f_{\sigma}(s) 
                  & \geq \frac{|\log \sigma|^2}{W_{-1}^2
                 	 \left(
                 	 	-\frac{\bar g(\lambda)\sigma}
                 	 	{8\pi^\frac{1}{2}K^\frac{1}{2}}
                 	 \right)}
                 	 \left(-1 + 
                 	 	\frac{1}{2}\min\left\{ 
                 	 		2\left|W_{-1}
                  				\left(
                  					 -\frac{\bar g(\lambda)\sigma}
                  					 {8\pi^\frac{1}{2}K^\frac{1}{2}}
                  				\right)
                  			\right|,
                  			|\log (s)| 
                  		\right\}
                 	 \right).
       \end{split}
	\end{align}
	Together with \eqref{lambert_expansion}, \eqref{gcC} and
        $K\in [\frac{K^*}{2},2K^*]$ we get for
        $\sigma \in (0,\sigma_0)$ with $\sigma_0>0$ small enough and
        some universal constant $C>1$ that for
        $s \in \left( 0,\frac{1}{C} \right)$ we have
        $\widetilde f_\sigma(s)>1$.  As a result of \eqref{boundbym},
        we therefore get $s_\sigma \geq \frac{1}{C}$.
	
	To handle the denominator in the second term on the right hand
        side of \eqref{tildef} in the case
        $s \in \left[ 1 ,\frac{16\pi \sigma^2}{K t_0}\right)$, note
        that for all such $s$ we have $s \leq C |\log \sigma|^{2}$
        for $C>0$ universal by \eqref{logt0L0}, \eqref{gcC} and
        $K \in \left[ \frac{K^*}{2},2K^*\right]$.  Again using
        \eqref{gcC} and $K \in \left[ \frac{K^*}{2},2K^*\right]$, we
        therefore get for $\sigma \in (0,\sigma_0)$ with $\sigma_0$
        sufficiently small that
	\begin{align}
		\log s \leq \left|W_{-1}
                  				\left(
                  					 -\frac{\bar g(\lambda)\sigma}
                  					 {8\pi^\frac{1}{2}K^\frac{1}{2}}
                  				\right)
                  			\right|.
	\end{align}
	Thus, together with \eqref{lambert_expansion}, \eqref{gcC} and
        $K\in [\frac{K^*}{2},2K^*]$ we deduce
	\begin{align}
          \widetilde f_\sigma(s) \geq \frac{1}{C}\left(s-1 -
          2 \log s  \right)
	\end{align}
	for $s \in \left[ 1 ,\frac{16\pi \sigma^2}{K t_0}\right)$,
        $\sigma \in (0,\sigma_0)$ with $\sigma_0>0$ sufficiently small
        and $C>0$ universal.  By the assumption \eqref{boundbym},
        we thus get $s_\sigma \leq C$, and in total
        $\frac{1}{C} \leq s_\sigma \leq C$ for some $C>0$ universal
        and $\sigma \in (0,\sigma_0)$ with $\sigma_0 >0$ small enough.
	 
	Finally, we can translate the estimate for $s_\sigma$ back to
        $L_\sigma$ using the relations
        $s_\sigma=\frac{t_\sigma}{t_0}$,
        $t_\sigma= K^{-1}L_\sigma^{-2}$ and $t_0 = K^{-1} L_0^{-2}$.
        Therefore, with the help of equation \eqref{logt0L0} and
        \eqref{gcC}, we obtain the desired estimate in
          \eqref{stability_L}.

        \textit{Step 5: Proof of stability for $\rho_\sigma$.}\\
        Using
        $\mathcal{E}_{\sigma,\lambda;K} (\rho_\sigma,\theta_\sigma,L_\sigma) \geq
        \min_{V_\sigma} \mathcal{E}_{\sigma,\lambda;K} = \min_{(0,
          \frac{16\pi}{K}\sigma^2]} f$, the fact that
        $\Delta (\lambda,\theta_\sigma)\geq 0$,
        $L_\sigma>\frac{1}{4 \sigma \sqrt{\pi}}$ and
        $K \in [\frac{K^*}{2},2K^*]$ in the identity
        \eqref{stability_representation} we obtain
	\begin{align}\label{stabilityrhofirst}
          \left( \rho_\sigma - \frac{|\log \sigma|}{\log(K^\frac{1}{2}
          L_\sigma)} \frac{\bar g(\lambda)}{16\pi} \right)^2 \leq
          \frac{C}{|\log \sigma |}. 
	\end{align}
	Again, using $L_\sigma>\frac{1}{4 \sigma \sqrt{\pi}}$
        together with $\sigma \in (0,\sigma_0)$ for $\sigma_0>0$
        sufficiently small and $K \in [\frac{K^*}{2},2K^*]$ gains
        us 
	\begin{align}
          \left(\frac{|\log \sigma|}{\log(K^\frac{1}{2} L_\sigma)}-1
          \right)^{2} =\frac{\log^2(K^\frac{1}{2}\sigma
          L_\sigma)}{\log^2(K^\frac{1}{2}L_\sigma)} \leq C \,
          \frac{\log^2(\sigma L_\sigma) + 1}{\log^2\sigma}, 
	\end{align}
	which by the estimate \eqref{stability_L} can be upgraded to
	\begin{align}
          \left|\frac{|\log
          \sigma|}{\log(K^\frac{1}{2}L_\sigma)}-1 \right| \leq C \,
          \frac{\log|\log\sigma| }{\log\sigma} \leq \frac{C}{|\log
          \sigma|^{1/2}}
	\end{align}
	for $\sigma \in (0, \sigma_0)$ and $\sigma_0>0$ small
        enough.  In particular, we conclude that
	\begin{align}\label{stabilityrholast}
	  \begin{split}
            \left| \rho_\sigma - \frac{\bar g(\lambda)}{16\pi} \right|
            & \leq \left| \rho_\sigma - \frac{|\log
                \sigma|}{\log(K^\frac{1}{2} L_\sigma)} \frac{\bar
                g(\lambda)}{16\pi} \right| + \frac{\bar
              g(\lambda)}{16\pi} \left|\frac{|\log
                \sigma|}{\log(K^\frac{1}{2}L_\sigma)}-1 \right| \leq
            \frac{C}{|\log \sigma|^{1/2}},
	  \end{split}
	\end{align}
        which gives the first limit in \eqref{limrhosthetas}.
	
      \textit{Step 6: Proof of stability for $\theta_\sigma$.}\\
      Turning towards proving an estimate for $\theta_\sigma$, as
        for estimate \eqref{stabilityrhofirst} we similarly get from
      the representation \eqref{stability_representation}, the
      estimate \eqref{stabilityrholast} and \eqref{gcC} that
 	\begin{align}\label{Deltagestimated}
          \Delta (\lambda,|\theta_\sigma|) \leq \frac{C}{|\log\sigma|}.
	\end{align}
	 Therefore,
        from estimate \eqref{Deltagestimated} and the form of the
        function $\Delta (\lambda, \theta)$ computed in equations
        \eqref{defdeltag1} and \eqref{defdeltag2} together with the
        facts that
        $\cos \theta^+_0 = \frac{32 }{3 \pi^2}
        \frac{\lambda}{1-\lambda}$ in the case $\lambda < \lambda_{c}$
        and
        $ \frac{1}{2}(1-\cos \theta_\sigma )^2 \leq (1-\cos
        \theta_\sigma )$, with $\theta_0^+ = 0$, in the case
        $\lambda \geq \lambda_{c}$, we obtain
	\begin{align}
          \lim_{\sigma \to 0} \cos |\theta_\sigma|  = \cos
          \theta^+_0. 
	\end{align}
	As $z \mapsto \arccos z$ is a continuous function from $
          [-1,1]$ to $[0, \pi]$, we obtain the second limit in
          \eqref{limrhosthetas}, which concludes the proof.
\end{proof}

\begin{proof}[Proof of Corollary \ref{cor:ansatz}]
  For $\sigma \in (0,\sigma_0)$ with $\sigma_0>0$ small enough, we use
  $\rho = \smash{\frac{\rho_0}{|\log\sigma|}}$, $\theta=\theta^+_0$
  and $L= \smash{2 \pi^\frac{1}{2} L_0}$ in Lemma
  \ref{lem:approximation_physical}, where $\rho_0$, $\theta^+_0$ and
  $L_0$ are from the first part of Proposition
  \ref{prop:reduced_minimization}.
\end{proof}

\begin{proof}[Proof of Lemma \ref{cor:upper_bound}]
  As the proof of Proposition \ref{prop:reduced_minimization} already
  contains the full details of the minimization, and as this proof
  closely follows that of \cite[Lemma 3.1]{melcher2014chiral}, we only
  provide a sketch here.  The main step is to find truncations of
  suitable Belavin-Polyakov profiles such that the sum of the DMI and
  stray field terms are negative.  In the case $\lambda =1$, we choose
  $\phi_{\rho,0,L}$, which by equation \eqref{BP_DMI} has negative DMI
  contribution for sufficiently large $L$ depending on $\rho$.  In the
  case $\lambda <1$ we choose $\phi_{\rho,\frac{\pi}{2},L}$, as for
  this function the DMI and volume charge contributions vanish and
  only the stray field terms contribute a negative term.  First
  minimizing in $\rho$ and then taking $L$ large enough gives the
  desired statement.
\end{proof}

\subsection{Existence of minimizers via the concentration compactness
  principle}

We are now in a position to prove existence of minimizers.

\begin{proof}[Proof of Theorem \ref{thm:existence}]
  Throughout the proof $C(\sigma,\lambda)$ denotes a generic constant
  depending on $\sigma$ and $\lambda$ that may change from estimate to
  estimate.

	By definition \eqref{def_energy_relaxation}, there exist
        $m_n \in \mathcal{D}$ for $n\in \N$ such that
	\begin{align}
          \lim_{n\to \infty} E_{\sigma,\lambda}(m_n)  = \inf_{m\in
          \mathcal{A}} E_{\sigma,\lambda}(m). 
	\end{align}
	Consider the Borel measures
	\begin{align}
		\mu_n(A) := \int_{A}\left( | \nabla m_n|^2  +  |m_n +
          e_3 |^2 \right) \intd x 
	\end{align}
	for all Borel sets $A \subset \R^2$.  By estimate
        \eqref{apriori_dirichlet_excess}, Lemma
        \ref{cor:upper_bound} and the assumption
        $0<\sigma^2(1+\lambda)^2 \leq 2$ we have
	\begin{align}\label{minimizer_in_interior}
		\int_{\R^2} |\nabla m_n|^2 \intd x  &< 16\pi
	\end{align}
	for $n$ large enough. Hence Lemma
        \ref{lem:topological_bound}, Lemma \ref{lemma:sobolev} and the
        estimate \eqref{apriori_anisotropy} imply
	\begin{align}\label{lbby8pi}
          8\pi \leq \mu_n(\R^2) & \leq C(\sigma,\lambda).
	\end{align}
	Consequently, we may apply the concentration compactness
        principle \cite{lions1984concentration}, see also
        \cite[Section 4.3]{Struwe08}, to see that the limiting
        behavior of the sequence $\mu_n$, up to taking subsequences,
        falls into three alternatives of vanishing, splitting, and
        compactness.
	
\textit{Case 1: Vanishing.}\\
Here, we have for all $R>0$ that
\begin{align}
	\lim_{n\to \infty} \sup_{x\in \R^2} \mu_n(\ball{x}{R}) = 0. 
\end{align}
In this setting, the proof of \cite[Lemma 4.2]{melcher2014chiral}
establishes that $m_n + e_3 \to 0$ in $L^4(\R^2)$.  The identity
$|m_n(x) + e_3|^2 = 2(1 + m_{3,n}(x))$ implies that
\begin{align}\label{l4tol2}
  \lim_{n\to \infty} \int_{\R^2} (1+ m_{n,3} )^2 \intd x = \lim_{n\to
  \infty} \int_{R^2} \frac{1}{4} |m_n + e_3|^4 \intd x =0. 
\end{align}
Integrating by parts and applying Cauchy-Schwarz inequality, we
thus see that
\begin{align}
  \lim_{n\to \infty}   \int_{\R^2} m'_n \cdot \nabla m_{n,3} \intd x= -
  \lim_{n\to \infty}  \int_{\R^2}   (m_{n,3}+1 )\nabla \cdot m'_n
  \intd x =0. 
\end{align}
Similarly, the interpolation inequality \eqref{surf_interpolation}
implies
$F_{\mathrm{surf}}\left(m_{3,n} +1\right)\leq \frac{1}{2} \|m_{3,n} +
1\|_{2} \|\nabla m_{3,n}\|_{2}$, which, combined with the
  convergence \eqref{l4tol2}, yields
\begin{align}
  \lim_{n\to \infty} F_{\mathrm{surf}}(m_{3,n})   =0. 
\end{align}
Together with \eqref{vol_nonnegative}, the topological bound
\eqref{topological_bound_in_lemma} and Lemma
\ref{cor:upper_bound}, this then yields a contradiction:
\begin{align}
  \liminf_{n\to \infty} E_{\sigma,\lambda}(m_n) \geq  8\pi
  >\inf_{m\in 
  \mathcal{A}} E_{\sigma,\lambda}(m) = \lim_{n\to \infty}
  E_{\sigma,\lambda}(m_n), 
\end{align}
ruling out the case of vanishing.

\textit{Case 2: Splitting.}\\
In this case, there exists $0<\eta<1$ with the following property: For
all $\eps>0$, after a suitable translation depending on $\eps$, there
exists $R>0$ such that for all $\tilde R>R$ we have
	\begin{align}\label{splitting_assumption}
          \limsup_{n\to \infty} \left( \left| \mu_n(\ball{0}{R}) - \eta \,
          \mu_n(\R^2) \right| + \left| \mu_n\left(\stcomp{B}_{\tilde
          R}(0)\right) -(1-\eta) \mu_n(\R^2) \right| \right) \leq \eps. 
	\end{align}
	Without loss of generality, we may assume $R\geq 1$.  Let
        $\tilde R > 32 R$.  Then the proof of \cite[Lemma
        8]{doring2017compactness} establishes the existence of
        $R_n \in (R,2R)$,
        $\tilde R_n \in \left(\frac{\tilde R}{4},\frac{\tilde
            R}{2}\right)$ and smooth
        $m^{(1)}_n, m^{(1)}_n : \R^2 \to \Sph^2$ such that
	\begin{align}
          m^{(1)}_n(x) & = m_n(x) &&  \text{ for } x \in
                                     \ball{0}{R_n},\label{split_up_first}\\ 
          m^{(1)}_n(x) &= -  e_3 && \text{ for }  x \in
                                    \stcomp{B}_{2R_n}(0),\\ 
          m^{(2)}_n(x) & = m_n(x) && \text{ for }  x\in
                                     \stcomp{B}_{2\tilde R_n}(0),\\ 
          m^{(2)}_n(x) &  = -  e_3 && \text{ for }  x\in
                                      \ball{0}{\tilde
                                      R_n}\label{split_up_intermediate} 
	\end{align}
	and
	\begin{align}
          \int_{\stcomp{B}_{R_n}(0)} \left( \left|\nabla
          m^{(1)}_n\right|^2 + \left| m^{(1)}_n + e_3 \right|^2
          \right) \intd x  & \leq C(\sigma,\lambda)
                             \eps,\label{split_first_component_estimate}\\ 
          \int_{\ball{0}{2\tilde R_n}} \left( \left|\nabla
          m^{(2)}_n\right|^2 + \left| m^{(2)}_n + e_3 \right|^2
          \right) \intd x  & \leq C(\sigma,\lambda)
                             \eps.\label{split_up_last} 
	\end{align}
	
	By the pointwise almost everywhere estimate
        $|m \cdot (\partial_1 m \times \partial_2 m) | \leq C |\nabla
        m|^2$ and the estimates \eqref{splitting_assumption},
        \eqref{split_first_component_estimate} and
        \eqref{split_up_last} we get that
        $|\mathcal{N}(m_n) - \mathcal{N}(m^{(1)}_n)
        -\mathcal{N}(m^{(2)}_n)| \leq C(\sigma,\lambda)\eps$.
        Discreteness of the degree then implies for $\eps>0$ small
        enough that
	\begin{align}\label{degreesadd}
		1 = \mathcal{N}(m_n) = \mathcal{N}(m^{(1)}_n)
          +\mathcal{N}(m^{(2)}_n). 
	\end{align}
	Next, combining the estimates
        \eqref{split_first_component_estimate}, \eqref{split_up_last}
        and \eqref{minimizer_in_interior}, gives
	\begin{align}\label{boundgradients}
          \int_{\R^2}\left( \left|\nabla m^{(1)}_n\right|^2 +
          \left|\nabla m^{(2)}_n\right|^2 \right) \intd x \leq
          \int_{\R^2} |\nabla m_n|^2 \intd x + C(\sigma,\lambda)\eps <
          16\pi + C(\sigma,\lambda)\eps. 
	\end{align}
	Therefore, for $\eps >0$ small enough we get, applying the
        topological bound \eqref{topological_bound_in_lemma} along the
        way, that
	\begin{align}\label{gradient_bound_full_space}
          8\pi\left(\left|\mathcal{N}(m^{(1)}_n)\right| +
          \left|\mathcal{N}(m^{(2)}_n)\right| \right)\leq
          \int_{\R^2}\left( \left|\nabla m^{(1)}_n\right|^2 +
          \left|\nabla m^{(2)}_n\right|^2 \right)\intd x < 24\pi. 
	\end{align}
	Elementary combinatorics using the identity \eqref{degreesadd}
        consequently give
		\begin{align}
                  \mathcal{N}\left(m^{(1)}_n\right)  = 1  \text{
                  and }   \mathcal{N}\left(m^{(2)}_n\right) = 0 
		\end{align}
		or
		\begin{align}
                  \mathcal{N}\left(m^{(1)}_n\right)  = 0  \text{ and }
                  \mathcal{N}\left(m^{(2)}_n\right) = 1. 
		\end{align}
		In the following we will only deal with the first
                case, as the other one can be handled similarly.
		
		By the estimate \eqref{lbby8pi} and the splitting
                alternative \eqref{splitting_assumption} for $\eps >0$
                small and $n$ large enough we obtain 
		\begin{align}\label{concat1}
                  4\pi (1-\eta) \leq \frac{1-\eta}{2} \mu_n(\R^2) \leq
                  \int_{\R^2}\left( \left|\nabla m^{(2)}_n\right|^2 +
                  | m^{(2)}_n + e_3|^2 \right) \intd x. 
		\end{align}
		Lemma \ref{lemma:sobolev} along with the bound
                \eqref{boundgradients} further implies
		\begin{align}\label{concat2}
		  \begin{split}
                    \int_{\R^2}\left( \left|\nabla m^{(2)}_n\right|^2
                      + | m^{(2)}_n + e_3|^2 \right) \intd x \leq
                    \int_{\R^2} \left|\nabla m^{(2)}_n\right|^2 \intd
                    x + C(\sigma,\lambda) \int_{\R^2} \left|
                      \left(m^{(2)}_{n}\right)' \right|^2\intd x.
		  \end{split}
		\end{align}
		As by the topological lower bound we have
                $\int_{\R^2} |\nabla m_n^{(1)}|^2 \intd x \geq 8\pi$,
                we obtain from estimate \eqref{boundgradients} for
                $\eps>0$ small enough that
		\begin{align}
			\int_{\R^2} |\nabla m_n^{(2)}|^2 \intd x < 16\pi.
		\end{align}
		Therefore we can apply Lemma \ref{lemma:lower_bound} to get
		\begin{align}\label{concat3}
                  \int_{\R^2} \left|\nabla m^{(2)}_n\right|^2 \intd x
                  +  \int_{\R^2} \left| \left(m^{(2)}_{n}\right)'
                  \right|^2\intd x \leq C(\sigma,\lambda)
                  E_{\sigma,\lambda}\left(m^{(2)}_n \right). 
		\end{align}
		Consequently, concatenating the estimates
                \eqref{concat1}, \eqref{concat2} and \eqref{concat3}
                we deduce that there exists $\delta>0$ such that
                  for all $n$ large enough we have
		\begin{align}\label{m2positive}
			E_{\sigma,\lambda}\left(m^{(2)}_n \right) \geq \delta>0.
		\end{align}
		
		As a result, in order to rule out splitting we only have to prove that
		\begin{align}\label{energy_splits}
                  \limsup_{n\to \infty}\left[
                  E_{\sigma,\lambda}\left(m^{(1)}_n \right)
                  +E_{\sigma,\lambda}\left(m^{(2)}_n \right)
                  -E_{\sigma,\lambda}\left(m_n \right) \right] \leq
                  g(\eps,\tilde R) 
		\end{align}
		for some function $g: (0,\infty)^2 \to (0,\infty)$
                with
                $\lim_{\eps \to 0} \lim_{\tilde R \to \infty}
                g(\eps,\tilde R) = 0$.  Indeed, assuming that the
                bound \eqref{energy_splits} holds, we can test the
                infimum $\inf_{ \mathcal{A}} E_{\sigma,\lambda}$ with
                $m_n^{(1)}$ and use the estimates \eqref{m2positive}
                and \eqref{energy_splits} to get
		\begin{align}
		  \begin{split}
                    \inf_{m\in \mathcal{A}} E_{\sigma,\lambda}(m) &
                    \leq \liminf_{n\to \infty}
                    E_{\sigma,\lambda}\left(m^{(1)}_n \right)\\
                    & \leq \limsup_{n\to \infty} \left(
                      E_{\sigma,\lambda}\left(m^{(1)}_n \right)
                      +E_{\sigma,\lambda}\left(m^{(2)}_n \right) -
                      \delta \right) \\
                    & \leq \lim_{n\to\infty}
                    E_{\sigma,\lambda}\left(m_n \right) +
                    g(\eps,\tilde R) - \delta.
		  \end{split}
		 \end{align}
		 Then, by first taking $\tilde R$ big enough and then
                 $\eps>0$ small enough we obtain a contradiction.
		
                 We now turn to proving the claim
                 \eqref{energy_splits}.  The local terms are
                 straightforward to handle using the Cauchy-Schwarz
                 inequality, see for example the proof of Lemma
                 \ref{lemma:lower_bound}, and give a contribution of
                 $C(\sigma,\lambda)\eps$ to $g(\eps,\tilde R)$.  Of
                 the nonlocal terms, we first deal with the volume
                 charges by computing
		\begin{align}
			\begin{split}
                          F_{\mathrm{vol}}(m'_n) & =
                          F_{\mathrm{vol}}\left(
                            \left(m^{(1)}_{n}\right)'\right) +
                          F_{\mathrm{vol}}\left(
                            \left(m^{(2)}_{n}\right)'\right) +
                          F_{\mathrm{vol}}\left(m'_{n}-
                            \left(m^{(1)}_{n}\right)'-
                            \left(m^{(2)}_{n}\right)'\right)\\
                          &\quad
                          +2F_{\mathrm{vol}}\left(\left(m^{(1)}_{n}\right)'
                            + \left(m^{(2)}_{n}\right)', m'_{n}-
                            \left( m^{(1)}_{n}\right)'-
                            \left(m^{(2)}_{n}\right)'\right) \\
                          & \quad + 2
                          F_{\mathrm{vol}}\left(\left(m^{(1)}_{n}\right)',
                            \left(m^{(2)}_{n}\right)'\right). \label{Fvolm1m2}
			\end{split}
		\end{align}
		By \eqref{vol_nonnegative}, we may discard the term
                $F_{\mathrm{vol}}\left(m'_{n}- m'^{(1)}_{n}-
                  m'^{(2)}_{n}\right) \geq 0$ right away since we only
                claim an upper bound in estimate
                \eqref{energy_splits}.  The usual interpolation
                inequality \eqref{vol_interpolation} for $p=2$ implies
		\begin{align}\label{f_vol_mixed}
		  \begin{split}
                    & \quad
                    F_{\mathrm{vol}}\left(\left(m^{(1)}_{n}\right)' +
                      \left(m^{(2)}_{n}\right), m'_{n}-
                      \left(m^{(1)}_{n}\right)'-\left( m^{(2)}_{n}
                      \right)'\right)  \\ 
                    & \leq C \left(\left\|\nabla \left(
                          m^{(1)}_{n}\right)'\right\|_{2} +
                      \left\|\nabla\left(
                          m^{(2)}_{n}\right)'\right\|_{2}
                    \right)\left\|m'_{n}- \left(m^{(1)}_{n}\right)'-
                      \left(m^{(2)}_{n}\right)' \right\|_{2}.\\ 
		  \end{split}
		\end{align}
		Using \eqref{gradient_bound_full_space} we obtain that
                the $\mathring H^1$-norms on the left-hand side
                  are uniformly bounded.  The identities
                  \eqref{split_up_first} through
                  \eqref{split_up_intermediate} imply
                \begin{align}
                  \operatorname{supp}\left( m_n' - \left( m^{(1)}_n
                  \right)' - \left( m^{(2)}_n \right)'\right) \subset
                  \overline{B}_{2 \tilde R_n}(0)\setminus
                  \ball{0}{R_n}, 
                 \end{align}
                 and on
                 $\overline{B}_{2 \tilde R_n}(0)\setminus
                 \ball{0}{R_n}$ we can estimate each of the terms
                 $m_n'$, $\left( m^{(1)}_n \right)'$, and
                 $\left(m^{(2)}_n \right)'$ separately by virtue of
                 \eqref{splitting_assumption},
                 \eqref{split_first_component_estimate} and
                 \eqref{split_up_last}. Therefore \eqref{f_vol_mixed}
                 gives
		\begin{align}\label{second_error_term}
		  \begin{split}
                    F_{\mathrm{vol}}\left(\left(m^{(1)}_{n}\right)' +
                      \left(m^{(2)}_{n}\right)', m'_{n}-
                      \left(m^{(1)}_{n}\right)'-\left( m^{(2)}_{n}
                      \right)'\right) \leq C(\sigma,\lambda)
                    \eps^\frac{1}{2}.
		  \end{split}
		\end{align}
		
		To estimate the last term in \eqref{Fvolm1m2}, we
                would like to exploit that the sequences
                $\left( m^{(1)}_{n} \right)'$ and
                $\left( m^{(2)}_{n} \right)'$ have disjoint
                supports.  To this end, we use the real space
                representation \eqref{vol_equivalence} and integrate
                by parts once in each of the two integrals:
		\begin{align}
                  & \quad
                    F_{\mathrm{vol}}\left(\left(m^{(1)}_{n}\right)',
                    \left(m^{(2)}_{n}\right)'\right)\notag\\  
                  & = \frac{1}{4\pi}\int_{\R^2}\int_{\R^2}
                    \frac{\nabla \cdot \left(m^{(1)}_{n}(x)\right)'
                    \nabla \cdot \left(m^{(2)}_{n}(\tilde
                    x)\right)'}{|x-\tilde x|} \intd \tilde x \intd x
                  \\ 
                  &  = \frac{1}{4\pi}\int_{\R^2}\int_{\R^2}\left(
                    \frac{ \left(m^{(1)}_{n}(x)\right)'  \cdot
                    \left(m^{(2)}_{n}(\tilde x)\right)'}{|x-\tilde
                    x|^3} -3 \frac{ \left(m^{(1)}_{n}(x)\right)' \cdot
                    (\tilde x -x)\, \left(m^{(2)}_{n}(x)\right)' \cdot
                    (\tilde x -x) }{|x-\tilde x|^5} \right) \intd
                    \tilde x \intd x.\notag 
		\end{align}
		To extract a quantitative estimate, note that we have
		\begin{align}\label{quantitative_difference}
                  \inf \left\{|x- \tilde x| :  x\in
                  \operatorname{supp} \left(m^{(1)}_{n,3} + 1\right),
                  \tilde x \in \operatorname{supp} \left(m^{(2)}_{n,3}
                  + 1\right)  \right\} \geq \frac{\tilde R}{4}- 4R. 
		\end{align}
		With
                $K_{\mathrm{vol}}(z):= \chi\left(|z| \geq \frac{\tilde
                    R}{4}- 4R\right)\left(\frac{1}{|z|^3}\id - 3
                  \frac{z\otimes z}{|z|^5}\right)$ for $z \in \R^2$,
                Young's inequality for convolutions implies for
                $\tilde R > 32 R$ that
		\begin{align}
		  \begin{split}
                    \int_{\R^2}\int_{\R^2}
                    \left(m^{(1)}_{n}(x)\right)' \cdot
                    K_{\mathrm{vol}}(\tilde x-x) \,
                    \left(m^{(2)}_{n}(\tilde x)\right)' \intd \tilde x
                    \intd x
                    & \leq C \left\| \left(m^{(1)}_{n}\right)'
                    \right\|_{2}  \left\| \left(m^{(2)}_{n}\right)'
                    \right\|_{2} \| K_{\mathrm{vol}} \|_{1}\\ 
                    & \leq C(\sigma,\lambda)\tilde R ^{-1}.
		  \end{split}
		\end{align}

		For the surface charges we similarly compute
		\begin{align}\label{surface_splitting}
			\begin{split}
                          F_{\mathrm{surf}}(m_{n,3}) & =
                          F_{\mathrm{surf}}\left( m^{(1)}_{n,3}\right)
                          + F_{\mathrm{surf}}\left(
                            m^{(2)}_{n,3}\right)
                          \\
                          &\quad + F_{\mathrm{surf}}\left(m_{n,3}+
                            m^{(1)}_{n,3} +m^{(2)}_{n,3} , m_{n,3} -
                            m^{(1)}_{n,3} -m^{(2)}_{n,3} \right) + 2
                          F_{\mathrm{surf}}\left( m^{(1)}_{n,3},
                            m^{(2)}_{n,3}\right).
			\end{split}
		\end{align}
		As in the proof of estimate \eqref{second_error_term},
                the interpolation inequality
                \eqref{surf_interpolation} together with the fact that
                $F_{\mathrm{surf}}$ is invariant under the addition of
                constant functions gives
		\begin{align}
		  \begin{split}
                    F_{\mathrm{surf}}\left(m_{n,3}+ m^{(1)}_{n,3}
                      +m^{(2)}_{n,3} , m_{n,3} - m^{(1)}_{n,3}
                      -m^{(2)}_{n,3} \right) \leq C(\sigma,\lambda)
                    \eps^\frac{1}{2}.
			\end{split}
		\end{align}
		The real-space representation
                \eqref{nonlocal_real_space_2_extended} allows us to
                write the last term in the identity
                \eqref{surface_splitting} as
		\begin{align}
                  F_{\mathrm{surf}}\left( m^{(1)}_{n,3},
                  m^{(2)}_{n,3}\right) = \frac{1}{8\pi}\int_{\R^2}
                  \int_{\R^2} \frac{\left( m^{(1)}_{n,3}(x) -
                  m^{(1)}_{n,3}(\tilde x)
                  \right)\left(m^{(2)}_{n,3}(x) - m^{(2)}_{n,3}(\tilde
                  x)\right)}{|x-\tilde x|^3} \intd \tilde x \intd x. 
		\end{align}
		We may now exploit the fact that $m^{(1)}_{n,3}
                +1$ and $m^{(2)}_{n,3}+ 1$ have disjoint
                supports to get
		\begin{align}
                  F_{\mathrm{surf}}\left( m^{(1)}_{n,3},
                  m^{(2)}_{n,3}\right) =
                  \frac{1}{4\pi}\int_{\R^2} \int_{\R^2} \frac{\left(
                  m^{(1)}_{n,3}(x) +1 \right)\left(
                  m^{(2)}_{n,3}(\tilde x)+1\right)}{|x-\tilde x|^3}
                  \intd \tilde x \intd x, 
		\end{align}
		so that Young's inequality for convolutions with
                $K_{\mathrm{surf}}(z):= \frac{\chi\left(|z| \geq
                    \frac{\tilde R}{4}- 4R\right) }{|z|^3}$ for
                $z \in \R^2$ implies
		\begin{align}
		  \begin{split}
                    \F_{\mathrm{surf}}\left( m^{(1)}_{n,3},
                      m^{(2)}_{n,3}\right) \leq C \left\|m^{(1)}_{n,3}
                      +1 \right\|_{2} \left\|m^{(2)}_{n,3} + 1
                    \right\|_{2} \| K_{\mathrm{surf}} \|_{1} \leq
                    C(\sigma,\lambda)\tilde R^{-1}.
		  \end{split}
		\end{align}
		
		All together, we see that estimate \eqref{energy_splits} holds with 
		\begin{align}
                  g(\eps,\tilde R) = C(\sigma,\lambda)\left(
                  \eps^\frac{1}{2} +  \tilde R^{-1} \right), 
		\end{align}
		which rules out splitting.
		
                \textit{Case 3: Compactness}\\
                As vanishing and splitting have been ruled out, we
                obtain that after extraction of a subsequence and
                suitable translations, for every $\eps>0$ there exists
                $R>0$ such that we have
	\begin{align}\label{compactness_alternative}
          \mu_n(\stcomp{B}_R(0)) \leq \eps
	\end{align}
	for all $n\in \N$.
	
	By the Rellich-Kondrachov compactness theorem, \cite[Theorem
        8.9]{lieb-loss}, there exists $m_\sigma : \R^2 \to \Sph^2$
        such that $m_n + e_3 \to m_\sigma + e_3$ in
        $L^2(\ball{0}{\tilde R};\R^3)$ for all $\tilde R>0$ and
        $ m_n + e_3\warr m_\sigma + e_3$ in $H^1(\R^2;\R^3)$.  We
        first argue that we even have $m_n \to m_\sigma$ in
        $L^2(\R^2;\R^3)$.  Let $\eps >0$ and let $R>0$ be such that
        the tightness estimate \eqref{compactness_alternative} holds.
        Then, by lower semi-continuity of the $L^2$-norm and the
        Minkowski inequality, we have
	\begin{align}\label{m_n_to_m_L_2}
          \limsup_{n\to\infty}\int_{\R^2} | m_n - m_\sigma|^2 \intd x
          \leq 2\eps + \limsup_{n\to\infty} \int_{\ball{0}{R}}
          |m_n-m_\sigma|^2\intd x  = 2\eps. 
	\end{align}
	Therefore, we see $m_n + e_3 \to m_\sigma + e_3$ in
        $L^2(\R^2;\R^3)$,
        and in particular we have $m_\sigma +e_3 \in L^2(\R^2; \R^3)$.
	
	Next, we argue that
	\begin{align}
          \liminf_{n\to \infty}\left( E_{\sigma,\lambda}(m_n) - 8\pi
          \mathcal{N}(m_n) \right) \geq E_{\sigma,\lambda}(m_\sigma) -
          8\pi\mathcal{N}(m_\sigma). 
	\end{align}
	By the identity \eqref{completed_square},
	we obtain for any $n\in \N$ that
	\begin{align}\label{representationdifference}
	  \begin{split}
            & \quad  E_{\sigma,\lambda}(m_n) - 8\pi \mathcal{N}(m_n) \\
            & = \int_{\R^2} \left( | \partial_1 m_n + m_n \times
              \partial_2 m_n|^2 + \sigma^2 |m'_{n}|^2 -2\sigma^2
              \lambda \, m'_{n}\cdot \nabla m_{n,3}\right) \intd x \\
            & \quad + \sigma^2(1-\lambda) \left(
              F_{\mathrm{vol}}(m'_{n}) - F_{\mathrm{surf}}(m_{n,3})
            \right).
	  \end{split}
	\end{align}
	We have $\partial_1 m_n \warr \partial_1 m_\sigma$ and
        $m_n \times \partial_2 m_n \warr m_\sigma \times \partial_2
        m_\sigma$ in $L^2(\R^2;\R^3)$, the latter by a
        weak-times-strong convergence argument.  In the first term, we
        can thus use lower semi-continuity of the $L^2$-norm.  The
        anisotropy term converges strongly by our previous
        argument. By \eqref{m_n_to_m_L_2} and weak convergence of the
        gradients we have
        \begin{align}
          \lim_{n\to \infty}  \int_{\R^2} m_n' \cdot \nabla m_{n,3}
          \intd x = \lim_{n\to \infty}  \int_{\R^2} m_\sigma ' \cdot
          \nabla m_{n,3} \intd x  =   \int_{\R^2} m_\sigma ' \cdot
          \nabla m_{\sigma,3} \intd x  
        \end{align}
        so that also the DMI-term converges.
        Finally, we see
        $F_{\mathrm{vol}}(m'_{n}) \to F_{\mathrm{vol}}(m'_\sigma)$ and
        $F_{\mathrm{surf}}(m_{3,n}) \to
        F_{\mathrm{surf}}(m_{\sigma,3})$ as $n\to \infty$ by the
        interpolation inequalities of Lemma \ref{lemma:fourier_basic}.
        Taking all of these things together, we see
	\begin{align}
	  \begin{split}
            & \quad \liminf_{n\to \infty} \left(
              E_{\sigma,\lambda}(m_n) - 8\pi \mathcal{N}(m_n) \right)
            \\
            & \geq \int_{\R^2} \Big(  | \partial_1 m_\sigma +
            m_\sigma \times \partial_2  m_\sigma|^2 + \sigma^2
            |m'_\sigma|^2 -2\sigma^2 \lambda (m'_\sigma\cdot \nabla)
            m_{\sigma,3}\Big) \intd x\\ 
            & \qquad \qquad  + \sigma^2(1-\lambda) \left(
              F_{\mathrm{vol}}(m'_\sigma) -
              F_{\mathrm{surf}}(m_{\sigma,3}) \right) \\ 
            & = E_{\sigma,\lambda}(m_\sigma) -
            8\pi\mathcal{N}(m_\sigma),
	  \end{split}
	\end{align}
	where in the last line we again use the identity
        \eqref{representationdifference}.
	
	Therefore, together with the upper bound of Lemma
        \ref{cor:upper_bound} and the observation that
        $E_{\sigma,\lambda}(m_\sigma) \geq 0$, which follows from the
        assumption of the theorem and Lemma \ref{lemma:lower_bound},
        we have
	\begin{align}
          0 >\liminf_{n\to \infty} \left( E_{\sigma,\lambda}(m_n) -
          8\pi \right) \geq E_{\sigma,\lambda}(m_\sigma) -
          8\pi\mathcal{N}(m_\sigma) > - 8\pi \mathcal{N}(m_\sigma), 
	\end{align}
	giving $\mathcal{N}(m_\sigma) >0$.  At the same time, by lower
        semi-continuity of the Dirichlet energy and the estimate
        \eqref{minimizer_in_interior} we furthermore have
	\begin{align}
          \int_{\R^2} |\nabla m_\sigma|^2 \intd x \leq \lim_{n \to
          \infty} \int_{\R^2} |\nabla m_n|^2 \intd x < 16\pi. 
	\end{align}
	Thus the topological bound \eqref{topological_bound_in_lemma}
        implies $\mathcal{N}(m)= 1$.  As we have already shown
        $m_\sigma + e_3 \in L^2(\R^2;\R^3)$ above, we therefore have
        $m_\sigma \in \mathcal{A}$.  Consequently, we have
	\begin{align}
          \inf_{\widetilde m \in \mathcal{A}}
          E_{\sigma,\lambda}(\widetilde m) =\liminf_{n\to \infty}
          E_{\sigma,\lambda}(m_n) \geq E_{\sigma,\lambda}(m_\sigma)
          \geq \inf_{\widetilde m \in \mathcal{A}}
          E_{\sigma,\lambda}(\widetilde m), 
	\end{align}
	which concludes the proof.
\end{proof}

\section{The conformal limit}\label{section:conformal}

In this section we prove Theorem \ref{thm:convergence}.  In the spirit
of a $\Gamma$-convergence argument, we do so by providing ansatz-free
lower bounds matching the upper bounds obtained in Corollary
\ref{cor:ansatz}.  As the Dirichlet term provides closeness to a
Belavin-Polyakov profile $\phi = S \Phi(\rho^{-1} x)$ for
$S \in \operatorname{SO}(3)$ and $\rho>0$ via Theorem
\ref{thm:quantitative_stability}, we have to capture the behavior of
the lower order terms as the magnetization approaches $\phi$.

Here the main difficulty is the fact that the limiting
Belavin-Polyakov profile $\phi$ from Theorem
\ref{thm:quantitative_stability} does not necessarily satisfy
$\lim_{|x| \to \infty} \phi(x) = - e_3$, which is a more subtle issue
than one might expect.  The fundamental problem is that for $r>0$ the
embedding of $H^1(\ball{0}{r})$ to $L^\infty(\ball{0}{r}))$ fails
logarithmically, and we only have
$H^1(\ball{0}{r}) \hookrightarrow BMO(\ball{0}{r})$ (as a simple
result of the Poincar{\'e} inequality and the definition
\cite[(0.5)]{brezis1995degree} of $BMO$), which in and of itself is
not strong enough to control the value at infinity.  Indeed, at this
stage it is entirely possible that the minimizers exhibit a
multi-scale structure: On the scale of the skyrmion radius the profile
might approach a tilted Belavin-Polyakov profile, while only on a
larger truncation scale decaying to $-e_3$, see for example
\cite[Step 2b in the proof of Lemma 8]{doring2017compactness} for a
construction.  Of course such a profile would have a large anisotropy
energy, which we exploit in the following Lemma
\ref{lem:lower_bound_anisotropy_notpinned}.  The idea is to replace
the logarithmic failure of the embedding
$H^1 \not \hookrightarrow L^\infty$ with the Moser-Trudinger
inequality proved in Lemma \ref{lem:linearization_estimate}.

Throughout this section we use the abbreviations
\begin{align}\label{L_convention}
  L & : = \left( \int_{\R^2} |\nabla m|^2 \intd x - 8\pi
      \right)^{-\frac{1}{2}},\\ 
  \nu & :=  \lim_{|x| \to \infty} \phi(x) = - Se_3,\label{nudef}
\end{align}
provided $\phi= S\Phi(\rho^{-1}(\bigcdot - x_0))$ for
$S\in \operatorname{SO}(3)$, $\rho >0$ and $x_0 \in \R^2$.  Note that
this choice of $L$ is consistent with its usage in
$\mathcal{E}_{\sigma,\lambda;K^*}$, see definition
\eqref{defreducedenergy}, in view of estimate \eqref{BP_excess} and
the substitution \eqref{scale_of_cutoff_to_excess}.  In particular, it
can also be thought of as the cut-off length relative to the
skyrmion radius.

\begin{lemma}\label{lem:lower_bound_anisotropy_notpinned}
  There exist universal constants $L_0>0$ and $C>0$ such that for
  $m\in \mathcal{A}$ with $L \geq L_0$ and the distance-minimizing
  Belavin-Polyakov profile from Theorem
  \ref{thm:quantitative_stability} given by
  $\phi(x) = S \Phi(\rho^{-1}(x-x_0))$ with
  $S\in \operatorname{SO}(3)$, $\rho>0$ and $x_0 \in \R^2$ we have
	\begin{align}
          \int_{\R^2} \left| m' \right|^2 \intd x \geq C |\nu + e_3|^2
          \rho^2 L^2 . 
	\end{align}
\end{lemma}

In order to use this bound to rule out $\nu$ being too far away from
$-e_3$ we also need a first lower bound for the DMI and the stray
field terms.

\begin{lemma}\label{lem:lowerbound_dmi}
  There exist universal constants $L_0>0$ and $C>0$ such that for
  $m\in \mathcal{A}$ with $L \geq L_0$ and the distance-minimizing
  Belavin-Polyakov profile from Theorem
  \ref{thm:quantitative_stability} given by
  $\phi(x) = S \Phi(\rho^{-1}(x- x_0))$ with
  $S\in \operatorname{SO}(3)$, $\rho>0$ and $x_0 \in \R^2$ we have
	\begin{align}
	  \begin{split}
            & \quad - 2 \lambda \int_{\R^2} m' \cdot \nabla m_3 \intd
            x + (1-\lambda) \left( F_{\mathrm{vol}}( m') -
              F_{\mathrm{surf}}( m_3) \right) \\
            & \geq -\frac{1}{2}\int_{\R^2} | m' |^2 \intd x - C \rho
            \, (\log L)^\frac{1}{2} - C L^{-2}.
	  \end{split}
	\end{align}
\end{lemma}

Armed with these estimates, we obtain that $\nu$ does indeed converge
to $e_3$ as $\sigma \to 0$.

\begin{lemma}\label{lem:up_to_wobbling}
  There exist universal constants $C>0$ and $\sigma_0 >0$ such that
  for $\sigma \in (0,\sigma_0)$ the following holds: Let $m_\sigma$ be
  a minimizer of $E_{\sigma,\lambda}$ over $\mathcal{A}$.  Let
  $\rho>0$, $S \in \operatorname{SO}(3)$ and $x_0 \in \R^2$ be such
  that $S \Phi (\rho^{-1}(x-x_0))$ is the distance-minimizing
  Belavin-Polyakov profile from Theorem
  \ref{thm:quantitative_stability} for $m = m_\sigma$.  Then we have
	\begin{align}
          L^{-2} & \leq 16\pi \sigma^2,\label{apriori_L}\\
          | \nu +  e_3 |^2 & \leq C \, \frac{ \log^2 L
                             }{L^2}.\label{normal_converges} 
	\end{align}
\end{lemma}

Now that we know that we essentially have pinning of the value at
infinity, we turn to proving a more precise lower bound for the
anisotropy energy which almost matches the expression
\eqref{BP_anisotropy} obtained in Lemma
\ref{lem:approximation_physical}.

\begin{lemma}\label{lem:lower_bound_anisotropy_pinned}
  There exist universal constants $C>0$ and $L_0>0$ such that for
  $\overline L\geq L_0$ the following holds: Let
  $m\in \mathcal{A}$ be such that there exist
  $\phi (x) := S\Phi(\rho^{-1}(x-x_0))$ for $x\in \R^2$ with
  $S \in \operatorname{SO}(3)$, $\rho>0$ and $x_0\in \R^2$ satisfying
	\begin{align}
          \int_{\R^2} |\nabla (m - \phi) |^2 \intd x \leq {\overline L}^{-2},
	\end{align}
	and
	\begin{align}
          \label{Se3e3L}
		|S e_3 - e_3 |^2 < \overline L^{-1}.
	\end{align}
	Then we have the estimate
	\begin{align}
          \label{mprime4pirho2log}
          \int_{\R^2} | m' |^2 \intd x \geq 4\pi\rho^2\log
          \left(K^* {\overline L}^2 \right) - C\rho^2 {\overline L}^{-\frac{2}{3}},
	\end{align}
	where $K^*$ is the constant defined in equation \eqref{K*}.
\end{lemma}

We also have another look at the DMI and stray field terms in order to
obtain sharper estimates matching those of Lemma
\ref{lem:approximation_physical}.  As therein the expressions depend
on the rotation angle $\theta$, we have to replace the profile
$\phi$ obtained in Theorem \ref{thm:quantitative_stability} by a
rotated one having the correct value at infinity.

\begin{lemma}\label{lem:lower_bound_dmi_and_stray_pinned}
  There exist universal constants $ C>0$ and $\sigma_0 >0$ such that
  for $\sigma \in (0,\sigma_0)$ and $\lambda \in [0,1]$ the following
  holds: Let $m_\sigma$ be a minimizer of $E_{\sigma,\lambda}$ over
  $\mathcal{A}$.  Let $S \in \operatorname{SO}(3)$, $\rho>0$ and
  $x_0 \in \R^2$ be such that $S \Phi( \rho^{-1} (x-x_0))$ is the
  distance-minimizing Belavin-Polyakov profile from Theorem
  \ref{thm:quantitative_stability} for $m = m_\sigma$. Then there
  exists $\theta \in [-\pi,\pi)$ such that with the rotation
  $S_\theta$ of angle $\theta$ around the $x_3$-axis defined in
  \eqref{def_rotation} we have
	\begin{align}
          \int_{\R^2} \left|\nabla \left( m_\sigma(x) - S_{\theta}
          \Phi(\rho^{-1} (x-x_0)) \right)\right|^2 \intd x
          & \leq
            \frac{C
            \log^2
            L}{L^2}, \label{distance_to_pinned_BP}\\ 
          \int_{\R^2} 2 m'_\sigma \cdot \nabla  m_{\sigma,3} \intd
          x & \leq 8\pi \rho \cos \theta +  C\sigma |\log
              \sigma|, \label{pinned_BP_DMI} 
          \\ 
          F_{\mathrm{vol}}( m'_\sigma) - F_{\mathrm{surf}}(
          m_{\sigma,3})
          & \geq \left(\frac{3\pi^3}{8}\cos^2 \theta  -
            \frac{\pi^3}{8} \right) \rho -  C
            \sigma^{\frac{1}{4}}|\log
            \sigma|^{-\frac{1}{2}}. \label{fvolmspfsurfms3}
	\end{align}
	Furthermore, we have
	\begin{align}
          \rho^2 & \leq \frac{C}{|\log \sigma|}\label{apriori_rho}.
	\end{align}
\end{lemma}

\begin{proof}[Proof of Lemma \ref{lem:lower_bound_anisotropy_notpinned}]
  The known scaling properties of the $L^2$- and $\mathring H^1$-norms
  allow us, without loss of generality, to set $\rho= 1$.
  Additionally, we may suppose $x_0 = 0$ by translation invariance.
  By Lemma \ref{lemma:sobolev} we notice that
	\begin{align}
          \int_{\R^2} |m'|^2 \intd x \geq \frac{1}{C} \int_{\R^2} | m
          + e_3 |^2 \intd x 
	\end{align}
	for some universal constant $C>0$.
	It is thus sufficient to estimate the right-hand side from below.
	
	For $R>0$, by the inequality $|a + b|^2 \leq 2 (|a|^2 +
          |b|^2)$ for $a, b \in \R^3$ we have
	\begin{align}\label{nonpinning_initial_estimate}
          \int_{\ball{0}{R}} | m + e_3 |^2 \intd x \geq
          \int_{\ball{0}{R}} \frac{1}{2} | \phi + e_3|^2 \intd x -
          \int_{\ball{0}{R}}  | m - \phi |^2 \intd x. 
	\end{align}	
	As the in-plane components of $\Phi$ average to $0$ on
        radially symmetric sets and the out-of-plane component is
        radial (recall the definition in
          \eqref{belavin-polyakov}), the first term in the above
        estimate can be computed explicitly to give
	\begin{align}
	  \begin{split}
            \int_{\ball{0}{R}}  | \phi + e_3|^2 \intd x & = 2
            \int_{\ball{0}{R}} (1 +  \phi \cdot e_3) \intd x\\ 
            & = 2\pi R^2 \left(1 + e_3 \cdot S e_3\,
              \dashint_{\ball{0}{R}} \frac{1-|x|^2}{1+ |x|^2} \intd x
            \right).
	  \end{split}
	\end{align}
	In view of $\nu =-S e_3$, see definition \eqref{nudef}, we have
	\begin{align}
	  \begin{split}
            \int_{\ball{0}{R}} | \phi + e_3|^2 \intd x & = \pi R^2 |
            \nu + e_3|^2 - 2 \nu\cdot e_3 \, \int_{\ball{0}{R}}
            \frac{2}{|x|^2 + 1} \intd x \\
            & = \pi R^2 | \nu + e_3|^2 - 4\pi  \log(1+ R^2) \, \nu\cdot e_3\\
            & = \left(\pi R^2 - 2 \pi\log(1+ R^2) \right) | \nu +
            e_3|^2 + 4\pi \log(1+ R^2).
	  \end{split}
	\end{align}
	For $R\geq R_0$ with $R_0>0$ big enough we therefore have
	\begin{align}\label{6.1lb1}
          \int_{\ball{0}{R}}  | \phi + e_3|^2 \intd x
          & \geq
            \frac{\pi}{2}R^2
            |\nu + e_3|^2
            + 4\pi
            \log(R^2). 
	\end{align}
	
	In order to control the second term on the right-hand side of
        estimate \eqref{nonpinning_initial_estimate}, we make use of
        the fact that $y( \log(y)-1)$ for $y>0$ is the Legendre
        transformation of the exponential map $e^x$, i.e., we have the
        sharp inequality $x y \leq e^x + y(\log(y) -1)$ for $x\in \R$
        and $y >0$, see for example \cite[Chapter 3.3 and Table
        3.1]{borwein2006convex}.  For
        $\smash{{\overline L}} := \|\nabla (m-\phi)\|_{2}^{-1}$ we
        consequently get
	\begin{align}
	  \begin{split}
            \int_{\ball{0}{R}} | m - \phi |^2 \intd x & =
            \int_{\ball{0}{R}} | m - \phi |^2 |\nabla \phi |^{-2}
            |\nabla \phi |^2 \intd x \\
            & \leq \int_{\ball{0}{R}} e^{\frac{2\pi}{3} {\overline L}^2
              |m-\phi|^2} |\nabla \phi |^2 \intd x +
            \int_{\ball{0}{R}} \frac{3}{2\pi {\overline L}^2} \left[
              \log\left( \frac{3 |\nabla \phi |^{-2}}{2\pi {\overline
                  L}^2} \right) -1\right] \intd x.
		\end{split}
	\end{align}
        Applying Lemma \ref{lem:linearization_estimate}, specifically
        our version of the Moser-Trudinger inequality
        \eqref{mosertrudinger}, where Theorem
        \ref{thm:quantitative_stability} and Lemma
        \ref{lem:conformal_invariance} ensure its applicability for
        $L\geq L_0$ with $L_0>0$ sufficiently big, we see that the
        first term on the right-hand side is universally bounded.  For
        $R \geq R_0$ for $R_0>0$ big enough, there furthermore exists
        a universal constant $C' >0$ such that
        $|\nabla \phi(x)| \geq \smash{\frac{1}{C' }} |R|^{-2}$ for
        $x \in \ball{0}{R}$.  We can therefore also estimate the
        second term on the right-hand side to see
	 \begin{align}
           \int_{\ball{0}{R}}  | m - \phi |^2 \intd x 
           \leq C\left[ 1 + \frac{R^2}{{\overline L}^2} +
           \frac{R^2}{{\overline L}^2} \log\left(\frac{R^4}{{\overline
           L}^2}  \right) \right]. 
	\end{align}
	Theorem \ref{thm:quantitative_stability} allows us to write
        this in the form
	\begin{align}\label{6.1lb2}
          \int_{\ball{0}{R}}  | m - \phi |^2 \intd x
          \leq C\left[ 1 + \frac{R^2}{ L^2} + \frac{R^2}{ L^2}
          \log\left(\frac{R^4}{  L^2}  \right) \right]. 
	\end{align}
	
	Choosing $R= \eta L$ for a suitable $\eta>0$ and requiring
        $L\geq L_0$ for some $L_0>0$ sufficiently big we combine the
        two bounds \eqref{6.1lb1} and \eqref{6.1lb2} with the one in
        \eqref{nonpinning_initial_estimate} to obtain
	\begin{align}
          & \int_{\ball{0}{R}} | m + e_3 |^2 \geq C|\nu + e_3|^2 L^2
	\end{align}
	for $C >0$ universal.
\end{proof}

\begin{proof}[Proof of Lemma \ref{lem:lowerbound_dmi}]
  \textit{Step 1: Estimate the DMI term.}\\
  Without loss of generality, we may take $x_0 = 0$.  Let
  $\phi_L(x):= S \Phi_L(\rho^{-1}x)$, where $\Phi_L$ was defined in
  equation \eqref{def_profiles}.  Estimate \eqref{BP_difference}
  and the fact that $\phi$ was chosen according to Theorem
  \ref{thm:quantitative_stability} give
\begin{align}\label{phiLinhomh1}
  \int_{\R^2} | \nabla (\phi_L -m) |^2 \intd x \leq  2 \int_{\R^2}
  |\nabla (\phi_L - \phi) |^2 \intd x +2 \int_{\R^2} | \nabla (\phi
  -m) |^2 \intd x \leq   C L^{-2}. 
\end{align}

We calculate
	\begin{align}
          \int_{\R^2}m' \cdot \nabla   m_3 \intd  x = \int_{\R^2} m'
          \cdot \nabla (m_3- \phi_{L,3}) \intd x + \int_{\R^2} m' \cdot
          \nabla \phi_{L,3}\intd x. 
	\end{align}
	By Young's inequality and the estimate
        \eqref{phiLinhomh1}, the first term is bounded from above by
	\begin{align}
          \int_{\R^2} m' \cdot \nabla (m_3- \phi_{L,3}) \intd x  \leq
          \frac{1}{8}\int_{\R^2} |m'|^2 \intd x + C L^{-2}. 
	\end{align}
	To estimate the second term, note that
        $\phi_{L,3}(x) - \nu_3 = \left(S(\Phi_L(\rho^{-1}x) +
          e_3) \right)_3$ for all $x \in \R^2$ by definition
        \eqref{nudef}.  Using estimate \eqref{BP_anisotropy} to
        control the in-plane contributions and \eqref{BP_outofplane}
        together with $\Phi_3 + 1 \in L^2(\R^2;\R^3)$ to control the
        out-of-plane contributions, we have
        \begin{align}
          \label{phioutofplanelog}
          \int_{\R^2} | \phi_{L,3} - \nu_3 |^2 \intd x \leq C
          \rho^2 \log L  
        \end{align}
        for $L\geq L_0$ with $L_0>0$ big enough.  Therefore, the
        function $x \mapsto ( \phi_{L,3}(x) - \nu_3) m'(x)$ is
        integrable, and we can integrate by parts to obtain
	\begin{align}
          \int_{R^2} m' \cdot \nabla \phi_{L,3}\intd x = - \int_{R^2}
          (\phi_{L,3} - \nu_3)\nabla \cdot m' \intd x \leq  C \rho
          (\log L)^\frac{1}{2}. 
	\end{align}
In total, we obtain
\begin{align}
  2 \int_{\R^2}  m' \cdot \nabla   m_3 \intd  x \leq \frac{1}{4}
  \int_{\R^2} |m' |^2 \intd x + C \rho (\log L)^\frac{1}{2} + C L^{-2} 
\end{align}
for all $L\geq L_0$ with $L_0 > 0$ sufficiently big universal.

\textit{Step 2: Estimate the nonlocal terms.} \\
The volume charges are simply estimated by
$F_{\mathrm{vol}}(m') \geq 0$, see Lemma \ref{lemma:fourier_basic}.
For the surface charges, we exploit bilinearity of $F_{\mathrm{surf}}$
and the fact that $F_{\mathrm{surf}}$ is invariant under addition of
constants, see \eqref{nonlocal_real_space_2_extended}, to get
	\begin{align}
          F_{\mathrm{surf}}( m_3) - F_{\mathrm{surf}}( \phi_{L,3})
          &
            =F_{\mathrm{surf}}( m_3 + \phi_{L,3} +1 - \nu_3
            ,m_3 - \phi_{L,3}). 
	\end{align}
	The inequality \eqref{surf_interpolation} together with
        \eqref{phiLinhomh1} then implies
	\begin{align}\label{F_surf_difference}
          \left| F_{\mathrm{surf}}( m_3) - F_{\mathrm{surf}}(
          \phi_{L,3})\right| \leq C \left(  \| \phi_{L,3} - \nu_3
          \|_{2} + \| m_3 + 1  \|_{2} \right) L^{-1}.  
	\end{align}
	By \eqref{surf_interpolation} and \eqref{BP_excess} we furthermore get
	\begin{align}\label{F_surf_bp_tilted}
          F_{\mathrm{surf}} (\phi_{L,3})= F_{\mathrm{surf}}
          (\phi_{L,3} - \nu_3 ) \leq C \| \phi_{L,3} - \nu_3
          \|_{2} . 
	\end{align}
	Thus we can combine \eqref{F_surf_difference} and
        \eqref{F_surf_bp_tilted}, estimating the
        $\| \phi_{L,3} - \nu_3 \|_{2}$ and $\| m_3 + 1 \|_2$
        contributions by \eqref{phioutofplanelog} and
          \eqref{controlina}, respectively, and applying Young's
          inequality, to get
	\begin{align}
          F_{\mathrm{surf}}( m_3) \leq \frac{1}{4}\int_{\R^2} | m' |^2
          \intd x  + C \rho (\log L)^\frac{1}{2} + C L^{-2}, 
	\end{align}
	provided $L\geq L_0$ for $L_0>0$ sufficiently big.
\end{proof}

\begin{proof}[Proof of Lemma \ref{lem:up_to_wobbling}]
  Combining the a priori bound \eqref{apriori_dirichlet_excess} with
  the upper bound of Lemma \ref{cor:upper_bound} we see that
	\begin{align}\label{bound_for_L}
          L^{-2} = \int_{\R^2} |\nabla m_\sigma |^2 \intd x - 8\pi
          \leq 16\pi \sigma^2 
	\end{align}
	for all $\sigma > 0$, which is the desired estimate
        \eqref{apriori_L}.  Therefore, for all $\sigma < \sigma_0$
          small enough universal we may apply Lemmas
        \ref{lem:lower_bound_anisotropy_notpinned} and
        \ref{lem:lowerbound_dmi}, the latter together with the bound
        \eqref{bound_for_L} to control the $C L^{-2}$-term, to get
	\begin{align}
	  \begin{split}
            & \quad \frac{|\log\sigma|}{\sigma^2} \left(E_{\sigma,
                \lambda}(m_\sigma) - 8\pi \right)\\
            & \geq |\log\sigma| (\sigma L)^{-2} + |\log\sigma| \bigg(
            C_1 |\nu + e_3|^2 L^2 \rho^2 - C_2 \left(\rho (\log
              L)^\frac{1}{2} + \sigma^2\right) \bigg)
	  \end{split}
	\end{align}
	for two universal constants $C_1, C_2>0$.  By \eqref{gcC}, we
        may define
        $\widetilde \rho := \frac{C_2|\log\sigma| (\log
          L)^\frac{1}{2}}{\bar g(\lambda)}\rho$.  Recalling the
        definition \eqref{defVsigma}, observe that for
        $\theta_0^+$ as in the first part of Proposition
        \ref{prop:reduced_minimization} we have
        $(\widetilde \rho, \theta_0^+, L) \in V_{\sigma}$ by estimate
        \eqref{bound_for_L}. We thus get from Corollary
        \ref{cor:ansatz}, \eqref{gcC} and
        \eqref{theta_0_maximizer} that 
	\begin{align}\label{bound_for_pinning}
          |\log\sigma| (\sigma L)^{-2} + \widetilde C_1 \frac{ |\nu +
          e_3|^2 L^2}{|\log \sigma| \log L} \widetilde \rho^2 -
          g(\lambda, \theta_0^+)  \widetilde \rho \leq
          \min_{V_\sigma} \mathcal{E}_{\sigma,\lambda;K^*} + \widetilde C_2
          \sigma^\frac{1}{4} |\log\sigma|, 
	\end{align}
        for some $\widetilde C_1, \widetilde C_2 > 0$ universal.
	
	Towards a contradiction, assume that
	\begin{align}
		\widetilde C_1 |\nu + e_3|^2 L^2 \geq 16\pi \log^2 L.
	\end{align}
	Then we have for all $\sigma < \sigma_0$ small enough
        universal:
	\begin{align}
          \widetilde C_1 \frac{ |\nu + e_3|^2 L^2}{|\log \sigma| \log
          L} \widetilde \rho^2 \geq \frac{4\pi \log\left(2K^*
          L^2\right)}{|\log \sigma|} \widetilde \rho^2 
	\end{align}
	by estimate \eqref{bound_for_L}.  Recalling the definition
        \eqref{defreducedenergy} and that
        $(\widetilde \rho, \theta_0^+,L) \in V_\sigma$, we therefore
        obtain from \eqref{theta_0_maximizer} and the bound
        \eqref{bound_for_pinning} that
	\begin{align}
          \min_{V_{\sigma}} \mathcal{E}_{\sigma,\lambda;2K^*}\leq
          \mathcal{E}_{\sigma,\lambda;2 K^*}( \widetilde \rho,
          \theta_0^+, L) \leq \min_{V_\sigma}
          \mathcal{E}_{\sigma,\lambda;K^*} + \widetilde
          C_2\sigma^{\frac{1}{4}} |\log\sigma| . 
	\end{align}
	This evidently contradicts the identity
	\begin{align}
          \min_{V_\sigma} \mathcal{E}_{\sigma,\lambda;2 K^*} -
          \min_{V_{\sigma}} \mathcal{E}_{\sigma,\lambda; K^*} =
          \frac{\bar g^2 (\lambda)\log 2}{64\pi|\log\sigma|} +
          O\left(\frac{\log^2|\log\sigma| }{|\log \sigma|^2}\right) 
	\end{align}
	resulting from the expansion \eqref{asymptotics_minimal_energy}.
\end{proof}

\begin{proof}[Proof of Lemma \ref{lem:lower_bound_anisotropy_pinned}]
  \textit{Step 1: Write the problem in Fourier space.}\\
  The strategy is, essentially, to relax the unit length constraint on
  $m$ and carry out the resulting quadratic minimization in Fourier
  space.  Without loss of generality, we may assume $\rho=1$ and
  $x_0 = 0$.

By assumption, we have $m' \in L^2(\R^2;\R^2)$, together with
	\begin{align}\label{constraintmprime}
		\int_{\R^2} | \nabla (m' - (S\Phi)') |^2 \intd x \leq  {\overline L}^{-2}
	\end{align}
	with $|S e_3 - e_3|^2 \leq {\overline L}^{-1}$ for
        $S\in \operatorname{SO}(3)$.  Letting
        $ h :=\operatorname{Im}\left(\mathcal{F}m'\right) \in
        L^2(\R^2;\R^2)$, where $\F$ denotes the Fourier transform
        defined via \eqref{Fourdef}, Plancherel's identity implies
\begin{align}\label{minimizationh}
  \int_{\R^2} |m'|^2 \intd x \geq  \int_{\R^2}\left| h\right|^2
  \frac{\intd k}{(2\pi)^2}. 
\end{align}
In order to express the constraint \eqref{constraintmprime} in Fourier
space, let $S'\in \R^{2\times 2}$ be defined by
$S'_{ij} := S_{ij}$ for $i,j=1,2$. Notice that by Lemma
\ref{lemma:fourier_transform} we have that
$\mathcal{F}(\nabla \Phi_3)$ is purely imaginary, while
$\mathcal{F}(\nabla \Phi')$ is purely real. Therefore, in view of
\eqref{FDPhip} we have
$\operatorname{Re}\big(\mathcal{F}(\partial_{i} (S \Phi)_{j}(k)
\big)= \mathcal F(\partial_{i} (S' \Phi')_{j})(k) = k_{i} (S'
g)_{j}$ for $i,j=1,2$ and almost all $k\in \R^2$, where
$g : \R^2 \to \R^2$ is defined as
	\begin{align}\label{minimization_g}
	  \begin{split}
		g(k) &  := - 4\pi K_1(|k|) \frac{k}{|k|}.
	  \end{split}
	\end{align}
	Furthermore, we have
        $\mathcal{F}(\partial_{i} m'_{j})(k)=\mathrm{i}k_{i}
        \mathcal{F}(m'_{j})$ for all $i,j=1,2$ and almost all
        $k \in \R^2$, so that we obtain
        $\operatorname{Re}\left(\mathcal{F}(\partial_{i}
          m'_{j})(k)\right) = - k_{i} h_{j}$.  Only keeping
        the real parts, Plancherel's identity and the assumption
        \eqref{constraintmprime} then give
	\begin{align}\label{constraintinfourier}
          \int_{\R^2} |k|^2 \left| h +  S' g \right|^2 \frac{\intd
          k}{(2\pi)^2}  \leq \int_{\R^2} \left| \mathcal{F}\left(
          \nabla\left(m'-\left( S\Phi\right)' \right) \right)\right|^2
          \frac{\intd k}{(2\pi)^2}  \leq {\overline L}^{-2}. 
	\end{align}
	
        \textit{Step 2: Introduce the expected minimizer into the
          quadratic expressions.}\\ 
	Let $\mu > 0$ be a proxy for the Lagrange multiplier
        associated to the minimization of the right-hand side of
        \eqref{minimizationh} under the constraint
        \eqref{constraintinfourier}.  By \eqref{bessel1upper} and
        \eqref{besselinfinity} we have
        $\frac{\mu |k|^2}{1+\mu |k|^2} S'g \in L^2(\R^2;\R^2)$ and
        $\frac{ |k|}{1+\mu |k|^2} S'g \in L^2(\R^2;\R^2)$. Therefore,
        we may calculate
	\begin{align}\label{identityh}
	  \begin{split}
            \int_{\R^2}\left| h\right|^2 \frac{\intd k}{(2\pi)^2} & =
            \int_{\R^2}\left| h + \frac{\mu |k|^2}{1+\mu |k|^2}
              S'g\right|^2 \frac{\intd k}{(2\pi)^2}\\
            & \quad - 2 \int_{\R^2} \frac{\mu |k|^2}{1+\mu |k|^2} S'g
            \cdot \left( h + \frac{\mu |k|^2}{1+\mu |k|^2}
              S'g\right)  \frac{\intd k}{(2\pi)^2} \\
            & \quad + \int_{\R^2} \frac{\mu^2 |k|^4}{(1+\mu |k|^2)^2}
            \left|S'g \right|^2 \frac{\intd k}{(2\pi)^2},
	  \end{split}
	\end{align}
and
	\begin{align}\label{quadratic_constraint}
	  \begin{split}
            \int_{\R^2} |k|^2 \left| h + S' g \right|^2 \frac{\intd
              k}{(2\pi)^2} & = \int_{\R^2} |k|^2 \left| h + \frac{\mu
                |k|^2}{1+\mu |k|^2} S'g\right|^2 \frac{\intd
              k}{(2\pi)^2} \\
            & \quad + 2 \int_{\R^2} \frac{|k|^2}{1+\mu |k|^2} S'g\cdot
            \left( h + \frac{\mu |k|^2}{1+\mu |k|^2} S'g\right)
            \frac{\intd k}{(2\pi)^2}\\
            & \quad + \int_{\R^2} \frac{|k|^2}{(1+\mu |k|^2)^2}
            |S'g|^2 \frac{\intd k}{(2\pi)^2}.
	  \end{split}
	\end{align}
        Multiplying \eqref{quadratic_constraint} by $\mu>0$ and
        rearranging the terms we get by estimate
        \eqref{constraintinfourier} that
	\begin{align}
          -2 \int_{\R^2} \frac{\mu |k|^2}{1+\mu |k|^2} S'g\cdot
          \left( h + \frac{\mu |k|^2}{1+\mu |k|^2} S'g\right)
          \frac{\intd k}{(2\pi)^2} \geq \int_{\R^2} \frac{\mu
          |k|^2}{(1+\mu |k|^2)^2} |S'g|^2 \frac{\intd k}{(2\pi)^2} -
          \mu{\overline L}^{-2}. 
	\end{align}
	Plugging this into the first identity \eqref{identityh}, we
        obtain 
	\begin{align}\label{lowerboundwithS}
	  \begin{split}
            \int_{\R^2}\left| h\right|^2 \frac{\intd k}{(2\pi)^2} &
            \geq \int_{\R^2}\frac{\mu |k|^2}{1+\mu |k|^2}\left| S'g
            \right|^2 \frac{\intd k}{(2\pi)^2} - \mu {\overline
              L}^{-2}.
	  \end{split}
	\end{align}

        \textit{Step 3: Conclusion.} \\
	Since $S \in \operatorname{SO}(3)$, we have
        $(S'^T S'- \operatorname{id})_{ij} = - S_{3,i}S_{3,j}$ for
        $i,j= 1,2$.  Therefore, with $v_i := S_{3,i}$ for $i=1,2$ we
        obtain that the symmetric $2 \times 2$ matrix
        $S'^T S'- \operatorname{id}$ has the eigenvalues
        $\lambda_1:=-|v|^2$ and $\lambda_2:= 0$ with eigenvectors $v$
        and $v^\perp$, respectively.  By \eqref{Se3e3L} we have
        $|v|^2 \leq |S e_3 - e_3|^2 < {\overline L}^{-1}$, so that we
        can calculate
	\begin{align}
          \label{Spgp2}
          | S' g'  |^2 = |g'|^2 +g'\cdot (S'^T S' - \operatorname{id}) 
          g'   \geq  |g' |^2\left( 1- |\lambda_1| \right) \geq |g'|^2
          \left(1-{\overline L}^{-1}\right).  
	\end{align}
	Recalling the definition \eqref{minimization_g}, we can thus
        upgrade the estimate \eqref{lowerboundwithS} to
	\begin{align}\label{lowerboundnoS}
	  \begin{split}
            \int_{\R^2}\left| h\right|^2 \frac{\intd k}{(2\pi)^2} &
            \geq 16 \pi^2 \left( 1 - {\overline L}^{-1} \right)
            \int_{\R^2}\frac{\mu |k|^2}{1+\mu |k|^2} K_1^2(|k|)
            \frac{\intd k}{(2\pi)^2} - \mu {\overline L}^{-2}.
	  \end{split}
	\end{align}
        Then by Lemma \ref{lem:calculation_for_lower_bound} we
          can rewrite the above inequality as
          \begin{align}
            \label{hLmu}
            \int_{\R^2}\left| h\right|^2 \frac{\intd k}{(2\pi)^2}
            &
              \geq 4 \pi \left( 1 - {\overline L}^{-1} \right) \log
              \left( \frac{4 \mu}{ e^{2 \gamma + 1}} \right) 
              - \mu {\overline L}^{-2} - C (1 - {\overline L}^{-1})
              \mu^{-\frac13} 
          \end{align}
          for some $C > 0$ and all $\mu$ sufficiently large
          universal. For ${\overline L} \geq {\overline L}_0$ with
        ${\overline L}_0>0$ big enough, the right-hand side of
          \eqref{hLmu} is maximized by $\mu=4\pi {\overline L}^2$
        to the leading order in ${\overline L}^{-1} \ll
          1$. Plugging in this value of $\mu$ into \eqref{hLmu} then
          yields
	\begin{align}
          \int_{\R^2}\left| h\right|^2 \frac{\intd k}{(2\pi)^2}
          \geq 4\pi \log\left( \frac{16 \pi }{e^{2\gamma +1}}
          {\overline L}^2\right)  - 4\pi - C {\overline
          L}^{-\frac{2}{3}}
	\end{align}
        for some $C > 0$ universal, which is the desired estimate.
\end{proof}

\begin{proof}[Proof of Lemma \ref{lem:lower_bound_dmi_and_stray_pinned}]
  \textit{Step 1: Preliminary bounds on the radius.}\\
  For $\sigma \in (0, \sigma_0)$ and $\sigma_0>0$ small enough, the
  lower bound \eqref{mprime4pirho2log} of Lemma
  \ref{lem:lower_bound_anisotropy_pinned}, which is applicable due to
  Lemma \ref{lem:up_to_wobbling} and Theorem
  \ref{thm:quantitative_stability}, gives
	\begin{align}\label{apriori_L2}
          2\pi \rho^2 \log \left( L^2 \right)  \leq \int_{\R^2} |
          m'_\sigma |^2 \intd x. 
	\end{align}
	From the topological bound \eqref{topological_bound_in_lemma}
        and the estimate \eqref{apriori_anisotropy} we get
	\begin{align}
	  \begin{split}
            \frac{\sigma^2}{2} \int_{\R^2} | m'_\sigma |^2 \intd x &
            \leq \int_{\R^2} |\nabla m_\sigma|^2 \intd x - 8\pi +
            \frac{\sigma^2}{2} \int_{\R^2} | m'_\sigma |^2 \intd x
            \\
            & \leq E_{\sigma,\lambda}(m_\sigma) - 8\pi +
            \sigma^2\frac{(1+\lambda)^2}{2} \int_{\R^2} |\nabla
            m_\sigma|^2 \intd x.
	  \end{split}
	\end{align}
        Therefore, Lemma
        \ref{cor:upper_bound} and the bound
        $\int_{\R^2} |\nabla m_\sigma|^2 \intd x < 16\pi$ resulting
        from $m_\sigma \in \mathcal{A}$, see definition \eqref{defa},
        imply
	\begin{align}\label{apriori_new}
	  \begin{split}
            \int_{\R^2} | m'_\sigma |^2 \intd x \leq C
	  \end{split}
	\end{align}
        for some $C > 0$ universal.  In particular, from the
        estimates \eqref{apriori_L2} and \eqref{apriori_L} we get the
        bound \eqref{apriori_rho}.
	
        \textit{Step 2: Estimate the DMI term.}\\
        Without loss of generality, we may assume $x_0=0$.  By Lemma
        \ref{lem:up_to_wobbling} and the fact that $\nu = -S e_3$, see
        definition \eqref{nudef}, we obtain
        $|Se_3 - e_3|^2 \leq C L^{-2}\log^2 L $ for $\sigma_0$ small
        enough.  On the other hand, by the Euler rotation theorem the
        matrix $S$ admits a representation $S = R S_\theta$, where
        $S_\theta$ is defined by \eqref{def_rotation} for some
        $\theta \in [-\pi,\pi)$ and $R \in \operatorname{SO}(3)$ is a
        composition of rotations around the $x_1$- and $x_2$-axes. It
        is not difficult to see that
          \begin{align}
            |R - \operatorname{id}| \leq C |R e_3 - e_3| = C |S e_3 -
            e_3| \leq C' L^{-1} \log L
          \end{align}
          for some universal $C, C' > 0$. Therefore, by the properties
          of the Frobenius norm we get
	\begin{align}
          |S-S_\theta|^2  = |R - \operatorname{id}|^2 \leq
          CL^{-2}\log^2 L  
	\end{align}
        for some universal $C > 0$. Together with Theorem
      \ref{thm:quantitative_stability}, Lemma \ref{lem:up_to_wobbling}
      and definition \eqref{L_convention}, we deduce that the
      function $\phi(x) := S_\theta \Phi(\rho^{-1} x)$ for $x\in \R^2$
      satisfies
	\begin{align}\label{suboptimal_stability}
          \int_{\R^2} |\nabla ( m_\sigma - \phi)|^2 \intd x \leq C
          \int_{\R^2} |\nabla ( m_\sigma - S \Phi(\rho^{-1} x))|^2
          \intd x + C |S-S_\theta|^2 \leq C \, \frac{\log^2 L}{L^2} 
	\end{align}
        for some $C > 0$ and $\sigma_0 > 0$ small enough, both
          universal, which is estimate \eqref{distance_to_pinned_BP}.
	
	By the identity \eqref{BP_DMI_infty} we have
	\begin{align}
          \int_{\R^2} 2  \phi' \cdot \nabla   \phi_3 \intd  x = 8\pi
          \rho \cos \theta. 
	\end{align}
	Therefore, the bound \eqref{pinned_BP_DMI} for the DMI term
        follows once we control
	\begin{align}\label{split_difference_DMI}
	  \begin{split}
            & \quad \int_{\R^2} \left( m'_\sigma \cdot \nabla
              m_{\sigma,3}- \phi' \cdot \nabla \phi_3 \right)
            \intd  x\\
            & =\int_{\R^2} m'_\sigma \cdot \nabla (m_{\sigma,3}-
            \phi_3 ) \intd x + \int_{\R^2} (m'_\sigma- \phi'
            )\cdot \nabla  ( \phi_3 + 1) \intd  x  \\
            & = \int_{\R^2} m'_\sigma \cdot \nabla (m_{\sigma,3} -
            \phi_3 ) \intd x - \int_{\R^2} ( \phi_3 + 1) \nabla \cdot
            (m'_\sigma - \phi' ) \intd x,
	  \end{split}
	\end{align}
	where the decay of $\phi_3+1$ at infinity is sufficiently
        strong to erase the boundary term in the integration by parts.
        By explicit calculation and \eqref{apriori_rho}, we have
	\begin{align}\label{phi3inL2}
          \int_{\R^2} |\phi_3 +1 |^2\intd x \leq C\rho^2 \leq 
          \frac{C}{|\log \sigma |}. 
	\end{align}
	Consequently, the Cauchy-Schwarz inequality applied to the
        right-hand side of \eqref{split_difference_DMI} and the
        estimates \eqref{apriori_new}, \eqref{phi3inL2}, and
        \eqref{suboptimal_stability} imply
	\begin{align}
          \left|\int_{\R^2} \left(  \phi' \cdot \nabla    \phi_3 -
          m'_\sigma \cdot \nabla   m_{\sigma,3} \right) \intd  x
          \right| \leq C \, \frac{\log L}{ L}. 
	\end{align}
	The bound \eqref{apriori_L} then gives the desired estimate
        \eqref{pinned_BP_DMI} for the DMI term.
	
        \textit{Step 3: Estimate the stray field terms.}\\
	We consider $\phi_{L}:= S_\theta \Phi_{L}(\rho^{-1}x)$, with
        $\Phi_L$ as defined in equation \eqref{def_profiles}, and note
        that we still have
	\begin{align}\label{still_suboptimal}
          \int_{\R^2} |\nabla (m_\sigma - \phi_L) |^2 \intd x
          \leq  C \, \frac{\log^2 L}{L^2}
	\end{align}
	by the bounds \eqref{suboptimal_stability} and \eqref{BP_difference}.
	In Lemma \ref{lem:approximation_physical} we also already computed
	\begin{align}
	  \begin{split}
            \label{fvolphiLfsurfphiL}
            F_{\mathrm{vol}}( \phi'_{L}) - F_{\mathrm{surf}}(
            \phi_{L,3}) & \geq \left(\frac{3\pi^3}{8}\cos^2 \theta -
              \frac{\pi^3}{8} \right) \rho - C \rho L^{-\frac{1}{4}}\\
            & \geq \left(\frac{3\pi^3}{8}\cos^2 \theta -
              \frac{\pi^3}{8} \right) \rho - C \sigma^\frac{1}{4}
            |\log \sigma|^{-\frac{1}{2}},
	  \end{split}
	\end{align}
	where in the last step we used estimates \eqref{apriori_L} and
        \eqref{apriori_rho}. 
	
	As the nonlocal terms are bilinear, we have
	\begin{align}
          F_{\mathrm{vol}}( m'_\sigma) - F_{\mathrm{vol}}( \phi'_{L})
          = F_{\mathrm{vol}}( m'_\sigma +\phi'_L ,m'_\sigma - \phi'_L)
	\end{align}
	and the interpolation inequality $\eqref{vol_interpolation}$
        for $p=2$ together with estimate \eqref{still_suboptimal}
        imply
	\begin{align}
          \left| F_{\mathrm{vol}}( m'_\sigma) - F_{\mathrm{vol}}(
          \phi'_{L})\right| \leq C \left(  \| \phi'_{L} \|_{2} + \|
          m'_\sigma  \|_{2} \right)\frac{\log L}{ L}. 
	\end{align}
	The fact that
        $ \| \phi'_{L} \|_{2} \leq C \rho (\log L)^\frac{1}{2}$, see
        \eqref{BP_anisotropy}, together with the estimates
        \eqref{apriori_new}, \eqref{apriori_rho} and \eqref{apriori_L}
        thus gives
	\begin{align}
          \left| F_{\mathrm{vol}}( m'_\sigma) - F_{\mathrm{vol}}(
          \phi'_{L})\right| \leq C\sigma  |\log\sigma|. 
	\end{align}
	A similar argument exploiting the estimates
        \eqref{surf_interpolation}, and \eqref{phi3inL2}, as well as
        Lemma \ref{lemma:sobolev} gives
	\begin{align}
          \left| F_{\mathrm{surf}}( m_{\sigma,3}) - F_{\mathrm{surf}}(
          \phi_{L,3})\right| \leq C\sigma |\log\sigma|.
	\end{align}
	Combining the last two estimates with
          \eqref{fvolphiLfsurfphiL} yields \eqref{fvolmspfsurfms3}.
\end{proof}

\subsection{Convergence to shrinking Belavin-Polyakov profiles via
  stability of the reduced energy
  \texorpdfstring{$\mathcal{E}_{\sigma,\lambda;K}$}.}

Having completed the preparatory work in the form of the previously 
presented statements, we now proceed to prove Theorem
\ref{thm:convergence}.

\begin{proof}[Proof of Theorem \ref{thm:convergence}]
  By Theorem \ref{thm:quantitative_stability} and definition
  \eqref{L_convention}, there exist $S \in \operatorname{SO}(3)$,
  $\rho_\sigma>0$ and $x_\sigma \in \R^2$ such that
	\begin{align}
          \int_{\R^2} \left| \nabla \left( m_\sigma(x) - S
          \Phi(\rho_\sigma^{-1}(x-x_\sigma)) \right) \right|^2 \intd x
          \leq C L^{-2}. 
	\end{align}
	Without loss of generality, we choose $x_\sigma=0$.  For
        $\eps>0$ to be chosen sufficiently small later and for
        $\sigma \in (0,\sigma_0)$ for $\sigma_0>0$ small enough
        depending only on $\eps$, we combine the bound
        \eqref{apriori_L} and the local version of the stability
        result in Lemma \ref{lem:local_nonlinear_stability} to improve
        the above estimate to
	\begin{align}
          \int_{\R^2} \left| \nabla \left( m_\sigma(x) - S
          \Phi(\rho_\sigma^{-1}x) \right) \right|^2 \intd x \leq
          \left( \frac{3}{2} + \eps \right) L^{-2}. 
	\end{align}
	Existence of $\theta_\sigma \in [-\pi,\pi)$ with
	\begin{align}\label{suboptimalestimate}
          \int_{\R^2} |\nabla ( m_{\sigma}(x) - S_{\theta_\sigma}
          \Phi(\rho_\sigma^{-1}x)) |^2 \intd x \leq  C \frac{\log^2
          L}{ L^{2}} 
	\end{align}
        follows from Lemma \ref{lem:lower_bound_dmi_and_stray_pinned},
        and
        $(|\log \sigma| \rho_\sigma,\theta_\sigma,L) \in
        V_\sigma$ is a result of estimate \eqref{apriori_L}.
	 
	Recalling the definition \eqref{defreducedenergy} of
        $\mathcal{E}_{\sigma,\lambda;K}$, for
        ${\overline L} := \left( \frac{3}{2}+\eps
        \right)^{-\frac{1}{2}} L$ and
        $K := \left(\frac{3}{2} + \eps \right)^{-1} K^* $ we
          have by Lemmas \ref{lem:up_to_wobbling},
        \ref{lem:lower_bound_anisotropy_pinned} and
        \ref{lem:lower_bound_dmi_and_stray_pinned}:
	 \begin{align}\label{energy_lower_bound}
           \mathcal{E}_{\sigma,\lambda;K}(|\log\sigma|\rho_\sigma,\theta_\sigma,L) 
           \leq \frac{|\log\sigma|}{\sigma^2}\left(
           E_{\sigma,\lambda}(m_\sigma) - 8\pi\right) +
           C\sigma^{\frac{1}{4}}  |\log\sigma|^\frac{1}{2}. 
	\end{align}
	Corollary \ref{cor:ansatz} gives
	\begin{align}\label{energy_upper_bound}
          \frac{|\log\sigma|}{\sigma^2}\left(
          E_{\sigma,\lambda}(m_\sigma) - 8\pi\right)  \leq
          \min_{V_\sigma} \mathcal{E}_{\sigma,\lambda;K^*}
          +C\sigma^{\frac{1}{4}}  |\log\sigma|. 
	\end{align}
	For $\eps \leq \frac{1}{2}$ we have $K \geq \frac{1}{2}K^*$,
        so that the expansion \eqref{asymptotics_minimal_energy}
        implies
	\begin{align}\label{difference_in_bounds}
          \min_{V_\sigma} \mathcal{E}_{\sigma,\lambda;K^*}  \leq
          \min_{V_\sigma} \mathcal{E}_{\sigma,\lambda;K} + \frac{\bar
          g^2(\lambda)}{64\pi} \frac{\log\left(\frac{K^*}{K}
          \right)}{|\log\sigma|} +  C \, \frac{\log^2|\log \sigma
          |}{|\log \sigma |^2}. 
	\end{align}
        Concatenating the estimates \eqref{energy_lower_bound},
        \eqref{energy_upper_bound} and
        \eqref{difference_in_bounds}, we get for
        $\sigma \in (0,\sigma_0)$ for $\sigma_0>0$ sufficiently small
        that
	\begin{align}
	  \begin{split}
            \mathcal{E}_{\sigma,\lambda;K}(|\log
            \sigma|\rho_\sigma,\theta_\sigma,L) \leq \min_{V_\sigma}
            \left(\mathcal{E}_{\sigma,\lambda;K}\right) + \frac{\bar
              g^{2}(\lambda)}{64\pi} \frac{\log\left(\frac{3}{2} +\eps
              \right)}{|\log\sigma|} + C \, \frac{\log^2|\log \sigma
              |}{|\log \sigma |^2},
	  \end{split}
	\end{align}
	where we also used the definition of $K$.
	
        Since $\log\frac{3}{2} < 1$, for $\eps>0$ and $\sigma_0>0$
          small enough universal we deduce
        \begin{align}
	  \begin{split}
            \mathcal{E}_{\sigma,\lambda;K}(|\log
            \sigma|\rho_\sigma,\theta_\sigma,L) \leq \min_{V_\sigma}
            \left(\mathcal{E}_{\sigma,\lambda;K}\right) + \frac{\bar
              g^{2}(\lambda)}{64\pi |\log\sigma|}.
	  \end{split}
	\end{align}
	Consequently, part 2 of Proposition
        \ref{prop:reduced_minimization} then implies the desired
        convergences for $\rho_\sigma$ and $\theta_\sigma$.
        Furthermore, the bounds \eqref{suboptimalestimate} and
        \eqref{stability_L} give
	\begin{align*}
          & \int_{\R^2} |\nabla ( m_{\sigma}(x) -
            S_{\theta_\sigma} \Phi(\rho_\sigma^{-1}x)) |^2 \intd x \leq
            C \sigma^2. 
	\end{align*}
	Finally, the estimate
	\begin{align}
          \left| \frac{|\log\sigma|}{\sigma^2}\left(
          E_{\sigma,\lambda}(m_\sigma) - 8\pi\right) - \left(-
          \frac{\bar g^2 (\lambda)}{32\pi} + \frac{ \bar g^2
          (\lambda)}{32\pi}\frac{\log|\log\sigma |}{|\log \sigma|}
          \right)  \right|   \leq \frac{C}{|\log\sigma|} 
	\end{align}
	follows from estimates \eqref{energy_lower_bound},
        \eqref{energy_upper_bound} and the expansion
        \eqref{asymptotics_minimal_energy} of Proposition
        \ref{prop:reduced_minimization}.
      \end{proof}

      \paragraph{Acknowledgements}
   A.\ B.-M. wishes to acknowledge support from DARPA TEE program
   through grant MIPR\# HR0011831554. The work of C.\ B.\ M. and T.\
   M.\ S. was supported, in part, by NSF via grants DMS-1614948 and
   DMS-1908709.

\appendix

\section{Appendix}\label{sec:appendix}

Here, we first provide a concise set-up of the differential geometry
necessary for our argument.  In a second section, we prove the
topological bound \eqref{topological_bound_intro} and prove that all
extremizers are in fact Belavin-Polyakov profiles.  Both sections are
included for the convenience of readers who may be unfamiliar with the
presented material, and we do not claim originality of the definitions
and results.  Finally, we will present a number of calculations
involving Bessel functions.

\subsection{Sobolev spaces on the sphere}\label{sec:diff_geo}

Let $u: \Sph^2 \to \R^n$ for $n\geq 1$ be a smooth map, which we may
extend to a smooth map on $\R^3\setminus{\{0\}}$ by setting
$U(x):= u\left(\frac{x}{|x|}\right)$.  Following \cite[Definition
7.25]{Ambrosio}, we consider its gradient
\begin{align}\label{tangential_differential}
  \nabla u(y) := (\partial_{\tau_1} U)(y) \tau_1(y) +
  (\partial_{\tau_2} U)(y) \tau_2(y) 
\end{align}
for $y \in \Sph^2$ and $\{\tau_1(y),\tau_2(y)\}$ an orthonormal basis
of the tangent space $T_y \Sph^2 := \{v \in \R^3: v \cdot y = 0
\}$. Due to the chain rule in $\R^n$, we recover the standard notion
of gradient in Riemannian geometry.  We will also need the tangential
divergence for smooth functions $\xi :\Sph^2 \to \R^3$, for
$y\in \Sph^2$ defined as
\begin{align}\label{tangential_divergence}
  \nabla \cdot \xi(y) := (\partial_{\tau_1} \Xi )(y) \cdot \tau_1(y)
  +(\partial_{\tau_2} \Xi) (y)\cdot \tau_2(y) , 
\end{align}
where again $\Xi(x):= \xi \left(\frac{x}{|x|}\right)$ for
$x \in \R^3\setminus \{ 0 \}$ and $\{\tau_1(y),\tau_2(y)\}$ is an
orthonormal basis of $T_y \Sph^2$, see \cite[Definition 7.27 and
Remark 7.28]{Ambrosio}.  The Laplace-Beltrami operator for a smooth
map $u: \Sph^2 \to \R$ then is $\Delta u := \nabla \cdot \nabla u$,
see \cite[(2.1.16)]{jostriemannian}.

We define the space $H^1(\Sph^2)$ as the completion of
$C^\infty(\Sph^2)$ with respect to the norm
\begin{align}
  \left\| u \right\|_{H^1(\Sph^2)} := \left(\int_{\Sph^2}\left( |
  \nabla u|^2 + |u|^2 \right) \intd \Hd^2 \right)^\frac{1}{2}.
\end{align}
  Let $H^1(\Sph^2;\R^3)$ be
defined analogously for $\R^3$-valued maps and set
\begin{align} 
  H^1(\Sph^2;\Sph^2) & := \left\{ \widetilde m \in H^1(\Sph^2;\R^3) :
                       \widetilde m(y) \in \Sph^2 \text{ for }
                       \Hd^2\text{-a.e.\ } y \in \Sph^2 \right\},\\ 
  H^1(\Sph^2;T\Sph^2) & := \left\{ \xi \in H^1(\Sph^2;\R^3) : \xi(y)
                        \in T_y\Sph^2 \text{ for } \Hd^2\text{-a.e.\ }
                        y \in \Sph^2
                        \right\},\label{def:sobolev_tangent_field} 
\end{align}
where $ T\Sph^2 := \bigcup_{y \in \Sph^2} \{y\} \times T_y\Sph^2 $ is
the tangent bundle of $\Sph^2$.  The weak gradient $\nabla u$ for
$u \in H^1(\Sph^2)$ and weak divergence $\nabla \cdot \xi$ for
$\xi \in H^1(\Sph^2; \R^3)$ exist as measurable maps characterized by
the following integration-by-parts formula:

\begin{lemma}\label{lem:integration_by_parts_manifolds}
	Let $u \in H^1(\Sph^2)$ and $\xi \in H^1(\Sph^2;\R^3)$.
	Then we have 
	\begin{align}\label{manifold_ibp}
          \int_{\Sph^2} \xi\cdot \nabla u  \intd \Hd^2(y) =
          \int_{\Sph^2} \left(2 u\, \xi \cdot y - u \nabla \cdot \xi
          \right)\intd \Hd^2(y), 
	\end{align}
	and this identity determines $\nabla u$ and $\nabla \cdot \xi$
        up to sets of $\Hd^2$-measure zero.  Furthermore, for smooth
        maps $\zeta, \xi : \Sph^2 \to \R$ we have
	\begin{align}\label{manifold_tangentialbeltrami}
          \int_{\Sph^2} \nabla \xi \cdot \nabla \zeta  \intd \Hd^2 = 
          - \int_{\Sph^2}  \zeta  \Delta \xi \intd \Hd^2.
	\end{align}
\end{lemma}

We furthermore remark that, following Brezis and Nirenberg
\cite{brezis1995degree}, we can define the Brouwer degree for
maps $\widetilde m \in H^1(\Sph^2;\Sph^2)$, and even for functions of
vanishing oscillation, in the following way: For $y\in \Sph^2$, let
$y \mapsto (\tau_1(y),\tau_2(y)) $ be an orthonormal frame of
$T_y\Sph^2$ which is smooth except in a single point.  For maps
$\widetilde m \in C^\infty(\Sph^2;\Sph^2)$ we use the integral formula
(see also definition \eqref{def_degree_on_sphere})
\begin{align}\label{representation_degree}
  \mathcal{N}_{\Sph^2}(\widetilde m) = \frac{1}{4\pi} \int_{\Sph^2}
  \det \left(\nabla \widetilde m\right) \intd \mathcal{H}^2, 
\end{align}
where for $y\in \Sph^2$ we define
$\det (\nabla \widetilde m (y)) := \det M(y)$ with
$M_{ij}(y) := \tau_j(\widetilde m(y)) \cdot [( \tau_i(y) \cdot \nabla)
\widetilde m(y)]$ for $i,j =1,2$ to be the determinant of the linear
map $v \mapsto (v\cdot \nabla) \widetilde m$ from $T_y\Sph^2$ to
$T_{\widetilde m(y)}\Sph^2$, expressed in the ordered bases
$(\tau_1(y),\tau_2(y))$ and
$(\tau_1(\widetilde m(y)),\tau_2(\widetilde m(y)))$.  It can be seen
that this definition is independent of the frame $(\tau_1,\tau_2)$; in
fact, for certain choices of $\tau_1$ and $\tau_2$ this is part
of the proof of the representation
\begin{align}
  \mathcal{N}(m) =\mathcal{N}_{\Sph^2}\left(m\circ\phi^{-1} \right),
\end{align}
for any $\phi \in \mathcal B$, found in Lemma
\ref{lem:conformal_invariance}.  The degree can then be extended as a
continuous map to $H^1(\Sph^2;\Sph^2)$ by approximation with smooth
maps provided by a result of Schoen and Uhlenbeck
\cite{schoen1983boundary}.  In particular, the above representation
\eqref{representation_degree} holds true, as the $2\times 2$
determinant is a quadratic function, and thus the integral is
continuous in the strong $H^1$-topology, see also \cite[Property
4]{brezis1995degree}.

Using the above definitions, we describe how the various quantities
behave under reparametrization by $\phi^{-1}$ with
$\phi \in \mathcal{B}$.  In particular, we prove that the harmonic map
problem is invariant under this operation.

\begin{lemma}\label{lem:conformal_invariance}
  Let $\phi \in \mathcal{B}$ and let $u: \R^2 \to \R$ be measurable.
  Then the map $x \mapsto u(x) |\nabla \phi(x)|^2$ is integrable on
  $\R^2$ if and only if $u \circ \phi^{-1}$ is integrable on $\Sph^2$,
  and we have
	\begin{align}\label{zero_order_transformation}
          \int_{\R^2} u |\nabla \phi|^2 \intd x = 2 \int_{\Sph^2} u
          \circ \phi^{-1} \intd \Hd^2. 
	\end{align}
	Furthermore, we have $ u\in H^{1}_{\mathrm{w}}(\R^2)$, where
        the space $H^{1}_{\mathrm{w}}(\R^2)$ is defined in
        \eqref{def_weighted_sobolev}, if and only if
        $u \circ \phi^{-1} \in H^1(\Sph^2)$, and for every
        $u, v \in H^{1}_{\mathrm{w}}(\R^2)$ there holds
	\begin{align}\label{conformal_invariance}
          \int_{\R^2} \nabla u \cdot \nabla v \, \intd x
          &  = \int_{\Sph^2} 
            \nabla (u \circ
            \phi^{-1})  \cdot \nabla (v \circ
            \phi^{-1}) \, \intd
            \Hd^2. 
	\end{align}
      We also have $m\in \mathring H^{1}(\R^2;\Sph^2)$ if and only if
      \mbox{$\widetilde m := m \circ \phi^{-1} \in
        H^1(\Sph^2;\Sph^2)$,} in which case we additionally have
      $\mathcal{N}(m) = \mathcal{N}_{\Sph^2}\left(\widetilde m \right)
      $.  In particular, we have that $m\in \mathcal{C}$ is a
      minimizer of $F$ if and only if
      $\widetilde m \in \mathcal{C}_{\Sph^2}$ is a minimizer of
      $F_{\Sph^2}$.
\end{lemma}

\begin{proof}[Proof of Lemma \ref{lem:integration_by_parts_manifolds}]
  The fact that $\nabla u$ and $\nabla \cdot \xi$ are determined up to
  sets of $\Hd^2$-measure zero is a standard fact in analogy to
  uniqueness of weak derivatives of functions defined in the
    Euclidean space.  By approximation, it is sufficient to prove the
  formula for smooth functions $u$ and $\xi$.  Using the definitions
  \eqref{tangential_differential} and \eqref{tangential_divergence} it
  is straightforward to check the identity
	\begin{align}
		\nabla \cdot (u\xi) = \nabla u \cdot \xi + u \nabla \cdot \xi.
	\end{align}
	We therefore have
	\begin{align}
          \int_{\Sph^2} \xi\cdot \nabla u  \intd \Hd^2
          = \int_{\Sph^2}\left( \nabla \cdot (u\xi) - u\nabla \cdot
          \xi \right)\intd \Hd^2. 
	\end{align}
	By the divergence theorem on manifolds \cite[Theorem
        7.34]{Ambrosio} we have
	\begin{align}
          \int_{\Sph^2} \nabla \cdot (u \, \xi) \intd
          \Hd^2 =  \int_{\Sph^2} u \, \xi \cdot \left(
          (\nabla \cdot y) y\right) \intd \Hd^2(y) = \int_{\Sph^2} 2 u
          \, \xi \cdot y \intd \Hd^2(y), 
	\end{align}
	where $ - (\nabla \cdot y) y =-2 y$ has the significance of
        being the mean curvature vector at $y \in \Sph^2$, see
        \cite[Definition 7.32]{Ambrosio}.  This proves the identity
        \eqref{manifold_ibp}, from which the formula
        \eqref{manifold_tangentialbeltrami} easily follows.
\end{proof}

\begin{proof}[Proof of Lemma \ref{lem:conformal_invariance}]
	For all $x\in \R^2$ we have the identities
	\begin{align}
          \partial_i \phi(x) \cdot \partial_i \phi(x)
          & =
            \frac{1}{2}|\nabla
            \phi(x) |^2>0,
          \\ 
          \partial_1 \phi(x) \cdot \partial_2 \phi(x) & = 0
	\end{align}
	for $i=1,2$, and thus the map
        $x \mapsto \left(\frac{\sqrt{2}}{|\nabla
            \phi(x)|}\partial_1\phi(x),\frac{\sqrt{2}}{|\nabla
            \phi(x)|}\partial_2\phi(x)\right)$ provides a smooth
        orthonormal frame for $T_{\phi(x)}\Sph^2$ for $x\in \R^2$.
        Equation \eqref{zero_order_transformation} and the
        corresponding equivalence are straightforward results of the
        area formula \cite[Theorem 2.71]{Ambrosio} and the fact that
        $\frac{1}{2}|\nabla \phi|^2 = (\det \nabla \phi^T \nabla
        \phi)^\frac{1}{2}$ is the Jacobian of $\phi$.
	
	Let $u,v : \R^2 \to \R$ be smooth functions and let $x\in
        \R^2$.  The chain rule implies
	\begin{align}\label{chainrule}
          \partial_i u (x) = \partial_i \phi(x) \cdot \left( \nabla
          \left( u \circ \phi^{-1}\right)\left(\phi(x)\right) \right) 
	\end{align}
	for $i = 1,2$. As a result, we obtain
	\begin{align}\label{transform_of_gradients}
          \nabla u(x) \cdot \nabla v(x)  = \frac{1}{2}|\nabla
          \phi(x)|^2 \, \nabla 
          \left( u \circ \phi^{-1}\right) (\phi(x)) \cdot \nabla
          \left(  v \circ 
          \phi^{-1} \right)(\phi(x)).
	\end{align}
	Since smooth functions are dense with respect to the
        $\mathring{H}^1$-topology in both spaces
        $H^{1}_{\mathrm{w}}(\R^2)$ and $H^1(\Sph^2)$, as can be easily
        seen via convolutions, we obtain equation
        \eqref{conformal_invariance}.  For
        $ m\in \mathring H^1(\R^2;\Sph^2)$ we thus have
        $\widetilde m := m \circ \phi^{-1} \in H^1(\Sph^2;\Sph^2)$.
	
	In order to show
        $\mathcal{N}(m) = \mathcal{N}_{\Sph^2}(\widetilde m)$, we
        first define the orthogonal frame
	\begin{align}
          \label{tau12}
          (\tau_1(y), \tau_2(y) ) := \left[ \left(
          \frac{\sqrt{2}}{|\nabla \phi|}\partial_1\phi
          ,\frac{\sqrt{2}}{|\nabla \phi|}\partial_2\phi \right)\circ
          \phi^{-1} \right](y) 
	\end{align}
	for $y\in \Sph^2$, which is smooth except in the single point
        $\nu:= \lim_{|x| \to \infty} \phi(x)$.  We may,
        therefore, calculate
	\begin{align}
          \mathcal{N}(m) & =  \frac{1}{4\pi}\int_{\R^2} m \cdot (
                           \partial_1 m \times \partial_2 m) \intd x
          \\ 
                         & = \frac{1}{4\pi} \int_{\Sph^2} \widetilde m
                           \cdot \left[ (\tau_1(y) \cdot \nabla)
                           \widetilde m \times (\tau_2(y) \cdot
                           \nabla) \widetilde m  \right] \intd \Hd^2
                           (y) 
	\end{align}
	due to the chain rule \eqref{chainrule}, the area formula and
        the fact that $\frac{1}{2}|\nabla \phi|^2$ is the
        Jacobian of $\phi$.  For almost all
        $y \in \widetilde m^{-1}(\nu)$ we have $\nabla m(y) = 0$ by
        standard statements about weak derivatives.  Therefore, for
        $i=1,2$ we can almost everywhere express
        $(\tau_i \cdot \nabla )\widetilde m(y)$ in the basis
        $\{\tau_1(\widetilde m(y)),\tau_2(\widetilde
        m(y))\}$ to get
	\begin{align}
	  \begin{split}
            \mathcal{N}(m) & = \frac{1}{4\pi}\int_{\Sph^2}
            \widetilde m\cdot \left( \tau_1(\widetilde m) \times 
                \tau_2 (\widetilde m) \right) \det \left(\nabla
                \widetilde m\right) \intd \Hd^2 =
            \mathcal{N}_{\Sph^2}(\widetilde m)
	  \end{split}
	\end{align}
	by virtue of $z \cdot (\tau_1(z) \times \tau_2 (z)) = 1$ for
        all $z \in \Sph^2\setminus\{\nu\}$ according to \eqref{tau12}.
\end{proof}

\subsection{The topological bound and energy minimizing harmonic maps
  of degree 1}

In this section, we prove the topological bound
\eqref{topological_bound_intro} and characterize the corresponding
minimizers for the convenience of the reader. The following statement
is an amalgam of results due to Belavin and Polyakov
\cite{belavin1975metastable}, Lemaire \cite{lemaire1978applications}
and Wood \cite{wood1974harmonic}, see the discussion in Section
\ref{sec:introduction}. Our approach below is to reduce the problem to
that of the solutions of an H-system treated by Brezis and Coron
\cite{brezis1985convergence}.

\begin{lemma}\label{lem:topological_bound}
	For all $m \in \mathring H^1(\R^2;\Sph^2)$ we have
	\begin{align}\label{completed_square}
          |\nabla m|^2 \pm 2 m\cdot(\partial_1 m \times
          \partial_2 m)  = | \partial_1 m \mp m \times \partial_2 m
          |^2 \geq 0, 
	\end{align}
	almost everywhere, as well as
	\begin{align}\label{topological_bound_in_lemma}
          \int_{\R^2} |\nabla m |^2 \intd x \geq 8\pi \left|
          \mathcal{N}(m) \right|. 
	\end{align}
	The functions with $\mathcal{N}=1$ achieving equality, i.e.,
        energy minimizing harmonic maps of degree 1, are given by the
        set of Belavin-Polyakov profiles $\mathcal{B}$, see definition
        \eqref{moduli_space}.  We furthermore have the representation
	\begin{align}\label{representation_moebius}
          \mathcal{B}_{\Sph^2} = \left\{ \Phi \circ f \circ \Phi^{-1}:
          f(x) := \frac{ax+ b}{cx +d} \text{ for } a, b, c, d \in
          \mathbb{C} \text{ with } ad-bc\neq 0 \right\} 
	\end{align}
	for the set $\mathcal{B}_{\Sph^2}$ of minimizing harmonic maps
        of degree 1 from $\Sph^2$ to itself, see definition
        \eqref{def_b_on_sphere}.
\end{lemma}

We also briefly state a version of \cite[Lemma 9]{li2018stability} in
our setting relating the energy excess to the Hamiltonian, see Section
\ref{sec:spectral_gap}, which will come in handy a number of times.

\begin{lemma}[{\cite[Lemma 9]{li2018stability}}]
\label{lem:relation_excess_hamiltonian} 
For $m\in \mathring H^1(\R^2;\Sph^2)$ and $\phi \in \mathcal{B}$ we
have the identity
	\begin{align}
          F(m) - 8\pi =  \int_{\R^2}
          \left( | \nabla (m-\phi)|^2 - (m-\phi)^2 |\nabla \phi |^2
          \right) \intd x. 
	\end{align}
\end{lemma}

\begin{proof}[Proof of Lemma \ref{lem:topological_bound}]
  The inequality \eqref{completed_square} is a result of completing
  the square, and the topological bound
  \eqref{topological_bound_in_lemma} then follows by integration.
	
  Let $\phi\in \mathring H^1(\R^2;\Sph^2)$ be such that
  $\mathcal{N}(\phi)=1$ and
	\begin{align}\label{dirichlet_integral_of_phi}
		\int_{\R^2} \left|\nabla \phi \right|^2 \intd x = 8\pi.
	\end{align}
	Then equation \eqref{completed_square} implies
        $\partial_1 \phi = - \phi \times \partial_2 \phi$ almost
        everywhere.  Together with the fact that
        $\phi \cdot \partial_i \phi =0$ for $i=1,2$ almost everywhere,
        we also have
        $\phi \times \partial_1\phi = - \phi \times \left(\phi \times
          \partial_2 \phi \right) =\partial_2\phi$ and
	\begin{align}
          2 \partial_1\phi \times \partial_2\phi =  \partial_1\phi
          \times ( \phi \times \partial_1 \phi) -  ( \phi \times
          \partial_2\phi) \times \partial_2 \phi = |\nabla \phi|^2
          \phi. 
	\end{align}
	Because $\phi$ is evidently an energy minimizing harmonic map,
       it satisfies \eqref{harmonic_map_equation}
        distributionally, and thus the map $\widetilde \phi := - \phi$
        satisfies
	\begin{align}
          \Delta \widetilde \phi = 2 \partial_1 \phi \times
          \partial_2 \phi, \qquad \qquad
          \int_{\R^2} \left|\nabla \widetilde \phi \right|^2 \intd x = 8\pi.
	\end{align}
	Thus \cite[Lemma A.1]{brezis1985convergence} implies for
        almost all $x\in \mathbb{C}$ that
	\begin{align}
          \widetilde \phi(x) = \Phi_n \left(\frac{P(x)}{Q(x)} \right)
	\end{align}
	for complex polynomials $P$ and $Q$ of degree 1 such that
        $\frac{P}{Q}$ is irreducible, where
        $\Phi_n(x) := -\Phi(x)$ for $x\in \mathbb{C}$ and
        $\Phi_n(\infty) := e_3$ is the stereographic projection with
        respect to the north pole and division by zero is taken to
        evaluate to infinity.  Therefore, we get the representation
	\begin{align}
		\phi(x) = \Phi\left(\frac{P(x)}{Q(x)} \right).
	\end{align}
	for the smooth representative of $\phi$.  Let
        $a, b, c ,d \in \mathbb{C}$ such that $P(x) = a x + b$ and
        $Q(x) = c x + d$ for $x \in \mathbb{C}$.  As $P$ and $Q$ are
        irreducible, they must be linearly independent polynomials.
        Consequently, we have
	\begin{align}
		ad - bc = \det \begin{pmatrix}
			a & b\\
			c & d
		\end{pmatrix} \neq 0,
	\end{align}
	and the representation \eqref{representation_moebius} follows
        from Lemma \ref{lem:conformal_invariance}.
	
	Let $ S \in \operatorname{SO}(3)$ be such that
        $\lim_{|x| \to \infty} S \phi(x) = - e_3$.  Because $ S \phi$
        also satisfies $\mathcal{N}( S \phi) =1$ and
        $\int_{\R^2} |\nabla S \phi|^2 \intd x = 8\pi$, there exist
        $\tilde a, \tilde b, \tilde c, \tilde d \in \mathbb{C}$ with
        $\tilde a \tilde d - \tilde b \tilde c \neq 0$ and
	\begin{align}
          S \phi(x) = \Phi \left(\frac{\tilde a x + \tilde b}{\tilde c
          x + \tilde d} \right) 
	\end{align}
	for all $x\in \mathbb{C}$.  From
        $\lim_{|x| \to \infty} S \phi(x) = - e_3$ it follows that
        $\lim_{|x|\to \infty} \left| \frac{\tilde a x + \tilde
            b}{\tilde c x + \tilde d} \right| = \infty$, and thus
        $\tilde c = 0$.  Therefore,
        $\tilde a \tilde d - \tilde b \tilde c \neq 0$ implies
        $\tilde a \neq 0$ and $\tilde d \neq 0$. Without loss of
        generality we may assume $\tilde d=1$, so that for all
        $x\in \mathbb{C}$ we have
	\begin{align}
		 S \phi(x) = \Phi \left(\tilde a x + \tilde b \right).
	\end{align}
	
	With $\rho := |\tilde a|^{-1}$ and
        $x_0 := - \frac{\tilde b}{\tilde a}$ we get for all
        $x\in \mathbb{C}$ that
	\begin{align}
          S \phi(\rho x + x_0) = \Phi \left(\frac{\tilde a}{|\tilde
          a|} x \right). 
	\end{align}
	Since we evidently have
        $\left| \frac{\tilde a}{|\tilde a|}\right| =1$, there exists
        $\theta \in [-\pi,\pi)$ such that
	\begin{align}
		\frac{\tilde a}{|\tilde a|} = (\cos \theta ,\sin \theta).
	\end{align}
	The symmetry properties of $\Phi$, see definition
        \eqref{belavin-polyakov}, immediately imply
	\begin{align}
          S \phi(\rho x + x_0)= \Phi \left(\frac{\tilde a}{|\tilde a|}
          x \right) = S_\theta \Phi(x), 
	\end{align}
	where $S_\theta$ was defined in equation \eqref{def_rotation}.
	Consequently, for all $x\in \mathbb{C}$ we obtain
	\begin{align}
		\phi(x) =  S^{-1} S_\theta \Phi\left(\rho^{-1}(x- x_0) \right),
	\end{align}
	concluding the proof.
\end{proof}

\begin{proof}[Proof of Lemma \ref{lem:relation_excess_hamiltonian}]
  We follow the arguments in the proof of \cite[Lemma
  9]{li2018stability}.  A straightforward algebraic computation gives
	\begin{align}
          \int_{\R^2} \left(|\nabla m|^2 - |\nabla \phi |^2 \right)
          \intd x = \int_{\R^2} \left( | \nabla (m-\phi)|^2 + 2 \nabla
          \phi : \nabla (m-\phi)\right) \intd x. 
	\end{align}
	By inspecting the definition \eqref{belavin-polyakov} of
        $\Phi$ we see that $|\nabla \Phi (x) | = O(|x|^{-2})$ as
        $x\to \infty$.  Consequently, we may integrate by parts in the
        second term on the right-hand side and use the fact that
        $\phi$ solves the harmonic map equation
        $\Delta \phi + |\nabla \phi |^2 \phi = 0$, see equation
        \eqref{harmonic_map_equation}, to obtain
	\begin{align}
          \int_{\R^2} \left(|\nabla m|^2 - |\nabla \phi |^2 \right)
          \intd x = \int_{\R^2} \left( | \nabla (m-\phi)|^2 + 2  \phi
          \cdot  (m-\phi) |\nabla\phi |^2\right) \intd x. 
	\end{align}
	The fact $| m-\phi |^2 =- 2 \phi \cdot( m- \phi )$ gives the claim.
\end{proof}

\subsection{Integrals involving Belavin-Polyakov
  profiles}\label{subsec:bessel}

Here we collect the results of a number of computations involving the
original and truncated Belavin-Polyakov profiles $\Phi$ and $\Phi_L$,
respectively. As they involve dealing with modified Bessel functions
of the second kind, specifically $K_0$ and $K_1$, we begin by
collecting some of the well-known properties of these functions (for
definitions, etc., see \cite[Section 9.6]{abramowitz-stegun}).

Recall that $K_0(r)$ and $K_1(r)$ are positive, monotonically
decreasing functions of $r > 0$. They have the following asymptotic
expansions as $r \to 0$:
  \begin{align}
    K_0 (r) & = | \log r |  + \log 2 - \gamma +O( r^2
              |\log r|),\label{bessel0expansion}\\
    K_1(r) & = \frac{1}{r} + O(r |\log r|) ,\label{bessel1lower}
  \end{align}
  where $\gamma \approx 0.5772$ is the Euler-Mascheroni constant,
  while as $r \to \infty$ we have
  \begin{align}
    K_{0,1}(r) & =\sqrt{ \frac{\pi}{2r}}e^{-r} \left( 1+
                 O(r^{-1}) \right).\label{besselinfinity}
  \end{align}
  Finally, we will need the following basic upper bound:
  \begin{align}
    K_1(r) < \frac{r_0 K_1(r_0)}{r} \qquad \forall r >
    r_0 > 0,\label{bessel1upper} 
  \end{align}
  which easily follows from the strong maximum principle for the
  differential equation
  \begin{align}
    r^2 K_1'' + r K_1' - (r^2 +1) K_1 = 0   
  \end{align}
  satisfied by $K_1$, the asymptotics in \eqref{besselinfinity}, and
  the facts that $g(r) := \frac{r_0 K_1(r_0)}{r}$ is a strict
  supersolution for the above equation.

  We next express the Fourier transforms of several quantities
  involving the Belavin-Polyakov profile $\Phi$ and use them to
  compute its nonlocal energies. We also compute the contribution of
  $\Phi$ to the DMI energy.

\begin{lemma}\label{lemma:fourier_transform}
	We have
	\begin{align}
          \label{FDPhip}
          \mathcal{F}(\nabla \Phi')(k) & = -4\pi K_1(|k|)
                                         \frac{k}{|k|}\otimes k,\\ 
          \mathcal{F}(\Phi_3 + 1)(k) & = 4\pi K_0(|k|).
	\end{align}
        Furthermore, it holds that
	\begin{align}
          F_{\mathrm{vol}}(\Phi') & = \frac{3}{8}\pi^3,\\
          F_{\mathrm{surf}}(\Phi_3 +1) & = \frac{1}{8}\pi^3, \\
          \int_{\R^2} 2 \Phi' \cdot \nabla \Phi_{3} \intd x
          &  =  8\pi. \label{stereo_dmi_exact}
	\end{align}
      \end{lemma}

      Having obtained the above formulas, we are in a position to
      derive the formulas that are useful in obtaining an upper bound
      for the energy of a truncated Belavin-Polyakov profile.

\begin{lemma}\label{lem:bessel_calculation}
  There exist universal constants $C>0$ and $L_0 >0$ such that
  for all $L\geq L_0$ the truncation $\Phi_L$ defined in
  \eqref{def_profiles} satisfies $\Phi_L \in \mathcal{A}$ and the
  estimates
	\begin{align}
          \int_{\R^2} |\nabla \Phi_L|^2 \intd x - 8\pi
          & \leq  \frac{4\pi}{L^2} +
            \frac{C \log^2 L}{L^3},\label{stereo_excess}\\ 
          \int_{\R^2} |\Phi_L'|^2 \intd x
          & \leq 4\pi \log \left( \frac{4 L^2}{e^{2(1+\gamma)}}
            \right) + \frac{C\log^2 L}{L},\label{stereo_anisotropy}\\           
          \int_{\R^2} 2 \Phi'_L \cdot \nabla \Phi_{L,3} \intd x
          &  =  8\pi +O\left(L^{-\frac{1}{2}}\right), \label{stereo_dmi}\\
          F_{\mathrm{vol}}(\Phi'_{L})
          & = \frac{3}{8} \pi^3 + O\left(L^{-\frac{1}{4}}
            \right), \label{stereo_vol}\\ 
          F_{\mathrm{surf}}(\Phi_{3,L})
          & = \frac{1}{8}\pi^3 + O\left(
            L^{-\frac{1}{2}}\right), \label{stereo_surf}\\ 
          \int_{\R^2} \left|\nabla(\Phi_L - \Phi) \right|^2\intd x
          & \leq C
            L^{-2},\label{stereo_homogeneous_h1}\\ 
          \int_{\R^2} \left| \Phi_{3,L} -\Phi_3 \right|^2\intd x
          & \leq C L^{-1}. \label{stereo_out_of_plane}
	\end{align}
\end{lemma}

Lastly, a direct computation allows to establish an estimate for an
integral appearing in the lower bound of the anisotropy energy in
Section \ref{section:conformal}. Here and everywhere below the
integrals and the series expansions have been carried out using {\sc
  Mathematica 11.2.0.0} software. We have also verified these
computations explicitly by hand, but the details are too tedious to be
presented here.

\begin{lemma}\label{lem:calculation_for_lower_bound}
  As $\mu \to \infty$, we have
	\begin{align}
          \int_0^{\infty} \frac{\mu r^3}{1+ \mu r^2} K_1^2(r) \intd r
          \geq \frac{1}{2}\log\left( \frac{4\mu}{e^{2\gamma
          +1}}\right) +O\left( \frac{\log^2 \mu}{\mu} \right). 
	\end{align}
\end{lemma}

\begin{proof}[Proof of Lemma \ref{lemma:fourier_transform}]
  For a radial function $U(x) := u(|x|)$ with
  $u \in L^1(\R^+, r dr)$ it is well known that the Fourier
  transform of $U$ reduces to the Hankel transform
	\begin{align}
          \mathcal{F}(U)(k) = 2\pi \int_0^\infty u(r) J_0(|k| r) r \intd r,
	\end{align}
	see for example \cite[p.\ 336]{bracewell2000fourier}, where
        $J_0$ is the zeroth order Bessel function of the first kind.
        Due to $\Phi_3(x) + 1 = \frac{2}{|x|^2+1}$,
        $\nabla \cdot \Phi' (x)= - \frac{4}{(1+|x|^2)^2}$ and
        \cite[Table 13.2]{bracewell2000fourier}, this allows us to
        compute
	\begin{align}
          \mathcal{F}(\Phi_3 + 1)(k) & =  4\pi K_0(|k|),\label{fourierphi3}\\
          \mathcal{F}(\nabla \cdot \Phi')(k) & = - 4\pi |k|
                                               K_1(|k|). \label{fourierphiprime}
	\end{align}
	  As $\Phi'(x) = - \nabla \log(1+|x|^2)$,
        there exists a tempered distribution $H$ such that
        $\mathcal{F}(\Phi') = \mathrm{i} k H$, and from equation
        \eqref{fourierphiprime} we get
	\begin{align}
		|k|^2 H = 4\pi |k| K_1(|k|).
	\end{align}
	Therefore, we have
	\begin{align}
          \mathcal{F}(\nabla \Phi')(k) = - k \otimes k\,  H =- 4\pi
          K_1(|k|) \frac{k}{|k|}\otimes k. 
	\end{align}
	
	Inserting the expressions \eqref{fourierphiprime} and
        \eqref{fourierphi3} into the representations
        \eqref{vol_extension_single} and
        \eqref{surf_fourierrepresentation}, respectively, we obtain
	\begin{align}
          F_{\mathrm{surf}} (\Phi_3+1)
          & = 4 \pi
            \int_0^\infty s^2
            K_0^2(s) \intd
            s = \frac{1}{8}\pi^3,\\ 
          F_{\mathrm{vol}} (\Phi')
          & = 4 \pi \int_0^\infty
            s^2 K_1^2(s) \intd
            s = \frac{3}{8}\pi^3.
	\end{align}
        Lastly, \eqref{stereo_dmi_exact} follows by direct computation. 
\end{proof}

\begin{proof}[Proof of Lemma \ref{lem:bessel_calculation}]
  \textit{Step 1: Proof of estimate \eqref{stereo_excess}.}\\
  For $L>1$, we first note that $f_L$ is piecewise smooth, so in view
  of \eqref{besselinfinity} we have
  $\Phi_L + e_3 \in H^1(\R^2;\Sph^2)$.  A direct computation also
  shows that $\mathcal N(\Phi_L) = 1$, as it should. Therefore,
  admissibility of $\Phi_L$ for large enough $L$ would follow, as soon
  as we establish \eqref{stereo_excess}.

  In the following, all estimates and expansions are valid for
  $L\geq L_0$ with $L_0>0$ sufficiently large.  We begin by observing
  that
  \begin{align}\label{formula_gradient}
    |\nabla \Phi_L(x)|^2 = \frac{(f_L')^2(|x|)}{1-f_L^2(|x|)} + \frac{f_L^2(|x|)}{|x|^2}
  \end{align}
  and thus an explicit calculation gives
  \begin{align}\label{besselinsideball}
    \int_{\ball{0}{ \sqrt{L}}} | \nabla \Phi_L|^2\intd x = \frac{8\pi  L }{1+  L}.
  \end{align}
	
  With the help of \eqref{bessel1upper} we then obtain for all
  $r > L^\frac{1}{2}$:
  \begin{align}\label{fLto0}
    f_L(r) < \frac{L^\frac{1}{2}}{r}f\left(L^\frac{1}{2}\right) \leq \frac{2}{r}.
  \end{align}
  Consequently, we have for $r>L^\frac{1}{2}$ that
  \begin{align}\label{justification_reduction_to_planar}
    \begin{split}
      \frac{1}{1-f_L^{2}(r)} & = 1+ f_L^2(r) +
      \frac{f_L^4(r)}{1-f_L^2(r)} \leq 1+ \frac{4}{r^2} +
      \frac{C}{r^4} \leq 1+ \left(1+ \frac{C}{L}\right)\frac{4}{r^2}.
    \end{split}
  \end{align}        
  We can insert this estimate into the identity
  \eqref{formula_gradient} and compute for all
  $x\in \stcomp{B}_{\sqrt{L}}(0)$:
  \begin{align}
    |\nabla \Phi_L(x)|^2 \leq \frac{4 L^2
    K_1 ^2\left(\frac{|x|}{L}\right)+\left(|x|^2+4 (1 + C L^{-1})
    \right)  \left(K_0\left(\frac{|x|}{L}\right)+K_2
    \left(\frac{|x|}{L}\right)\right)^2}{L  
    (L+1)^2 K_1^2\left(L^{-\frac12}\right) |x|^2}, 
  \end{align}
  where $K_2$ is the modified Bessel function of the second kind.
  Integrating in radial coordinates and then expanding in the powers
  of $L^{-1}$ yields
	 \begin{align}
           \int_{ \stcomp{B}_{\sqrt{ L}}(0)} |\nabla \Phi_L|^2 \intd x
           & \leq \frac{8 \pi}{L} - \frac{4 \pi}{L^2} + O\left(
             \frac{\log^2 L}{  L^3} \right), 
	\end{align}
	which together with equation \eqref{besselinsideball} finally
        implies \eqref{stereo_excess}. In particular,
        $\Phi_L \in \mathcal A$ for all $L \geq L_0$ with some
        $L_0 > 0$ sufficiently large.
	
        \textit{Step 2: Estimate the rates of convergence of $\Phi_L$
          to $\Phi$ in several norms.}\\
	We start with an $L^2$-estimate for the out-of-plane
        components.  For $r \geq L^{\frac{1}{2}}$, by the estimate
        \eqref{fLto0} we have
	\begin{align}\label{outofplanedensity}
	  \begin{split}
            \left( \sqrt{1-f^2(r)} - \sqrt{1-f_L^2(r)} \right)^2 & =
            \frac{\left( f_L^2(r) - f^2(r) \right)^2}{\left(
                \sqrt{1 - f^2(r)} + \sqrt{1 - f_L^2(r)} \right)^2} \\
            & \leq \frac{C}{r^2}\left(f_L(r) - f(r) \right)^2.
	  \end{split}
	\end{align}
        Thus the right-hand side decays as $r^{-4}$ for
        $r \to \infty$, and we have
	\begin{align}\label{rate_for_out_of_plane}
          \int_{\R^2} \left| \Phi_3 - \Phi_{L,3} \right|^2 \intd
          x \leq C L^{-1}.  
	\end{align}
        Similarly, together with \eqref{fLto0},
        $f(r) \leq \frac{2}{r}$ for $r>0$ and the fact that
        $|\nabla \Phi(x)|^2 \leq C |x|^{-4}$ for $x\in \R^2$ we
        obtain
	\begin{align}
          \int_{\R^2} | \Phi_L -\Phi|^2 |\nabla \Phi|^2 \intd x
          & \leq C L^{-2}. 
	\end{align}
	Combining this with Lemma
        \ref{lem:relation_excess_hamiltonian} and the estimate
        \eqref{stereo_excess} we get the bound
        \eqref{stereo_homogeneous_h1}, meaning
	\begin{align}\label{rate_for_gradients}
          \int_{\R^2} |\nabla (\Phi_L -\Phi) |^2 \intd x   \leq C L^{-2}.
	\end{align}
        To handle the volume charges, we need $L^p$-estimates for
        $p\neq 2$ in view of the fact that
        $\Phi' \not\in L^2(\R^2;\R^2)$.  To this end, we can use
        estimates \eqref{fLto0} and $f(r)\leq \frac{2}{r}$ for $r>0$
        to obtain
	\begin{align}\label{l4bound}
		\int_{\R^2} |\Phi'_L - \Phi' |^4 \intd x \leq C L^{-1}.
	\end{align}
	
	Additionally, we will need a matching $L^\frac{4}{3}$-estimate
        for $\nabla \Phi'_L$ in order to apply Lemma
        \ref{lemma:fourier_basic} later.  For
        $x\in \stcomp{B}_{\sqrt{L}}(0)$ we use
        \eqref{formula_gradient} and \eqref{fLto0} to calculate
	\begin{align}\label{phiprimelestimated}
		| \nabla  \Phi'_L(x)| \leq  C \left( \left|
          f'_L(|x|)\right| + \frac{ f_L(|x|)}{|x|} \right). 
	\end{align}
	By the identity $K_1'(r) = - K_0(r) - \frac{K_1(r)}{r}$ for
        all $r>0$, the estimate \eqref{fLto0} and the expansions
        \eqref{bessel0expansion} and \eqref{bessel1lower} we have for
        all $x \in \stcomp{B}_{\sqrt{L}}(0)$:
	\begin{align}
          |\nabla \Phi'_L (x)|  \leq C \left( \frac{1}{ L^2}  
          + \frac{1}{|x|^2}\right) e^{-\frac{|x|}{L}}.
	\end{align}
       Integrating this bound we arrive at
	\begin{align}
          \int_{\stcomp{B}_{\sqrt{L}}(0)} \left|
          \nabla  \Phi'_L(x)\right|^\frac{4}{3} \intd x \leq C L^{-\frac{2}{3}}. 
	\end{align}
	The remaining integral over $B_{\sqrt{L}}(0)$ is bounded since
        $| \nabla \Phi'| \in L^\frac{4}{3}(\R^2)$ and we get
	 \begin{align}\label{l43bound}
           \int_{\R^2} |\nabla \Phi'_L|^\frac{4}{3} \intd x \leq C.
	 \end{align}
		
         \textit{Step 3: Estimate the anisotropy, DMI and stray field
           contributions.}\\
         First compute the contribution to the anisotropy energy from
         the core region:
\begin{align}
  \int_{B_{\sqrt{L}}(0)} |\Phi_L'|^2 \intd x = 4 \pi
  \left(\frac{1}{L+1}+\log (L+1)-1\right). 
\end{align}
Next, evaluate the contribution of the tail region:
\begin{align}
  \int_{\stcomp{B}_{\sqrt{L}}(0)} |\Phi_L'|^2 \intd x = \frac{4 \pi
  L^2 \left(K_0 ^2\left(L^{-\frac12}\right)+2 L^{\frac12} 
  K_1\left(L^{-\frac12}\right)
  K_0\left(L^{-\frac12}\right)-K_1 ^2\left(L^{-\frac12}\right)\right)}{(L+ 
  1)^2 K_1 ^2\left(L^{-\frac12}\right)}. 
\end{align}
Combining these two expressions and expanding in $L^{-1}$ yields
\eqref{stereo_anisotropy}.

As $\Phi'_L$ decays exponentially at infinity by virtue of estimate
\eqref{besselinfinity} and $\Phi_3+ 1$ decays as $r^{-2}$, we can
integrate by parts in the difference of the DMI terms
	\begin{align}
	  \begin{split}
            \int_{\R^2} \left(\Phi'_L \cdot \nabla \Phi_{L,3} - \Phi'
              \cdot \nabla \Phi_3 \right) \intd x & = \int_{\R^2}
            \left(\left(\Phi_3 + 1 \right) \nabla \cdot  \Phi'   -
              \left(\Phi_{L,3} + 1 \right)\nabla \cdot \Phi'_L \right)
            \intd x\\ 
            & = \int_{\R^2} \left(\left(\Phi_3 + 1 \right) \nabla
              \cdot \left(\Phi' - \Phi'_L\right) - \left(\Phi_{L,3} -
                \Phi_3 \right)\nabla \cdot \Phi'_L \right) \intd x.
	  \end{split}
	\end{align}
	By the facts that $\Phi_3+1 \in L^2(\R^2)$, the estimates
        \eqref{rate_for_gradients} and \eqref{rate_for_out_of_plane},
        and the Cauchy-Schwarz inequality we deduce
	\begin{align}
          \left| \int_{\R^2} \left(\Phi'_L \cdot \nabla \Phi_{L,3} -
          \Phi' \cdot \nabla \Phi_3 \right) \intd x  \right| \leq C
          L^{-\frac{1}{2}}. 
	\end{align}
	Together with \eqref{stereo_dmi_exact}, this then yields
        \eqref{stereo_dmi}.
	
	Similarly, by Lemma \ref{lemma:fourier_transform} we only need
        to estimate the error terms in the stray field contributions
        to prove estimates \eqref{stereo_vol} and \eqref{stereo_surf}.
        By bilinearity and the estimates, \eqref{vol_interpolation}
        with $p=4$, \eqref{l4bound}, \eqref{l43bound} and
        $|\nabla \Phi'| \in L^\frac{4}{3}(\R^2;\R^2)$ we get
	\begin{align}\label{calculatevolumech}
	  \begin{split}
            \left| F_{\mathrm{vol}}(\Phi'_L) - F_{\mathrm{vol}}(\Phi')
            \right|
            &  \leq \left| F_{\mathrm{vol}}(\Phi'_L+ \Phi', \Phi'_L-
              \Phi'  )\right| \\ 
            & \leq C \| \Phi'_L - \Phi' \|_4 \|\nabla (\Phi'_L+ \Phi')
            \|_\frac{4}{3}\\ 
            &\leq C L^{-\frac{1}{4}}.
	  \end{split}
	\end{align}
	A similar calculation for the surface term gives
	\begin{align}
	  \begin{split}
            \left| F_{\mathrm{surf}}(\Phi_{L,3}) -
              F_{\mathrm{surf}}(\Phi_3) \right| \leq C \| \Phi_{L,3} -
            \Phi_{3}\|_2 \| \nabla ( \Phi_{L,3} + \Phi_{3} ) \|_2
	  \end{split}
	\end{align}
	We can now apply the interpolation inequality
        \eqref{surf_interpolation} together with the estimates
        \eqref{stereo_excess} and \eqref{rate_for_out_of_plane} in
        order to obtain
	\begin{align}\label{calculatesurfacech}
          \left| F_{\mathrm{surf}}(\Phi_{L,3}) -
          F_{\mathrm{surf}}(\Phi_3) \right|  \leq C L^{-\frac{1}{2}}, 
	\end{align}
	concluding the proof.
\end{proof}

% \bibliographystyle{plain}
% \bibliography{bibliography}

\end{document}